\RequirePackage{fix-cm}
\documentclass[smallextended,referee,envcountsect]{svjour3}       % onecolumn (second format)
\smartqed  % flush right qed marks, e.g. at end of proof

\usepackage{graphicx}
\usepackage{amsmath, ntheorem}
\usepackage{amssymb}
\usepackage{stackrel}
\usepackage{marvosym, enumerate}
\usepackage{amstext, amscd, latexsym}
\usepackage{amsfonts, ragged2e}
\usepackage{mathrsfs, pgfplots, enumerate, setspace}
\usepackage{subfigure}
\usepackage{cases}
\smartqed
\usepackage{cite}
\usepackage{amstext, graphicx, amscd, latexsym, amssymb}
\usepackage{amsfonts, ragged2e}
\usepackage{pgfplots, setspace}
\usepackage{geometry}
\usepackage{latexsym,bm}
\usepackage{float}
\usepackage{epstopdf,caption}
 
\usepackage{amsfonts,amsmath,amsbsy,amssymb,amscd,mathrsfs}
\usepackage{graphicx}
\usepackage{epstopdf}
\usepackage{algorithmic} 
\usepackage{graphicx,epstopdf}

 \newcommand{\nm}[1]{\left\lVert {#1} \right\rVert}
 
 \newcommand{\dual}[1]{\left\langle {#1} \right\rangle}

\usepackage[figuresright]{rotating}
\usepackage[sort&compress,numbers]{natbib}
\usepackage[ruled,linesnumbered]{algorithm2e} 
\usepackage[colorlinks,linkcolor=blue,anchorcolor=blue,citecolor=blue,urlcolor=blue]{hyperref}

\journalname{}

\begin{document}

\title{Barzilai-Borwein Proximal Gradient Methods for Multiobjective Composite Optimization Problems with Improved Linear Convergence}

%\titlerunning{Convergence Rates Analysis of Interior Bregman Gradient Method for Vector Optimization Problems}        % if too long for running head

\author{Jian Chen \and Liping Tang \and  Xinmin Yang  }

\institute{J. Chen \at National  Center  for  Applied  Mathematics in Chongqing, Chongqing Normal University, Chongqing 401331, China, and School of Mathematical Sciences, University of Electronic Science and Technology of China, Chengdu, Sichuan 611731, China\\
                    chenjian\_math@163.com\\
                   L.P. Tang \at National Center for Applied Mathematics in Chongqing, and School of Mathematical Sciences,  Chongqing Normal University, Chongqing 401331, China\\
                   tanglipings@163.com\\
        \Letter X.M. Yang \at National Center for Applied Mathematics in Chongqing, and School of Mathematical Sciences,  Chongqing Normal University, Chongqing 401331, China\\
        xmyang@cqnu.edu.cn  \\}

\date{Received: date / Accepted: date}

\maketitle

\begin{abstract}
When minimizing a multiobjective optimization problem (MOP) using multiobjective gradient descent methods, the imbalances among objective functions often decelerate the convergence.  In response to this challenge, we propose two types of the Barzilai-Borwein proximal gradient method for multi-objective composite optimization problems (BBPGMO). We establish convergence rates for BBPGMO, demonstrating that it achieves rates of $O(\frac{1}{\sqrt{k}})$, $O(\frac{1}{k})$, and $O(r^{k})(0<r<1)$ for non-convex, convex, and strongly convex problems, respectively. Furthermore, we show that BBPGMO exhibits linear convergence for MOPs with several linear objective functions. Interestingly, the linear convergence rate of BBPGMO surpasses the existing convergence rates of first-order methods for MOPs, which indicates its enhanced performance and its ability to effectively address imbalances from theoretical perspective. Finally, we provide numerical examples to illustrate the efficiency of the proposed method and verify the theoretical results.

\end{abstract}

\section{Introduction}
In the realm of multiobjective optimization, the primary goal is to simultaneously  optimize multiple objective functions. Generally, finding a single solution that achieves the optima for all objectives at once is not feasible. As a result, the concept of optimality is defined by either {\it Pareto optimality} or {\it efficiency}. A solution is deemed Pareto optimal or efficient if no objective can be improved without sacrificing the others. As society and the economy progress, the applications of this type of problem have proliferated across a multitude of domains, such as engineering \cite{MA2004}, economics \cite{FW2014,TC2007}, management science \cite{E1984}, and machine learning \cite{SK2018,YL2021}, etc. 
\par  Solution strategies play a pivotal role in the realm of applications involving multiobjective optimization problems (MOPs).  Over the past two decades, multiobjective gradient descent methods have gained escalating attention within the multiobjective optimization community. These methods generate descent directions by solving subproblems, eliminating the need for predefined parameters. Subsequently, line search techniques are employed along the descent direction to ensure sufficient improvement for all objectives. Attouch et al. \cite{AGG2015} pointed out an attractive property of this method in fields like game theory, economics, social science, and management: {\it it improves each of the objective functions}. As far as we know, the study of multiobjective gradient descent methods can be traced back to the pioneering work by Mukai \cite{M1980}. Later, Fliege and Svaiter \cite{FS2000} independently reinvented the steepest descent method for MOPs (SDMO). Their work elucidated that the multiobjective steepest descent direction reduces to the steepest descent direction when dealing with a single objective. This observation inspired researchers to extend ordinary numerical algorithms for solving MOPs (see, e.g., \cite{AP2021,BI2005,CL2016,FD2009,FV2016,GI2004,LP2018,MP2018,MP2019,P2014,QG2011} and references therein). 
\par Although multiobjective gradient descent methods are derived from their single-objective counterparts, there exist certain theoretical gaps between the two types approaches. Recently, Zeng et al. \cite{ZDH2019} and Fliege et al. \cite{FVV2019} have studied the convergence rates of SDMO. They proved that SDMO converges at rates of  $O(\frac{1}{\sqrt{k}})$, $O(\frac{1}{k})$, and $O(r^{k})$ $(0<r<1)$ for nonconvex, convex, and strongly convex problems, respectively. Tanabe et al. \cite{TFY2023} obtained similar results for the proximal gradient method for MOPs (PGMO) \cite{TFY2019}. It is worth noting that when minimizing a $\mu$-strongly convex and $L$-smooth function using vanilla gradient method, the rate of convergence in terms of $\{\|x^{k}-x^{*}\|\}$ is $\sqrt{1-\frac{\mu}{L}}$. However, for MOPs, the linear convergence rate of SDMO is $\sqrt{1-\frac{\mu_{\min}}{L_{\max}}}$, where $\mu_{\min}:=\min\{\mu_{i}:i=1,2,...,m\}$ and $L_{\max}:=\max\{L_{i}:i=1,2,...,m\}$. Consequently, imbalances among objective functions, arising from the substantially distinct curvature matrices of different objective functions, can lead to a small value of $\frac{\mu_{\min}}{L_{\max}}$. Particularly, even though each of the objective functions are not ill-conditioned (a relative small $\frac{L_{i}}{\mu_{i}}$), the {\it overall condition number} $\frac{L_{\max}}{\mu_{\min}}$ can be tremendous. This observation explains why each objective is relatively easy to optimize individually but challenging when attempting to optimize them simultaneously. To the best of our knowledge, most of first-order methods for MOPs suffer slow convergence due to the imbalances among objectives. Naturally, questions arise: How to accelerate multiobjective first-order methods and bridge the theoretical gap between first-order methods for SOPs and MOPs?
\par To accelerate multiobjective first-order methods, Lucambio P\'{e}rez and Prudente \cite{LP2018} utilized previous information and propose nonlinear conjugate gradient methods for MOPs. EI Moudden and EI Mouatasim \cite{EE2021} approximated the Hessian using diagonal matrices and introduced diagonal steepest descent methods for MOPs. Very recently, motivated by Nesterov's accelerated method \cite{N1983}, Tanabe et al. proposed an accelerated proximal gradient method for MOPs. Sonntag and Peitz \cite{SP2022,SP2023} investigated accelerated multiobjective gradient methods from continuous-time perspective. Although these methods achieved improved performance, the theoretical gap between first-order methods for SOPs and MOPs remains open. On the other hand, some studies \cite{GK2021,MF2019,MP2016} have pointed out that Armijo line search often generates a relative small stepsize in SDMO, which slows down convergence. Chen et al. \cite{CTY2023} elucidated that the small stepsize is mainly due to the imbalances among objectives. Regarding the multiobjective steepest descent direction, it holds that
$$\dual{\nabla f_{i}(x^{k}),d^{k}}=-\nm{d^{k}}^{2},~\forall \lambda_{i}^{k}>0,$$
where $\lambda^{k}$ is the dual variable of direction-finding subproblem. The relation implies that each iteration produces a similar amount of descent for different objectives, which results in small stepsize due to the imbalances among objectives. To address this issue, Chen et al. \cite{CTY2023} proposed the Barzilai-Borwein descent method for MOPs (BBDMO), which dynamically tunes gradient magnitudes using Barzilai-Borwein's rule \cite{BB1988} in the direction-finding subproblem. Along the Barzilai-Borwein descent direction, different objectives have distinct amount of descent. Specifically, 
$$\dual{\nabla f_{i}(x^{k}),d^{k}_{BB}}=-\alpha^{k}_{i}\nm{d^{k}_{BB}}^{2},~\forall \lambda_{i}^{k}>0,$$ where $\alpha^{k}_{i}$ is given by Barzilai-Borwein method. Theoretical results indicate that BBDMO can achieve a better stepsize, and numerical results demonstrate that it requires fewer iterations and function evaluations. Despite the excellent performance in practice, the theoretical guarantee of faster convergence of BBDMO remains unknown.
\par In this paper, we turn our attention towards the generic model of unconstrained multiobjective composite optimization problems, which is formulated as follows:
\begin{align*}
	\min\limits_{x\in\mathbb{R}^{n}} F(x), \tag{MCOP}\label{MCOP}
\end{align*}
where $F:\mathbb{R}^{n}\rightarrow\mathbb{R}^{m}$ is a vector-valued function. Each component $F_{i}$, $i=1,2,...,m$, is defined by
$$F_{i}:=f_{i}+g_{i},$$ 
where $f_{i}$ is continuously differentiable and $g_{i}$ is proper convex and lower semicontinuous but not necessarily differentiable. This type of problem finds wide applications in machine learning and statistics, and gradient descent methods tailored for it have received increasing attention (see, e.g., \cite{A2023,AFP2023,TFY2019,TFY2022}). To this end, we propose two types of Barzilai-Borwein proximal gradient methods for MCOPs (BBPGMO). We analyze the convergence rates of BBPGMO and provide new theoretical
results, paving the way for explaining its fast convergence behavior in practice.
%The motivation for studying this type of problem can be summarized as follows: $1)$ To apply the Barzilai-Borwein method to a wider range of problems and verify its effectiveness. $2)$ This type of problem coincides with an unconstrained MOP or a constrained MOP when $g_{i}(x)=0$, $i\in[m]$ or $g_{i}(x)=\mathbb{I}_{\mathcal{X}}(x)$, $i\in[m]$, respectively. Consequently, the convergence rates of the Barzilai-Borwein proximal gradient method for MOPs (BBPGMO) align with those of BBDMO and the Barzilai-Borwein projected gradient method for MOPs.  
The main contributions of this paper can be summarized in the following points:
\par (i) To mitigate the imbalances among objective functions, we propose two types of Barzilai-Borwein proximal gradient methods for MCOPs. The first method employs Armijo line search, the Barzilai-Borwein's rule is applied to every objective in direction-finding subproblem. It coincides with BBDMO when $g_{i}(x)=0$, $i\in[m]$. Additionally, we devise a new proximal gradient method for MCOPs without line search, where the smooth parameter $L_{i},~i\in[m],$ is employed to tune the corresponding objective in the direction-finding subproblem. It is worth noting that the global smoothness parameters for a general MOP are unknown and tend to be conservative, we thus propose an adaptive method to estimate the local smoothness parameters, in which the initial values are obtained through the Barzilai-Borwein method.
\par (ii) With line search, we prove that every accumulation point generated by BBPGMO is a Pareto critical point. Moreover, We establish strong convergence of the sequence generated by BBPGMO under standard convexity assumption. In the strongly convex case, it is proved that the produced sequence converges linearly to a Pareto solution.  We also provide the convergence rates of the adaptive Barzilai-Borwein proximal gradient method for MCOPs (ABBPGMO). Notably, in the case of strong convexity, the rate of convergence in terms of $\|x^{k}-x^{*}\|$ is $\sqrt{1-\min\limits_{i\in[m]}\left\{\frac{\mu_{i}}{L_{i}}\right\}}$. The improved linear convergence explains why the BBPGMO outperforms the PGMO from a theoretical perspective. 
\par (iii) We establish the linear convergence of BBPGMO for MOPs with some linear objectives. This finding shows that BBPGMO can achieve fast convergence even when dealing with problems with linear objectives, which are known to impose significant imbalances in multiobjective optimization \cite{CTY2023}. 
\par The paper is organized as follows. In section \ref{sec2}, we present some necessary notations and definitions that will be used later. In section \ref{sec4}, we propose two types of Barzilai-Borwein proximal gradient methods, and present some preliminary lemmas. The convergence rates of BBPGMO are analyzed in section \ref{sec5}. In section \ref{sec6}, we present an efficient approach to solve the subproblem using its dual. The numerical results are presented in section \ref{sec7}, which demonstrate that BBPGMO outperforms PGMO and verify the theoretical results. Finally, we draw some conclusions at the end of the paper.

\section{Preliminaries}\label{sec2}
Throughout this paper, the $n$-dimensional Euclidean space $\mathbb{R}^{n}$ is equipped with the inner product $\langle\cdot,\cdot\rangle$ and the induced norm $\|\cdot\|$. We denote by $Jf(x)\in\mathbb{R}^{m\times n}$ the Jacobian matrix of $f$ at $x$, by $\nabla f_{i}(x)\in\mathbb{R}^{n}$ the gradient of $f_{i}$ at $x$. Moreover, we denote
$$F^{\prime}_{i}(x;d):=\lim\limits_{t\downarrow0}\frac{F_{i}(x+td)-F_{i}(x)}{t}$$
the directional derivative of $F_{i}$ at $x$ in the direction $d$. The Moreau envelope of $g$ is given by
$$\mathcal{M}_{g}(x):=\min\limits_{y\in\mathbb{R}^{n}}\left\{g(y)+\frac{1}{2}\|y-x\|^{2}\right\}.$$
The proximal operator of $g$ is denoted by $${\rm Prox}_{g}(x):=\mathop{\arg\min}\limits_{y\in\mathbb{R}^{n}}\left\{g(y)+\frac{1}{2}\|y-x\|^{2}\right\}.$$ For simplicity, we denote $[m]:=\{1,2,...,m\}$, and $$\Delta_{m}:=\left\{\lambda:\sum\limits_{i\in[m]}\lambda_{i}=1,\lambda_{i}\geq0,\ i\in[m]\right\}$$ the $m$-dimensional unit simplex. In case of misunderstand, we define the order $\preceq(\prec)$ in $\mathbb{R}^{m}$ as $$u\preceq(\prec)v~\Leftrightarrow~v-u\in\mathbb{R}^{m}_{+}(\mathbb{R}^{m}_{++}).$$
\par In the following, we introduce the concepts of optimality for (\ref{MCOP}) in the Pareto sense. 
\vspace{2mm}
\begin{definition}\label{def1}
	A vector $x^{\ast}\in\mathbb{R}^{n}$ is called Pareto solution to (\ref{MCOP}), if there exists no $x\in\mathbb{R}^{n}$ such that $F(x)\preceq F(x^{\ast})$ and $F(x)\neq F(x^{\ast})$.
\end{definition}
\vspace{2mm}
\begin{definition}\label{def2}
	A vector $x^{\ast}\in\mathbb{R}^{n}$ is called weakly Pareto solution to (\ref{MCOP}), if there exists no $x\in\mathbb{R}^{n}$ such that $F(x)\prec F(x^{\ast})$.
\end{definition}
\vspace{2mm}
\begin{definition}\label{def3}
	A vector $x^{\ast}\in\mathbb{R}^{n}$ is called  Pareto critical point of (\ref{MCOP}), if
	$$\max\limits_{i\in[m]}F_{i}^{\prime}(x^{*};d)\geq0,~\forall d\in\mathbb{R}^{n}.$$
\end{definition}

%\begin{defi}\rm\citep{FS2000}
%	A vector $d\in\mathbb{R}^{n}$ is called descent direction for $F$ at $x$, if
%	$$\max\limits_{i\in[m]}F_{i}^{\prime}(x;d)<0.$$
%\end{defi}
\par From Definitions \ref{def1} and \ref{def2}, it is evident that Pareto solutions are always weakly Pareto solutions. The following lemma shows the relationships among the three concepts of Pareto optimality.
\vspace{2mm}
\begin{lemma}[Theorem 3.1 of \cite{FD2009}] The following statements hold.
	\begin{itemize}
		\item[$\mathrm{(i)}$]  If $x\in\mathbb{R}^{n}$ is a weakly Pareto solution to (\ref{MCOP}), then $x$ is Pareto critical point.
		\item[$\mathrm{(ii)}$] Let every component $F_{i}$ of $F$ be convex. If $x\in\mathbb{R}^{n}$ is a Pareto critical point of (\ref{MCOP}), then $x$ is weakly Pareto solution.
		\item[$\mathrm{(iii)}$] Let every component $F_{i}$ of $F$ be strictly convex. If $x\in\mathbb{R}^{n}$ is a Pareto critical point of (\ref{MCOP}), then $x$ is Pareto solution.
	\end{itemize}
\end{lemma}
\vspace{2mm}
\begin{definition}
	A differentiable function $h:\mathbb{R}^{n}\rightarrow\mathbb{R}$ is $L$-smooth if $$\nm{\nabla h(y)-\nabla h(x)}\leq\nm{y-x}$$ holds for
	all $x,y\in\mathbb{R}^{n}$. And f is $\mu$-strongly convex if $$\dual{\nabla h(y)-\nabla h(x),y-x}\geq\mu\nm{y-x}^{2}$$ holds for all $x,y\in\mathbb{R}^{n}$.
\end{definition}
\vspace{2mm}
\par $L$-smoothness of $h$ implies the following quadratic upper bound:
$$h(y)\leq h(x) + \dual{\nabla h(x),y-x}+\frac{L}{2}\|y-x\|^{2},~\forall x,y\in\mathbb{R}^{n}.$$ 
On the other hand, $\mu$-strong convexity yields the quadratic lower bound:
$$h(y)\geq h(x) + \dual{\nabla h(x),y-x}+\frac{\mu}{2}\|y-x\|^{2},~\forall x,y\in\mathbb{R}^{n}.$$

\section{BBPGMO:~Barzilai-Borwein proximal gradient method for MCOPs}\label{sec4}
\subsection{Proximal gradient method for MCOPs}
In this subsection, we recall two types of multiobjective proximal gradient methods for (\ref{MCOP}). Defined the function $\psi_{x}:\mathbb{R}^{n}\rightarrow\mathbb{R}$ by 
$$\psi_{x}(d):=\max\limits_{i\in[m]}\left\{\left\langle\nabla f_{i}(x),d\right\rangle+g_{i}(x+d)-g_{i}(x)\right\}.$$
\par The proximal gradient method updates iterates as follows:
$$x^{k+1} = x^{k}+t_{k}d_{\ell}^{k},$$
where $d_{\ell}^{k}$ is a descent direction and $t_{k}$ is the stepsize. The descent direction $d_{\ell}^{k}$ is the unique optimal solution to the following subproblem with $x=x^{k}$:
\begin{equation}\label{sub}
	\min\limits_{d\in\mathbb{R}^{n}}\psi_{x}(d)+\frac{\ell}{2}\|d\|^{2},~\ell>0.
\end{equation}
By Sion's minimax theorem \cite{S1958}, there exists $\lambda^{k}\in\Delta_{m}$ such that
$$d_{\ell}^{k}=\mathop{\arg\min}\limits_{d\in\mathbb{R}^{n}}\left\{\sum\limits_{i\in[m]}\lambda^{k}_{i}(\left\langle\nabla f_{i}(x^{k}),d\right\rangle+g_{i}(x^{k}+d)-g_{i}(x^{k}))+\frac{\ell}{2}\|d\|^{2}\right\},$$
and 
\begin{equation}\label{Ee}
	\left\langle\nabla f_{i}(x^{k}),d_{\ell}^{k}\right\rangle+g_{i}(x^{k}+d_{\ell}^{k})-g_{i}(x^{k})=\psi_{x^{k}}(d_{\ell}^{k}),~\forall\lambda^{k}_{i}>0.
\end{equation}
To compute the stepsize $t_{k}$, let $\sigma\in(0,1)$ be a predefined constant, the condition for accepting $t_{k}$ is given by: 
\begin{equation}\label{Ed}
	F_{i}(x^{k}+t_{k}d_{\ell}^{k})-F_{i}(x^{k})\leq t_{k}\sigma\psi_{x^{k}}(d_{\ell}^{k}),~i\in[m].
\end{equation} 
Initially, set $t_{k}=1$. If (\ref{Ed}) is not satisfied, we update $t_{k}$ using the following rule:
$$t_{k}=\gamma t_{k},~\gamma\in(0,1).$$
\par The proximal gradient method for MCOPs with line search is described as follows.

\begin{algorithm}  
	\caption{{\ttfamily{proximal\_gradient\_method\_for\_MCOPs\_with\_line\_search}}~\cite{TFY2019}}\label{algw}
	\begin{algorithmic}[1]
		\REQUIRE{$x^{0}\in\mathbb{R}^{n},~\ell>0,~\sigma,\gamma\in(0,1)$}
		\FOR{$k=0,...$}
		\STATE{Compute $d_{\ell}^{k}$ by solving subproblem (\ref{sub}) with $x=x^{k}$}
		\IF{ $d_{\ell}^{k}=0$}
		\RETURN{Pareto critical point $x^{k}$ }
		\ELSE{
			\STATE{Compute the stepsize $t_{k}\in(0,1]$ as the maximum of\\ ~~~~~~$T_{k}:=\{\gamma^{j}:j\in\mathbb{N},~F_{i}(x^{k}+t_{k}d_{\ell}^{k})-F_{i}(x^{k})\leq \gamma^{j}\sigma\psi_{x^{k}}(d_{\ell}^{k}),~i\in[m]\}$}
			\STATE{Update $x^{k+1}:= x^{k}+t_{k}d_{\ell}^{k}$}}
		\ENDIF
		\ENDFOR
	\end{algorithmic}
\end{algorithm}

\vspace{2mm}
Assume that $f_{i}$ is $L_{i}$-smooth for $i\in[m]$, the stepsize $t_{k}$ can be fixed as $1$. Denote $L_{\max}:=\max\{L_{i}:i\in[m]\}$, the proximal gradient method for MCOPs without line search is described as follows.

\begin{algorithm}  
	\caption{{\ttfamily{proximal\_gradient\_method\_for\_MCOPs\_without\_line\_search}}~\cite{TFY2019}}\label{algwo} 
	\begin{algorithmic}[1]
		\REQUIRE{$x^{0}\in\mathbb{R}^{n},~\ell>\frac{L_{\max}}{2}$}
		\FOR{$k=0,...$}
		\STATE{Compute $d_{\ell}^{k}$ by solving subproblem (\ref{sub}) with $x=x^{k}$}
		\IF{$d_{\ell}^{k}=0$}
		\RETURN{Pareto critical point $x^{k}$  }
		\ELSE{
			\STATE{Update $x^{k+1}:= x^{k}+d_{\ell}^{k}$}  }
		\ENDIF
		\ENDFOR
	\end{algorithmic}
\end{algorithm}
\subsection{Barzilai-Borwein proximal gradient method for MCOPs}
As described in \cite[Section 4]{CTY2023}, the relation (\ref{Ee}) implies that each iteration yields a similar amount of descent for different objective functions ($\lambda^{k}_{i}\neq0$), which can result in slow convergence for imbalanced multiobjective optimization problems. To address this issue and achieve distinct amounts of descent for different objective functions, we introduce the {\it Barzilai-Borwein proximal gradient direction} as follows:
\begin{equation}\label{d}
	d^{k} = P_{\alpha^{k}}(x^{k}) - x^{k},
\end{equation}
where $P_{\alpha^{k}}(x^{k})$ is the minimizer of
\begin{equation}\label{dk}
	\min\limits_{x\in\mathbb{R}^{n}}\max\limits_{i\in[m]}\left\{\frac{
		\left\langle\nabla f_{i}(x^{k}),x-x^{k}\right\rangle + g_{i}(x)-g_{i}(x^{k})}{\alpha^{k}_{i}}+\frac{1}{2}\|x-x^{k}\|^{2}\right\},
\end{equation}
and $\alpha^{k}\in\mathbb{R}^{m}_{++}$ is set as follows:
\begin{equation}\label{alpha_k}
	\alpha^{k}_{i}=\left\{
	\begin{aligned}
		&\max\left\{\alpha_{\min},\min\left\{\frac{\left\langle s^{k-1},y^{k-1}_{i}\right\rangle}{\left\langle s^{k-1},s^{k-1}\right\rangle},\ \alpha_{\max}\right\}\right\}, & \left\langle s^{k-1},y^{k-1}_{i}\right\rangle&>0, \\
		&\max\left\{\alpha_{\min},\min\left\{\frac{\|y^{k-1}\|}{\|s^{k-1}\|},\ \alpha_{\max}\right\}\right\},& \left\langle s^{k-1},y^{k-1}_{i}\right\rangle&<0,\\
		& \alpha_{\min}, & \left\langle s^{k-1},y^{k-1}_{i}\right\rangle&=0,
	\end{aligned}
	\right.
\end{equation}
for all $i\in[m]$, where $\alpha_{\max}$ is a sufficient large positive constant and $\alpha_{\min}$ is a sufficient small positive constant, $s^{k-1}=x^{k}-x^{k-1},\ y^{k-1}_{i}=\nabla f_{i}({x^{k}})-\nabla f_{i}(x^{k-1}),\ i\in[m].$ 
\vspace{2mm}
\begin{proposition}\label{p3}
	Let $P_{\alpha^{k}}(x^{k})$ be defined as (\ref{dk}), then there exists $\lambda^{k}\in\Delta_{m}$ such that
	\begin{equation}\label{prox}
		P_{\alpha^{k}}(x^{k}) = {\rm Prox}_{\sum\limits_{i\in[m]}\lambda_{i}^{k}\frac{g_{i}}{\alpha^{k}_{i}}}\left(x^{k}-\sum\limits_{i\in[m]}\lambda_{i}^{k}\frac{\nabla f_{i}(x^{k})}{\alpha^{k}_{i}}\right),
	\end{equation}
	and 
	\begin{equation}\label{Ediff}
		\begin{aligned}
			&~~~~\left\langle\nabla f_{i}(x^{k}),d^{k}\right\rangle + g_{i}(x^{k}+ d^{k})-g_{i}(x^{k})\\
			&=\alpha_{i}^{k}\max\limits_{i\in[m]}\left\{\frac{
				\left\langle\nabla f_{i}(x^{k}),d^{k}\right\rangle + g_{i}(x^{k}+d^{k})-g_{i}(x^{k})}{\alpha^{k}_{i}}\right\},~\forall \lambda_{i}^{k}>0.
		\end{aligned}
	\end{equation}
\end{proposition}
\begin{proof}
	By (\ref{dk}) and Sion's minimax theorem \cite{S1958}, there exists $\lambda^{k}\in\Delta_{m}$ such that
	\begin{equation}\label{prox_lamb}\small
		\begin{aligned}
			&~~~~P_{\alpha^{k}}(x^{k})\\
			&=\mathop{\arg\min}\limits_{x\in\mathbb{R}^{n}}\left\{\frac{
				\left\langle\sum\limits_{i\in[m]}\lambda_{i}^{k}\nabla f_{i}(x^{k}),x-x^{k}\right\rangle + \sum\limits_{i\in[m]}\lambda_{i}^{k}(g_{i}(x)-g_{i}(x^{k}))}{\alpha^{k}_{i}}+\frac{1}{2}\|x-x^{k}\|^{2}\right\},
		\end{aligned}
	\end{equation}
	and 
	\begin{align*}
		&~~~~\left\langle\nabla f_{i}(x^{k}),P_{\alpha^{k}}(x^{k}) - x^{k}\right\rangle + g_{i}(P_{\alpha^{k}}(x^{k}))-g_{i}(x^{k})\\
		&=\alpha_{i}^{k}\max\limits_{i\in[m]}\left\{\frac{
			\left\langle\nabla f_{i}(x^{k}),P_{\alpha^{k}}(x^{k}) - x^{k}\right\rangle + g_{i}(P_{\alpha^{k}}(x^{k}))-g_{i}(x^{k})}{\alpha^{k}_{i}}\right\}
	\end{align*}
	for all $\lambda_{i}^{k}>0.$
	The desired result follows by the definitions of proximal operator and $d^{k}$.
\end{proof}
\vspace{2mm}
\begin{remark}
	Since $\alpha^{k}_{i}$ is objective-based, equation (\ref{Ediff}) indicates that, along with the Barzilai-Borwein proximal gradient direction, different objective functions have distinct amount of descent. 
\end{remark}
\vspace{2mm}
\par Next, we will present several properties of $d^{k}$. 
\vspace{2mm}
\begin{lemma}
	Let $d^{k}$ be defined as (\ref{d}), then we have
	\begin{equation}\label{E6}
		\left\langle\nabla f_{i}(x^{k}),d^{k}\right\rangle + g_{i}(x^{k}+ d^{k})-g_{i}(x^{k})\leq-\alpha_{i}^{k}\|d^{k}\|^{2},~\forall i\in[m].
	\end{equation}
\end{lemma}
\begin{proof}
	The assertion can be obtained by using the same arguments as in the proof of \cite[Lemma 4.1]{TFY2019}.
\end{proof}
\vspace{2mm}
\begin{lemma}\label{lemma}
	Let $d^{k}$ be defined as (\ref{d}), then the following statements hold.
	\begin{itemize}
		\item[$\mathrm{(i)}$] the following assertions are equivalent:
		\subitem$\mathrm{(a)}$ The point $x^{k}$ is non-critical;
		\subitem$\mathrm{(b)}$ $d^{k}\neq0$;
		\subitem$\mathrm{(c)}$ $d^{k}$ is a descent direction.
		\item[$\mathrm{(ii)}$] if there exists a convergent subsequence $x^{k}\stackrel{\mathcal{K}}{\longrightarrow} x^{*}$ such that $d^{k}\stackrel{\mathcal{K}}{\longrightarrow}0$, then $x^{*}$ is Pareto critical.
	\end{itemize}
\end{lemma}
\vspace{2mm}
\begin{proof}
	Since $\alpha_{\min}\leq\alpha_{i}^{k}\leq\alpha_{\max}$, assertion (i)  can be obtained by using the same arguments as in the proof of \cite[Lemma 3.2]{TFY2019}. Nexct, we prove assertion (ii). We use the definition of $d^{k}$ and the fact that $\alpha_{i}^{k}\leq\alpha_{\max}$ to get 
	\begin{equation}\label{ew13}
		\begin{aligned}
			&~~~~\min\limits_{i\in[m]}\left\{\frac{
				\left\langle\nabla f_{i}(x^{k}),-d^{k}\right\rangle + g_{i}(x^{k})-g_{i}(x^{k}+d^{k})}{\alpha_{i}^{k}}-\frac{1}{2}\|d^{k}\|^{2}\right\}\\
			&=\max\limits_{y\in\mathbb{R}^{n}}\min\limits_{i\in[m]}\left\{\frac{
				\left\langle\nabla f_{i}(x^{k}),x^{k}-y\right\rangle + g_{i}(x^{k})-g_{i}(y)}{\alpha_{i}^{k}}-\frac{1}{2}\|x^{k}-y\|^{2}\right\}\\
			&\geq\max\limits_{y\in\mathbb{R}^{n}}\min\limits_{i\in[m]}\left\{\frac{
				\left\langle\nabla f_{i}(x^{k}),x^{k}-y\right\rangle + g_{i}(x^{k})-g_{i}(y)}{\alpha_{\max}}-\frac{1}{2}\|x^{k}-y\|^{2}\right\}\\
			&=\frac{1}{\alpha_{\max}}\max\limits_{y\in\mathbb{R}^{n}}\min\limits_{i\in[m]}\left\{\left\langle\nabla f_{i}(x^{k}),x^{k}-y\right\rangle + g_{i}(x^{k})-g_{i}(y)-\frac{\alpha_{\max}}{2}\|x^{k}-y\|^{2}\right\}\\
			&\geq\frac{\ell}{(\alpha_{\max})^{2}}\max\limits_{y\in\mathbb{R}^{n}}\min\limits_{i\in[m]}\left\{\left\langle\nabla f_{i}(x^{k}),x^{k}-y\right\rangle + g_{i}(x^{k})-g_{i}(y)-\frac{\ell}{2}\|x^{k}-y\|^{2}\right\}\\
			&\geq\frac{\ell^{2}}{2(\alpha_{\max})^{2}}\|d_{\ell}^{k}\|^{2},
		\end{aligned}
	\end{equation}
	where the second inequality follows by \cite[Theorem 4.2]{TFY2020} and the fact that $\alpha_{\max}>\ell$ ($\alpha_{\max}$ is a large positive constant), and the last inequality is given by \cite[Lemma 4.1]{TFY2019}. On the other hand, from $d^{k}\stackrel{\mathcal{K}}{\longrightarrow}0$ and the continuity of $g_{i}$ for $i\in[m]$, we obtain $$\min\limits_{i\in[m]}\left\{\frac{
		\left\langle\nabla f_{i}(x^{k}),-d^{k}\right\rangle + g_{i}(x^{k})-g_{i}(x^{k}+d^{k})}{\alpha_{i}^{k}}-\frac{1}{2}\|d^{k}\|^{2}\right\}\stackrel{\mathcal{K}}{\longrightarrow}0.$$ This together with (\ref{ew13}) gives $d_{\ell}^{k}\stackrel{\mathcal{K}}{\longrightarrow}0$. Moreover, from the continuity of $d_{\ell}$ (\cite[Lemma 3.2]{TFY2019}) and the fact that $x^{k}\stackrel{\mathcal{K}}{\longrightarrow} x^{*}$, we can deduce that $d_{\ell}(x^{*})=0$. The desired result follows.
\end{proof}
\vspace{2mm}
\begin{remark}
	The continuity of $d_{\ell}$ plays a key role in proving the global convergence of PGMO. For BBPGMO, the corresponding condition can be replaced by Lemma \ref{lemma}(ii). 
\end{remark}
\subsubsection{Barzilai-Borwein proximal gradient method with line search}
\par For each iteration $k$, once the unique descent direction $d^{k}\neq0$ is obtained, the classical Armijo technique is employed for line search. 

\begin{algorithm}  
	\caption{\ttfamily Armijo\_line\_search}\label{alg1} 
	\begin{algorithmic}[1]
		\REQUIRE{ $x^{k}\in\mathbb{R}^{n},d^{k}\in\mathbb{R}^{n},Jf(x^{k})\in\mathbb{R}^{m\times n},\sigma,\gamma\in(0,1), t_{k}=1$}
		\WHILE{$F(x^{k}+t_{k} d^{k})- F(x^{k}) \not\preceq t_{k}\sigma  (Jf(x^{k})d^{k}+g(x^{k}+ d^{k})-g(x^{k}))$}
		\STATE{Update $t_{k}:= \gamma t_{k}$ } 
		\ENDWHILE
		\RETURN{$t_{k}$}
	\end{algorithmic}
\end{algorithm}

\par The following result demonstrates that the Armijo technique will accept a stepsize along with $d^{k}\neq0$.
\vspace{2mm}
\begin{lemma}
	Assume that $x^{k}\in\mathbb{R}^{n}$ is not Pareto critical. Then there exists $\bar{t}_{k}\in (0,1]$ such that
	$$F(x^{k}+t d^{k})- F(x^{k}) \preceq t\sigma(Jf(x^{k})d^{k}+g(x^{k}+ d^{k})-g(x^{k}))$$
	holds for all $t\in (0, \bar{t}_{k}]$.
\end{lemma}
\begin{proof}
	The proof is similar to \cite[Lemma 3.3]{TFY2019}, we omit it here.
\end{proof}
\vspace{2mm}
\par The stepsize obtained by Algorithm \ref{alg1} has a lower bound. 
\vspace{2mm}
\begin{lemma}
	Assume $f_{i}$ is $L_{i}$-smooth for $i\in[m]$, then the stepsize generated by Algorithm \ref{alg1} satisfies $t_{k}\geq t_{\min}:=\min\left\{\bar{t},1\right\}$, where $\bar{t}:=\min\{\frac{2\gamma(1-\sigma)\alpha_{\min}}{L_{i}}:i\in[m]\}$.
\end{lemma}
\vspace{2mm}
\begin{proof} It is sufficient to prove $t_{k}<1$, then backtracking is conducted, leading to the inequality:
	\begin{equation}\label{E4.4}
		F_{i}\left(x^{k}+\frac{t_{k}}{\gamma}d^{k}\right)-F_{i}(x^{k})>\sigma\frac{t_{k}}{\gamma}(\left\langle\nabla f_{i}(x^{k}),d^{k}\right\rangle + g_{i}(x^{k}+ d^{k})-g_{i}(x^{k}))
	\end{equation}
	for some $i\in[m]$. Since $f_{i}$ is $L_{i}$-smooth for $i\in[m]$, we can derive the following inequalities:
	\begin{align*}
		F_{i}\left(x^{k}+\frac{t_{k}}{\gamma}d^{k}\right)-F_{i}(x^{k})&\leq\frac{t_{k}}{\gamma}\left\langle\nabla f_{i}(x^{k}),d^{k}\right\rangle + g_{i}(x^{k}+\frac{t_{k}}{\gamma}d^{k}) - g_{i}(x^{k})+ \frac{L_{i}}{2}\left\|\frac{t_{k}}{\gamma}d^{k}\right\|^{2}\\
		&\leq\frac{t_{k}}{\gamma}(\left\langle\nabla f_{i}(x^{k}),d^{k}\right\rangle + g_{i}(x^{k}+ d^{k})-g_{i}(x^{k}))+\frac{L_{i}}{2}\left\|\frac{t_{k}}{\gamma}d^{k}\right\|^{2},
	\end{align*}
	where the second inequality follows from the convexity of $g_{i}$.  Combining this inequality with (\ref{E4.4}), we obtain
	$$(\sigma - 1)(\left\langle\nabla f_{i}(x^{k}),d^{k}\right\rangle + g_{i}(x^{k}+ d^{k})-g_{i}(x^{k}))\leq\frac{L_{i}t_{k}}{2\gamma}\left\|d^{k}\right\|^{2}$$
	for some $i\in[m]$. Utilizing (\ref{E6}), we arrive at
	\begin{equation}\label{et}
		t_{k}\geq\frac{2\gamma(1-\sigma)\alpha^{k}_{i}}{L_{i}}
	\end{equation}
	for some $i\in[m]$, it holds that $t_{k}\geq t_{\min}$. This completes the proof.
\end{proof}
\vspace{2mm}
%\begin{remark}
%	As mentioned in \cite[Lemma 4]{CTY2023}, the upper bound of the stepsize is determined under the assumption of strong convexity. However, for BBPGMO, we cannot establish such an upper bound for the stepsize. The main reason is that the relation (\ref{E6}) holds with an inequality, and the equivalence between the left and right-hand sides of (\ref{E6}) is unknown.
%\end{remark}

The Barzilai-Borwein proximal gradient method for MCOPs with line search is described as follows.
\begin{algorithm}  
	\caption{{\ttfamily{Barzilai-Borwein\_proximal\_gradient\_method\_for\_MCOPs}}}\label{alg2} 
	\begin{algorithmic}[1]
		\REQUIRE{$x^{0}\in\mathbb{R}^{n}$}
		\STATE{Choose $x^{-1}$ in a small neighborhood of $x^{0}$}
		\FOR{$k=0,...$}
		\STATE{Update $\alpha^{k}_{i}$ as (\ref{alpha_k}),\ $i\in[m]$}
		\STATE{Update $P_{\alpha^{k}}(x^{k}):=	\mathop{\arg\min}\limits_{x\in\mathbb{R}^{n}}\max\limits_{i\in[m]}\left\{\frac{
				\left\langle\nabla f_{i}(x^{k}),x-x^{k}\right\rangle + g_{i}(x)-g_{i}(x^{k})}{\alpha^{k}_{i}}+\frac{1}{2}\|x-x^{k}\|^{2}\right\}$}
		\STATE{Update $d^{k}:= P_{\alpha^{k}}(x^{k}) - x^{k}$}
		\IF{$d^{k}=0$}
		\RETURN{Pareto critical point $x^{k}$  }
		\ELSE{
			\STATE{Update $t_{k}:=$ {\ttfamily Armijo\_line\_search}$\left(x^{k},d^{k},Jf(x^{k})\right)$}
			\STATE{Update $x^{k+1}:= x^{k}+t_{k}d^{k}$}}
		\ENDIF
		\ENDFOR
	\end{algorithmic}
\end{algorithm}

\subsubsection{Adaptive Barzilai-Borwein proximal gradient method}

Under the assumption that $f_{i}$ is $L_{i}$-smooth for $i\in[m]$, we also devise the following Barzilai-Borwein proximal gradient method for MCOPs without line search. 

\begin{algorithm}  
	\caption{{\ttfamily{New\_proximal\_gradient\_method\_for\_MCOPs\_without\_line\_search}}}\label{newalg}
	\begin{algorithmic}[1] 
		\REQUIRE{$x^{0}\in\mathbb{R}^{n}$}
		\FOR{$k=0,...$}
		\STATE{Compute $x^{k+1}$ by solving subproblem (\ref{dk}) with $\alpha^{k}_{i}=L_{i},~i\in[m]$}
		\IF{$x^{k+1}=x^{k}$}
		\RETURN{Pareto critical point $x^{k}$  } 
		\ENDIF
		\ENDFOR
	\end{algorithmic}
\end{algorithm}
\vspace{2mm}
\begin{remark}
	It is worth noting that Algorithm \ref{algwo} utilizes the maximal global smoothness parameter for all objectives, which may be too conservative for objectives with small global smoothness parameters. Instead, Algorithm \ref{newalg} employs a separate global smoothness parameter for each objective. This strategy can help alleviate interference among the objectives.
\end{remark}
\vspace{2mm}
\par However, in practice, the global smoothness parameter is often unknown and tends to be conservative. To address this issue, we propose an adaptive Barzilai-Borwein proximal gradient method to estimate the local smoothness parameters. The method is described as follows:

\begin{algorithm}  
	\caption{{\ttfamily{\small Adaptive\_Barzilai-Borwein\_proximal\_gradient\_method\_for\_MCOPs}}}\label{alg3}
	\begin{algorithmic}[1]  
		\REQUIRE{$x^{0}\in\mathbb{R}^{n},~\tau>1$.}
		\STATE{Choose $x^{-1}$ in a small neighborhood of $x^{0}$}
		\FOR{$k=0,...$}
		\STATE{Initialize $\alpha^{k}_{i}$ as (\ref{alpha_k})}
		\STATE{Update $x^{k+1}:=	\mathop{\arg\min}\limits_{x\in\mathbb{R}^{n}}\max\limits_{i\in[m]}\left\{\frac{
				\left\langle\nabla f_{i}(x^{k}),x-x^{k}\right\rangle + g_{i}(x)-g_{i}(x^{k})}{\alpha^{k}_{i}}+\frac{1}{2}\|x-x^{k}\|^{2}\right\}$}
		\IF{$x^{k+1}=x^{k}$}
		\RETURN{Pareto critical point $x^{k}$}
		\ELSE{
			\REPEAT{
				\FOR{$i=1,...,m$}
				\IF{$f_{i}(x^{k+1})-f_{i}(x^{k})>\left\langle\nabla f_{i}(x^{k}),x^{k+1}-x^{k}\right\rangle+\frac{\alpha^{k}_{i}}{2}\|x^{k+1}-x^{k}\|^{2} $}
				\STATE{Update $\alpha^{k}_{i}:=\tau\alpha^{k}_{i}$}
				\ENDIF
				\ENDFOR}
			\STATE{Update $x^{k+1}:=	\mathop{\arg\min}\limits_{x\in\mathbb{R}^{n}}\max\limits_{i\in[m]}\left\{\frac{
					\left\langle\nabla f_{i}(x^{k}),x-x^{k}\right\rangle + g_{i}(x)-g_{i}(x^{k})}{\alpha^{k}_{i}}+\frac{1}{2}\|x-x^{k}\|^{2}\right\}$}
			\UNTIL{$f_{i}(x^{k+1})-f_{i}(x^{k})\leq\left\langle\nabla f_{i}(x^{k}),x^{k+1}-x^{k}\right\rangle+\frac{\alpha^{k}_{i}}{2}\|x^{k+1}-x^{k}\|^{2},~i\in[m] $}
		}
		\ENDIF
		\ENDFOR
	\end{algorithmic}
\end{algorithm}

In Algorithm \ref{alg3}, lines 8-15 are responsible for estimating the local smoothness parameter for $f_{i},~i\in[m]$. The following proposition demonstrates that the procedure is well-defined.
\vspace{2mm}
\begin{proposition}
	If $f_{i}$ is $L_{i}$-smooth for $i\in[m]$, then the {\it {repeat} loop} of Algorithm \ref{alg3} terminates in a finite number of iterations, and $\alpha_{i}^{k}<\tau L_{i},~i\in[m]$.
\end{proposition}
\vspace{2mm}
\begin{proof}
	Note that $f_{i}$ is $L_{i}$-smooth for $i\in[m]$, we have
	$$f_{i}(x^{k+1})-f_{i}(x^{k})\leq\left\langle\nabla f_{i}(x^{k}),x^{k+1}-x^{k}\right\rangle+\frac{L_{i}}{2}\|x^{k+1}-x^{k}\|^{2},~i\in[m].$$
	The desired results follow directly from this inequality.
\end{proof}
\vspace{2mm}
\begin{remark}[Adaptivity to local smoothness]
	The Barzilai-Borwein's rule (line 3) plays an important role in estimating the initial smoothness parameters. One significant advantage of this approach, compared to simply setting $\alpha^{k}_{i}=L_{i}$, is that lines 8-15 can adapt to the local smoothness based on the iteration trajectory. As a result, the procedure can better adapt to the characteristics of the problem and improve real-world performance significantly.
\end{remark}
\subsection{Merit function}
Before presenting the convergence results of BBPGMO, we introduce two types of merit functions for (\ref{MCOP}) that quantify the gap between the current point and the optimal solution. The merit functions will be used in convergence rates analysis.
\begin{equation}\label{u}
	u_{0}^{\alpha}(x):=\sup\limits_{y\in\mathbb{R}^{n}}\min\limits_{i\in[m]}\left\{\frac{F_{i}(x)-F_{i}(y)}{\alpha_{i}}\right\},
\end{equation}

\begin{equation}\label{v}
	w_{\ell}^{\alpha}(x):=\max\limits_{y\in\mathbb{R}^{n}}\min\limits_{i\in[m]}\left\{\frac{
		\left\langle\nabla f_{i}(x),x-y\right\rangle + g_{i}(x)-g_{i}(y)}{\alpha_{i}}-\frac{\ell}{2}\|x-y\|^{2}\right\},
\end{equation}
where $\alpha\in\mathbb{R}^{m}_{++}$, $\ell>0$.

\par We can demonstrate that $u_{0}^{\alpha}$ and $v_{\ell}^{\alpha}$ serve as merit functions, satisfying the criteria of weak Pareto and critical point, respectively.
\vspace{2mm}
\begin{proposition}\label{p1}
	Let $u_{0}^{\alpha}$ and $v_{\ell}^{\alpha}$ be defined as {\rm(\ref{u})} and {\rm(\ref{v})}, respectively. Then, the following statements hold.
	\begin{itemize}
		\item[$\mathrm{(i)}$]  $x\in\mathbb{R}^{n}$ is a weak Pareto solution of {\rm(\ref{MCOP})} if and only if $u_{0}^{\alpha}(x)=0$.
		\item[$\mathrm{(ii)}$]  $x\in\mathbb{R}^{n}$ is a Pareto critical point of {\rm(\ref{MCOP})} if and only if $w_{\ell}^{\alpha}(x)=0$.
	\end{itemize}
\end{proposition}
\vspace{2mm}
\begin{proof}
	The assertion (i) and (ii) can be obtained by using the same arguments as in the proofs of \cite[Theorem 3.1]{TFY2020} and \cite[Theorem 3.9]{TFY2020}, respectively.
\end{proof}
\vspace{2mm}
\begin{proposition}\label{p2}
	Let $u_{0}^{\alpha}$ and $v_{\ell}^{\alpha}$ be defined as {\rm(\ref{u})} and {\rm(\ref{v})}, respectively. Then, the following statements hold.
	\begin{itemize}
		\item[$\mathrm{(i)}$] If $0\prec\alpha^{2}\preceq\alpha^{1}$, then 
		$$u_{0}^{\alpha^{1}}(x)\leq u_{0}^{\alpha^{2}}(x)\leq \max\limits_{i\in[m]}\left\{\frac{\alpha^{1}_{i}}{\alpha^{2}_{i}}\right\} u_{0}^{\alpha^{1}}(x),\ \forall x\in\mathbb{R}^{n}.$$
		\item[$\mathrm{(ii)}$] If $0<\ell\leq r$, then
		$$w_{r}^{\alpha}(x)\leq w_{\ell}^{\alpha}(x)\leq\frac{r}{\ell}w_{r}^{\alpha}(x),\ \forall x\in\mathbb{R}^{n}. $$
		\item[$\mathrm{(iii)}$] If $0\prec\alpha^{2}\preceq\alpha^{1}$, then 
		$$w_{\ell}^{\alpha^{1}}(x)\leq w_{\ell}^{\alpha^{2}}(x)\leq\left(\max\limits_{i\in[m]}\left\{\frac{\alpha^{1}_{i}}{\alpha^{2}_{i}}\right\}\right)^{2}w_{\ell}^{\alpha^{1}}(x),\ \forall x\in\mathbb{R}^{n}. $$
		%		\item[$\mathrm{(iii)}$] If $\nabla f_{i}$ is Lipschitz continuous with constant $L_{i}>0$, then
		%		$$u_{0}^{\alpha}(x)\leq\max\limits_{i\in[m]}\left\{\frac{\alpha_{i}}{\mu_{i}}\right\}w_{1}^{\alpha}(x),\ \forall x\in\mathbb{R}^{n}. $$
		%		\item[$\mathrm{(iv)}$] If $f_{i}$ is strongly convex with modulus $\mu_{i}>0$, then
		%		$$u_{0}^{\alpha}(x)\leq\max\limits_{i\in[m]}\left\{\frac{\alpha_{i}}{\mu_{i}}\right\}w_{1}^{\alpha}(x),\ \forall x\in\mathbb{R}^{n}. $$
		
	\end{itemize}
\end{proposition}
\vspace{2mm}
\begin{proof}
	(i) Assertion (i) follows directly by the definition of $u_{0}^{\alpha}(x)$.
	\par (ii) The assertion can be obtained by using the same arguments as in the proof of \cite[Theorem 4.2]{TFY2020}.
	\par (iii) From the definition of $w_{\ell}^{\alpha}(x)$, we obtain $w_{\ell}^{\alpha^{1}}(x)\leq w_{\ell}^{\alpha^{2}}(x)$. Next, we need to prove $$w_{\ell}^{\alpha^{2}}(x)\leq\left(\max\limits_{i\in[m]}\left\{\frac{\alpha^{1}_{i}}{\alpha^{2}_{i}}\right\}\right)^{2}w_{\ell}^{\alpha^{1}}(x).$$ 
	A direct calculation gives
	\begin{align*}
		&~~~w_{\ell}^{\alpha^{2}}(x)\\
		&=\max\limits_{y\in\mathbb{R}^{n}}\min\limits_{i\in[m]}\left\{\frac{
			\left\langle\nabla f_{i}(x),x-y\right\rangle + g_{i}(x)-g_{i}(y)}{\alpha_{i}^{2}}-\frac{\ell}{2}\|x-y\|^{2}\right\}\\
		&=\max\limits_{y\in\mathbb{R}^{n}}\min\limits_{i\in[m]}\left\{\frac{\alpha_{i}^{1}}{\alpha_{i}^{2}}\frac{
			\left\langle\nabla f_{i}(x),x-y\right\rangle + g_{i}(x)-g_{i}(y)}{\alpha_{i}^{1}}-\frac{\ell}{2}\|x-y\|^{2}\right\}\\
		&\leq r_{\max}\max\limits_{y\in\mathbb{R}^{n}}\min\limits_{i\in[m]}\left\{\frac{
			\left\langle\nabla f_{i}(x),x-y\right\rangle + g_{i}(x)-g_{i}(y)}{\alpha_{i}^{1}}-\frac{\ell}{2r_{\max}}\|x-y\|^{2}\right\}\\
		&\leq r_{\max}^{2}\max\limits_{y\in\mathbb{R}^{n}}\min\limits_{i\in[m]}\left\{\frac{
			\left\langle\nabla f_{i}(x),x-y\right\rangle + g_{i}(x)-g_{i}(y)}{\alpha_{i}^{1}}-\frac{\ell}{2}\|x-y\|^{2}\right\},
	\end{align*}
	where $r_{\max}:=\max\limits_{i\in[m]}\left\{\frac{\alpha^{1}_{i}}{\alpha^{2}_{i}}\right\}\geq1$ and the last inequality is given by the assertion (ii). The desired result follows.
\end{proof}
\section{Convergence rates analysis}\label{sec5}
Relation (\ref{Ediff}) shows that, along the Barzilai-Borwein proximal gradient direction, different objectives can achieve distinct descent. Hence, BBPGMO has the ability to mitigate the imbalances among objectives. Naturally, a question arises that: Does BBPGMO exhibit improved convergence rates? The primary objective of this section is to analyze the convergence rates of BBPGMO, and provide a positive answer to this question.  
\par In Algorithms \ref{alg2}-\ref{alg3}, it can be observed that these algorithms terminate either with a Pareto critical point in a finite number of iterations or generates an infinite sequence of points. In the subsequent analysis, we will assume that these algorithms produce an infinite sequence of noncritical points.
\subsection{Convergence rates analysis of BBPGMO with line search}
First, we analyze the convergence rates of BBPGMO with line search.
\subsubsection{Global convergence}

\begin{theorem}\label{global1}
	Assume that $\Omega=\{x:F(x)\preceq F(x^{0})\}$ is a bounded set. Let $\{x^{k}\}$ be the sequence generated by Algorithm \ref{alg2}. Then, the following statements hold.
	\begin{itemize}
		\item[$\mathrm{(i)}$] $\{x_{k}\}$ has at least one accumulation point, and every accumulation point $x^{*}\in\Omega$ is a Pareto critical point.
		\item[$\mathrm{(ii)}$] If $f_{i}$ is $L_{i}$-smooth for $i\in[m]$, then $$\min\limits_{0\leq s\leq k-1}\|d^{s}\|\leq\frac{\sqrt{\min\limits_{i\in[m]}\{F_{i}(x^{0})-F_{i}^{*}\}}}
		{\sqrt{\sigma t_{\min}\alpha_{\min}}\sqrt{k}}.$$
	\end{itemize}
\end{theorem}
\begin{proof}
	(i) From the boundedness of $\Omega$ and the fact that $\{F(x^{k})\}$ is decreasing, there exists $F^{*}$ such that $F^{*}\preceq F(x^{k})$ and $\lim\limits_{k\rightarrow\infty}F(x^{k})=F^{*}$. The assertion (i) can be obtained by using the same arguments as in the proof of  \cite[Theorem 4.2]{TFY2019}.
	
	(ii) By using (\ref{E6}) and the line search condition, we have
	$$F_{i}(x^{k+1})-F_{i}(x^{k})\leq-\sigma t_{k}\alpha^{k}_{i}\|d^{k}\|^{2}\leq-\sigma t_{\min}\alpha_{\min}\|d^{k}\|^{2},\ \forall i\in[m].$$
	Taking the sum of the above inequality over $0,1,...,k-1$, we obtain
	\begin{align*}
		F_{i}(x^{k})-F_{i}(x^{0})\leq-\sigma t_{\min}\alpha_{\min}\sum\limits_{0\leq s\leq k-1}\|d^{s}\|^{2}, \ \forall i\in[m].
	\end{align*}
	Rearranging the terms and using the fact that $F^{*}\leq F(x^{k})$, we have
	$$\sigma t_{\min}\alpha_{\min}\sum\limits_{0\leq s\leq k-1}\|d^{s}\|^{2}\leq F_{i}(x^{0})-F_{i}^{*}, \ \forall i\in[m].$$
	The desired result follows.
\end{proof}
\vspace{2mm}
\subsubsection{Strong convergence}
Before presenting the strong convergence of Algorithm \ref{alg2}, we recall the following result on non-negative sequences.
\vspace{2mm}
\begin{lemma}[Lemma 2 of \cite{P1987}]\label{l8}
	Let $\{a^{k}\}$ and $\{\epsilon^{k}\}$ be non-negative sequences. Assume that  $a^{k+1}\leq a^{k}+\epsilon^{k}$ and $\sum\limits_{k=0}^{\infty}\epsilon^{k}<\infty$, then $\{a^{k}\}$ converges.
\end{lemma} 
\vspace{2mm}
\begin{theorem}\label{t3}
	Assume that $\Omega=\{x:F(x)\preceq F(x^{0})\}$ is a bounded set, $f_{i}$ is convex and $L_{i}$-smooth for $i\in[m]$. Let $\{x^{k}\}$ be the sequence generated by Algorithm \ref{alg2}. Then, the following statements hold.
	\begin{itemize}
		\item[$\mathrm{(i)}$] $\{x^{k}\}$ converges to some weak Pareto solution $x^{*}$.
		\item[$\mathrm{(ii)}$] There exists a constant $c>0$ such that $u_{0}^{\alpha_{\min}}(x^{k})\leq\frac{c}{k}.$
	\end{itemize}
\end{theorem}
\vspace{2mm}
\begin{proof}
	(i) From the convexity of $F_{i}$ for $i\in[m]$, we have
	\begin{equation}\label{e10}
		F_{i}(x^{k+1})-F_{i}(x^{k})=F_{i}(x^{k}+t_{k}d^{k})-F_{i}(x^{k})\leq t_{k}(F_{i}(x^{k}+d^{k})-F_{i}(x^{k})).
	\end{equation}
	Applying Theorem \ref{global1}(i), denote $x^{*}$ an accumulation point of $\{x^{k}\}$. We use $x^{k+1}=x^{k}+t_{k}d^{k}$ to get
	\begin{equation}\label{e11}
		\|x^{k+1}-x^{*}\|^{2}=\|x^{k}-x^{*}+t_{k}d^{k}\|^{2}=\|x^{k}-x^{*}\|^{2}+t^{2}_{k}\|d^{k}\|^{2}+2t_{k}\left\langle d^{k},x^{k}-x^{*}\right\rangle.
	\end{equation}
	Furthermore, utilizing \cite[Theorem 6.39 (iii)]{B2017} and the expression (\ref{prox}), we can derive the inequality:
	$$\left\langle x^{k}-\sum\limits_{i\in[m]}\lambda_{i}^{k}\frac{\nabla f_{i}(x^{k})}{\alpha_{i}^{k}}-P_{\alpha^{k}}(x^{k}),x^{*}-P_{\alpha^{k}}(x^{k})\right\rangle\leq\sum\limits_{i\in[m]}\lambda_{i}^{k}\frac{g_{i}(x^{*})-g_{i}(P_{\alpha^{k}}(x^{k}))}{\alpha_{i}^{k}}.$$
	By rearranging the terms and utilizing the facts that $f_{i}$ is convex and $L_{i}$-smooth for $i\in[m]$, we can obtain
	\begin{align*}
		&~~~~\left\langle x^{k}-P_{\alpha^{k}}(x^{k}),x^{*}-P_{\alpha^{k}}(x^{k})\right\rangle\\
		&\leq\left\langle\sum\limits_{i\in[m]}\lambda_{i}^{k}\frac{\nabla f_{i}(x^{k})}{\alpha_{i}^{k}},x^{*}-P_{\alpha^{k}}(x^{k})\right\rangle+\sum\limits_{i\in[m]}\lambda_{i}^{k}\frac{g_{i}(x^{*})-g_{i}(P_{\alpha^{k}}(x^{k}))}{\alpha_{i}^{k}}\\
		&=\left\langle\sum\limits_{i\in[m]}\lambda_{i}^{k}\frac{\nabla f_{i}(x^{k})}{\alpha_{i}^{k}},x^{*}-x^{k}\right\rangle+\left\langle\sum\limits_{i\in[m]}\lambda_{i}^{k}\frac{\nabla f_{i}(x^{k})}{\alpha_{i}^{k}},x^{k}-P_{\alpha^{k}}(x^{k})\right\rangle\\
		&~~~~+\sum\limits_{i\in[m]}\lambda_{i}^{k}\frac{g_{i}(x^{*})-g_{i}(P_{\alpha^{k}}(x^{k}))}{\alpha_{i}^{k}}\\
		&\leq\sum\limits_{i\in[m]}\lambda_{i}^{k}\frac{f_{i}(x^{*})-f_{i}(x^{k})}{\alpha_{i}^{k}}+\left\langle\sum\limits_{i\in[m]}\lambda_{i}^{k}\frac{\nabla f_{i}(P_{\alpha^{k}}(x^{k}))}{\alpha_{i}^{k}},x^{k}-P_{\alpha^{k}}(x^{k})\right\rangle\\
		&~~~~+
		\left\langle\sum\limits_{i\in[m]}\frac{\lambda_{i}^{k}(\nabla f_{i}(x^{k})-\nabla f_{i}(P_{\alpha^{k}}(x^{k})))}{\alpha_{i}^{k}},x^{k}-P_{\alpha^{k}}(x^{k})\right\rangle\\
		&~~~~+ \sum\limits_{i\in[m]}\lambda_{i}^{k}\frac{g_{i}(x^{*})-g_{i}(P_{\alpha^{k}}(x^{k}))}{\alpha_{i}^{k}}
	\end{align*}
	\begin{align*}
		&\leq\sum\limits_{i\in[m]}\lambda_{i}^{k}\frac{f_{i}(x^{*})-f_{i}(x^{k})}{\alpha_{i}^{k}}+\sum\limits_{i\in[m]}\lambda_{i}^{k}\frac{f_{i}(x^{k})-f_{i}(P_{\alpha^{k}}(x^{k}))}{\alpha_{i}^{k}}+\sum\limits_{i\in[m]}\frac{\lambda_{i}^{k}}{\alpha_{i}^{k}}L_{i}\|d^{k}\|^{2}\\
		&~~~~+\sum\limits_{i\in[m]}\lambda_{i}^{k}\frac{g_{i}(x^{*})-g_{i}(P_{\alpha^{k}}(x^{k}))}{\alpha_{i}^{k}}\\
		&=\sum\limits_{i\in[m]}\frac{\lambda_{i}^{k}}{\alpha_{i}^{k}}L_{i}\|d^{k}\|^{2}+\sum\limits_{i\in[m]}\lambda_{i}^{k}\frac{F_{i}(x^{*})-F_{i}(P_{\alpha^{k}}(x^{k}))}{\alpha_{i}^{k}}\\
		&\leq\sum\limits_{i\in[m]}\frac{\lambda_{i}^{k}}{\alpha_{i}^{k}}L_{i}\|d^{k}\|^{2}+\sum\limits_{i\in[m]}\lambda_{i}^{k}\frac{F_{i}(x^{k})-F_{i}(P_{\alpha^{k}}(x^{k}))}{\alpha_{i}^{k}}\\
		&\leq\sum\limits_{i\in[m]}\frac{\lambda_{i}^{k}}{\alpha_{i}^{k}}L_{i}\|d^{k}\|^{2}+\sum\limits_{i\in[m]}\lambda_{i}^{k}\frac{F_{i}(x^{k})-F_{i}(x^{k+1})}{\alpha_{i}^{k}t_{k}},
	\end{align*}
	where the last inequality is due to (\ref{e10}).
	By substituting $P_{\alpha^{k}}(x^{k})=x^{k}+d^{k}$, we have
	\begin{align*}
		\left\langle d^{k},x^{k}-x^{*}\right\rangle&=\left\langle x^{k}-P_{\alpha^{k}}(x^{k}),x^{*}-x^{k}\right\rangle\\
		&=\left\langle x^{k}-P_{\alpha^{k}}(x^{k}),x^{*}-P_{\alpha^{k}}(x^{k})\right\rangle-\|d^{k}\|^{2}\\
		&\leq\sum\limits_{i\in[m]}\frac{\lambda_{i}^{k}}{\alpha_{i}^{k}}L_{i}\|d^{k}\|^{2}+\sum\limits_{i\in[m]}\lambda_{i}^{k}\frac{F_{i}(x^{k})-F_{i}(x^{k+1})}{\alpha_{i}^{k}t_{k}}-\|d^{k}\|^{2}.
	\end{align*}
	Substituting the above inequality into (\ref{e11}), it follows that
	\begin{align*}
		&~~~~\|x^{k+1}-x^{*}\|^{2}\\
		&\leq\|x^{k}-x^{*}\|^{2}+(t^{2}_{k}-2t_{k})\|d^{k}\|^{2}+2\sum\limits_{i\in[m]}\frac{\lambda_{i}^{k}}{\alpha_{i}^{k}}L_{i}t_{k}\|d^{k}\|^{2}+2\sum\limits_{i\in[m]}\lambda_{i}^{k}\frac{F_{i}(x^{k})-F_{i}(x^{k+1})}{\alpha_{i}^{k}}\\
		&\leq\|x^{k}-x^{*}\|^{2}+2\left(\sum\limits_{i\in[m]}\frac{\lambda_{i}^{k}}{\alpha_{i}^{k}}L_{i}t_{k}\|d^{k}\|^{2}+\sum\limits_{i\in[m]}\lambda_{i}^{k}\frac{F_{i}(x^{k})-F_{i}(x^{k+1})}{\alpha_{i}^{k}}\right),
	\end{align*}
	where the second inequality follows by $0<t_{k}\leq1$. Denote
	$$\epsilon^{k}:=\sum\limits_{i\in[m]}\frac{\lambda_{i}^{k}}{\alpha_{i}^{k}}L_{i}t_{k}\|d^{k}\|^{2}+\sum\limits_{i\in[m]}\lambda_{i}^{k}\frac{F_{i}(x^{k})-F_{i}(x^{k+1})}{\alpha_{i}^{k}}.$$
	We use (\ref{E6}) to get
	\begin{align*}
		\sum\limits_{k=0}^{\infty}\epsilon^{k}&=	\sum\limits_{k=0}^{\infty}\left(\sum\limits_{i\in[m]}\frac{\lambda_{i}^{k}}{\alpha_{i}^{k}}L_{i}t_{k}\|d^{k}\|^{2}+\sum\limits_{i\in[m]}\lambda_{i}^{k}\frac{F_{i}(x^{k})-F_{i}(x^{k+1})}{\alpha_{i}^{k}}\right)\\
		&\leq\sum\limits_{k=0}^{\infty}\left(\sum\limits_{i\in[m]}\frac{\lambda_{i}^{k}L_{i}(F_{i}(x^{k})-F_{i}(x^{k+1}))}{\sigma(\alpha_{i}^{k})^{2}}+\sum\limits_{i\in[m]}\lambda_{i}^{k}\frac{F_{i}(x^{k})-F_{i}(x^{k+1})}{\alpha_{i}^{k}}\right)\\
		&\leq\sum\limits_{k=0}^{\infty}\left(\sum\limits_{i\in[m]}\left(\frac{L_{\max}}{\sigma(\alpha_{\min})^{2}}+\frac{1}{\alpha_{\min}}\right)(F_{i}(x^{k})-F_{i}(x^{k+1}))\right)\\
		&\leq\sum\limits_{i\in[m]}\left(\frac{L_{\max}}{\sigma(\alpha_{\min})^{2}}+\frac{1}{\alpha_{\min}}\right)(F_{i}(x^{0})-F_{i}(x^{*}))\\
		&<+\infty,
	\end{align*} 
	where the second inequality is given by $\lambda^{k}_{i}\leq 1$.
	Then, by applying Lemma \ref{l8}, we can conclude that the sequence $\{\|x^{k}-x^{*}\|\}$ converges. This, together with the fact that $x^{*}$ is an accumulation point of $\{x^{k}\}$, implies $\{x^{k}\}$ converges to $x^{*}$. Moreover, the convexity of $f_{i}$ yields that $x^{*}$ is a weakly Pareto solution. 
	
	(ii) Dividing (\ref{e10}) by $\alpha^{k}_{i}$ and using the $L_{i}$-smoothness of $f_{i}$, we can deduce the following result:
	\begin{align*}
		&~~~~\frac{F_{i}(x^{k+1})-F_{i}(x^{k})}{\alpha^{k}_{i}}\\
		&\leq t_{k}\left(\frac{F_{i}(x^{k}+d^{k})-F_{i}(x^{k})}{\alpha^{k}_{i}}\right)\\
		&\leq t_{k}\left(\frac{
			\left\langle\nabla f_{i}(x^{k}),d^{k}\right\rangle + g_{i}(x^{k}+d^{k})-g_{i}(x^{k})}{\alpha^{k}_{i}}+\frac{L_{i}}{2\alpha^{k}_{i}}\|d^{k}\|^{2}\right)\\
		&\leq t_{k}\left(\frac{
			\left\langle\nabla f_{i}(x^{k}),d^{k}\right\rangle + g_{i}(x^{k}+d^{k})-g_{i}(x^{k})}{\alpha^{k}_{i}}+\frac{1}{2}\|d^{k}\|^{2}\right)+\frac{t_{k}L_{i}}{2\alpha^{k}_{i}}\|d^{k}\|^{2}.
	\end{align*}
	Recall that $x^{k}+d^{k}$ is the minimizer of (\ref{dk}) and $f_{i}$ is convex for all $i\in[m]$, we can derive the following inequalities:
	\begin{align*}
		&\ \frac{
			\left\langle\nabla f_{i}(x^{k}),d^{k}\right\rangle + g_{i}(x^{k}+d^{k})-g_{i}(x^{k})}{\alpha^{k}_{i}}+\frac{1}{2}\|d^{k}\|^{2}\\
		\leq&\ \max\limits_{i\in[m]}\left\{\frac{
			\left\langle\nabla f_{i}(x^{k}),d^{k}\right\rangle + g_{i}(x^{k}+d^{k})-g_{i}(x^{k})}{\alpha^{k}_{i}}+\frac{1}{2}\|d^{k}\|^{2}\right\}\\
		=&\ \min\limits_{x\in\mathbb{R}^{n}}\max\limits_{i\in[m]}\left\{\frac{
			\left\langle\nabla f_{i}(x^{k}),x-x^{k}\right\rangle + g_{i}(x)-g_{i}(x^{k})}{\alpha^{k}_{i}}+\frac{1}{2}\|x-x^{k}\|^{2}\right\}\\
		\leq&\ \min\limits_{x\in\mathbb{R}^{n}}\max\limits_{i\in[m]}\left\{\frac{
			F_{i}(x)-F_{i}(x^{k})}{\alpha^{k}_{i}}+\frac{1}{2}\|x-x^{k}\|^{2}\right\}.
	\end{align*}
	Now, let's select $x = (1-\delta) x^{k} + \delta y$ with $\delta\in[0,1]$ and $y\in\mathop{\arg\max}\limits_{x\in\mathbb{R}^{n}}\min\limits_{i\in[m]}\left\{\frac{F_{i}(x^{k})-F_{i}(x)}{\alpha^{k}_{i}}\right\}$. From the convexity of $F_{i}$, $i\in[m]$, we obtain
	\begin{equation}\label{E9}
		\begin{aligned}
			&~~~~\min\limits_{x\in\mathbb{R}^{n}}\max\limits_{i\in[m]}\left\{\frac{
				F_{i}(x)-F_{i}(x^{k})}{\alpha^{k}_{i}}+\frac{1}{2}\|x-x^{k}\|^{2}\right\}\\
			&\leq\max\limits_{i\in[m]}\left\{\frac{F_{i}(y)-F_{i}(x^{k})}{\alpha^{k}_{i}}\right\}\delta + \frac{\delta^{2}}{2}\|y-x^{k}\|^{2}\\
			&=-\max\limits_{x\in\mathbb{R}^{n}}\min\limits_{i\in[m]}\left\{\frac{F_{i}(x^{k})-F_{i}(x)}{\alpha^{k}_{i}}\right\}\delta + \frac{\delta^{2}}{2}\|y-x^{k}\|^{2}\\
			&= -u_{0}^{\alpha^{k}}(x^{k})\delta + \frac{\delta^{2}}{2}\|y-x^{k}\|^{2}.
		\end{aligned}
	\end{equation}
	By the definition of $y$, we can deduce that $y\in\mathcal{L}_{k}:=\{x:F(x)\preceq F(x^{k})\}$. On the other hand, the monotonicity of the $\{x^{k}\}$ implies $\mathcal{L}_{k}\subset\Omega$. Therefore, considering the boundedness of $\Omega$ (denoted by $c_{1}$ as the diameter of $\Omega$), we can conclude that
	$$\|y-x^{k}\|\leq c_{1}.$$
	Substituting the bound into (\ref{E9}), we obtain 
	\begin{equation}\label{E10}
		\begin{aligned}
			\min\limits_{x\in\mathbb{R}^{n}}\max\limits_{i\in[m]}\left\{\frac{
				F_{i}(x)-F_{i}(x^{k})}{\alpha^{k}_{i}}+\frac{1}{2}\|x-x^{k}\|^{2}\right\}&\leq-u_{0}^{\alpha^{k}}(x^{k})\delta + \frac{c_{1}^{2}}{2}\delta^{2}.
		\end{aligned}
	\end{equation}
	By the arbitrary of $\delta\in[0,1]$, we observe that the minimum on the right-hand side of (\ref{E10}) is attained at
	$$\delta_{\min}=\min\left\{1,\frac{u_{0}^{\alpha^{k}}(x^{k})}{c^{2}_{1}}\right\}.$$
	Since $\{x^{k}\}$ converges to some weak Pareto solution, it follows by Proposition \ref{p1} (ii) that $u_{0}^{\alpha^{k}}(x^{k})$ converges to $0$ ($\alpha^{k}\geq\alpha_{\min}$). Then there exists $k_{0}$ such that $\delta_{\min}<1$ for all $k\geq k_{0}$. Consquently, for $k\geq k_{0}$,
	\begin{equation}\label{E11}
		\begin{aligned}
			&~~~~\frac{F_{i}(x^{k+1})-F_{i}(x^{k})}{\alpha^{k}_{i}}\\
			&\leq t_{k}\left(\frac{
				\left\langle\nabla f_{i}(x^{k}),d^{k}\right\rangle + g_{i}(x^{k}+d^{k})-g_{i}(x^{k})}{\alpha^{k}_{i}}+\frac{1}{2}\|d^{k}\|^{2}\right)+\frac{t_{k}L_{i}}{2\alpha^{k}_{i}}\|d^{k}\|^{2}\\
			&\leq -t_{k}\frac{(u_{0}^{\alpha^{k}}(x^{k}))^{2}}{2c^{2}_{1}} + \frac{t_{k}L_{i}}{2\alpha^{k}_{i}}\|d^{k}\|^{2}.
		\end{aligned}
	\end{equation}
	On the other hand, we can utilize the Armijo line search condition and (\ref{E6}) to derive the following inequality:
	\begin{align*}
		F_{i}(x^{k+1}) - F_{i}(x^{k}) &\leq t_{k}\sigma(\left\langle\nabla f_{i}(x^{k}),d^{k}\right\rangle + g_{i}(x^{k}+d^{k})-g_{i}(x^{k}))\\
		&\leq -\sigma\alpha^{k}_{i}t_{k}\|d^{k}\|^{2}.
	\end{align*}
	Hence, $$t_{k}\|d^{k}\|^{2}\leq\frac{1}{\sigma\alpha^{k}_{i}}(F_{i}(x^{k})-F_{i}(x^{k+1})).$$
	Substituting the preceding relation into (\ref{E11}), for all $k\geq k_{0}$, we have 
	\begin{equation*}
		\begin{aligned}
			\frac{F_{i}(x^{k+1})-F_{i}(x^{k})}{\alpha^{k}_{i}}&\leq-t_{k}\frac{(u_{0}^{\alpha^{k}}(x^{k}))^{2}}{2c^{2}_{1}} +\frac{L_{i}(F_{i}(x^{k})-F_{i}(x^{k+1}))}{2\sigma(\alpha^{k}_{i})^{2}}\\
			&\leq -t_{\min}\frac{(u_{0}^{\alpha^{k}}(x^{k}))^{2}}{2c^{2}_{1}}+\frac{L_{i}}{2\sigma\alpha^{k}_{i}}\frac{F_{i}(x^{k})-F_{i}(x^{k+1})}{\alpha^{k}_{i}}.
		\end{aligned}
	\end{equation*}
	Rearranging and utilizing the fact that $\alpha_{\min}\leq\alpha^{k}_{i}\leq\alpha_{\max}$, we have
	\begin{equation}\label{E12}
		\frac{F_{i}(x^{k+1})-F_{i}(x^{k})}{\alpha_{\min}}\leq -c_{2}(u_{0}^{\alpha_{\max}}(x^{k}))^{2},\ \forall k\geq k_{0},
	\end{equation}
	where $c_{2} = \frac{t_{\min}}{2c^{2}_{1}(1+\frac{L_{\max}}{2\sigma\alpha_{\min}})}$, with $L_{\max}:=\max\{L_{i}:i\in[m]\}$. Rearranging and taking the minimum and supremum with respect to $i\in [m]$ and $x\in\mathbb{R}^{n}$ on both sides, respectively, we obtain
	$$\sup\limits_{x\in\mathbb{R}^{n}}\min\limits_{i\in[m]}\frac{F_{i}(x^{k+1})-F_{i}(x)}{\alpha_{\min}}\leq\sup\limits_{x\in\mathbb{R}^{n}}\min\limits_{i\in[m]}\frac{F_{i}(x^{k})-F_{i}(x)}{\alpha_{\min}}-c_{2}(u_{0}^{\alpha_{\max}}(x^{k}))^{2},\ \forall k\geq k_{0},$$
	namely,
	$$u_{0}^{\alpha_{\min}}(x^{k+1})\leq u_{0}^{\alpha_{\min}}(x^{k})-c_{2}(u_{0}^{\alpha_{\max}}(x^{k}))^{2},\ \forall k\geq k_{0}.$$
	Dividing the above inequality by $u_{0}^{\alpha_{\min}}(x^{k+1})u_{0}^{\alpha_{\min}}(x^{k})$ and using $u_{0}^{\alpha_{\min}}(x^{k+1})\leq u_{0}^{\alpha_{\min}}(x^{k})$ (the monotonicity of $\{F(x^{k})\}$), for all $k\geq k_{0}$, we have
	$$\frac{1}{u_{0}^{\alpha_{\min}}(x^{k})}\leq\frac{1}{u_{0}^{\alpha_{\min}}(x^{k+1})}-c_{2}\left(\frac{u_{0}^{\alpha_{\max}}(x^{k})}{u_{0}^{\alpha_{\min}}(x^{k})}\right)^{2}=\frac{1}{u_{0}^{\alpha_{\min}}(x^{k+1})}-c_{2}\left(\frac{\alpha_{\min}}{\alpha_{\max}}\right)^{2}.$$
	Now, taking the sum of the preceding relation over $k_{0},k_{0}+1,...,k-1$, we obtain
	$$\frac{1}{u_{0}^{\alpha_{\min}}(x^{k_{0}})}\leq\frac{1}{u_{0}^{\alpha_{\min}}(x^{k})}-c_{2}\left(\frac{\alpha_{\min}}{\alpha_{\max}}\right)^{2}(k-k_{0}-1),\ \forall k>k_{0}.$$
	It follows that
	$$u_{0}^{\alpha_{\min}}(x^{k})\leq\frac{c_{3}}{k-k_{0}-1},\ \forall k>k_{0},$$
	where $c_{3} =\frac{1}{c_{2}}\left(\frac{\alpha_{\max}}{\alpha_{\min}}\right)^{2}$. Without loss of generality, there exists $c>0$ such that
	$$u_{0}^{\alpha_{\min}}(x^{k})\leq\frac{c}{k},\ \forall k.$$
	This completes the proof.
\end{proof}
\vspace{2mm}
\subsubsection{Linear convergence}
\begin{theorem}\label{t4}
	Assume that $f_{i}$ is $L_{i}$-smooth and strongly convex with modulus $\mu_{i}>0$, for $i\in[m]$. Let $\{x^{k}\}$ be the sequence generated by Algorithm \ref{alg2}. Then, the following statements hold.
	\begin{itemize}
		\item[$\mathrm{(i)}$] $\{x^{k}\}$ converges to some Pareto solution $x^{*}$.
		\item[$\mathrm{(ii)}$] $u_{0}^{\mu}(x^{k+1})\leq\left(1-2\gamma \sigma(1-\sigma)\left(\min\limits_{i\in[m]}\left\{\frac{\mu_{i}}{L_{i}}\right\}\right)^{3}\right)u_{0}^{\mu}(x^{k}).$
	\end{itemize}
\end{theorem}
\vspace{2mm}
\begin{proof}
	(i) Since $f_i$ is strongly convex and $g_{i}$ is convex, the level set $\{x : F(x) \preceq F(x^0)\}\subset\{x:F_{i}(x)\leq F_{i}(x^{0})\}$ is bounded, and any weak Pareto solution is a Pareto solution. Therefore, assertion (i) is the consequence of Theorem \ref{t3}(i).
	
	(ii) Since $f_{i}$ is $\mu_{i}$-strongly convex and $L_{i}$-smooth, we can establish the following bounds:
	$$\mu_{i}\|x^{k}-x^{k-1}\|^{2}\leq \left\langle\nabla f_{i}(x^{k})-\nabla f_{i}(x^{k-1}), x^{k}-x^{k-1}\right\rangle\leq L_{i}\|x^{k}-x^{k-1}\|^{2},\ \forall i\in[m].$$
	This, together with the facts that $\alpha_{\min}$ is a sufficient small positive constant and $\alpha_{\max}$ is a sufficient large positive constant, leads to
	\begin{equation}\label{ebound}
		\mu_{i}\leq\alpha^{k}_{i}\leq L_{i}, \ \forall i \in[m].
	\end{equation}
	Recall that the Armijo line search satisfies
	$$F_{i}(x^{k+1})-F_{i}(x^{k})\leq t_{k}\sigma(\left\langle\nabla f_{i}(x^{k}),d^{k}\right\rangle+g_{i}(x^{k}+d^{k})-g_{i}(x^{k})).$$
	A direct calculation gives
	\begin{equation}\label{E15}
		\begin{aligned}
			&~~~~\frac{F_{i}(x^{k+1})-F_{i}(x^{k})}{\alpha^{k}_{i}}\\
			&\leq t_{k}\sigma\left(\frac{\left\langle\nabla f_{i}(x^{k}),d^{k}\right\rangle+g_{i}(x^{k}+d^{k})-g_{i}(x^{k})}{\alpha^{k}_{i}}+\frac{1}{2}\|d^{k}\|^{2}\right)\\
			&\leq t_{k}\sigma\max\limits_{i\in[m]}\left\{\frac{\left\langle\nabla f_{i}(x^{k}),d^{k}\right\rangle+g_{i}(x^{k}+d^{k})-g_{i}(x^{k})}{\alpha^{k}_{i}}+\frac{1}{2}\|d^{k}\|^{2}\right\}\\
			&=-t_{k}\sigma w_{1}^{\alpha^{k}}(x^{k})\\
			&\leq -2\gamma\sigma(1-\sigma) \min\limits_{i\in[m]}\left\{\frac{\mu_{i}}{L_{i}}\right\}w_{1}^{\alpha^{k}}(x^{k}),
		\end{aligned}
	\end{equation}
	where the last inequality is is a consequence of (\ref{et}) and (\ref{ebound}).
	Furthermore, due to the strong convexity of $f_{i}$, we have
	$$\frac{F_{i}(x^{k})-F_{i}(x)}{\mu_{i}}\leq\frac{\left\langle\nabla f_{i}(x^{k}),x^{k}-x\right\rangle+g_{i}(x^{k})-g_{i}(x)}{\mu_{i}}-\frac{1}{2}\|x-x^{k}\|^{2},\ \forall x\in\mathbb{R}^{n}.$$
	Taking the supremum and minimum  with respect to $x\in\mathbb{R}^{n}$ and $i\in [m]$ on both sides, respectively, we obtain
	\begin{align*}
		&~~~~\sup\limits_{x\in\mathbb{R}^{n}}\min\limits_{i\in[m]}\left\{\frac{F_{i}(x^{k})-F_{i}(x)}{\mu_{i}}\right\}\\
		&\leq\sup\limits_{x\in\mathbb{R}^{n}}\min\limits_{i\in[m]}\left\{\frac{\left\langle\nabla f_{i}(x^{k}),x^{k}-x\right\rangle+g_{i}(x^{k})-g_{i}(x)}{\mu_{i}}-\frac{1}{2}\|x-x^{k}\|^{2}\right\}\\
		&\leq\left(\max\limits_{i\in[m]}\left\{\frac{\alpha^{k}_{i}}{\mu_{i}}\right\}\right)^{2}\sup\limits_{x\in\mathbb{R}^{n}}\min\limits_{i\in[m]}\left\{\frac{\left\langle\nabla f_{i}(x^{k}),x^{k}-x\right\rangle+g_{i}(x^{k})-g_{i}(x)}{\alpha^{k}_{i}}-\frac{1}{2}\|x-x^{k}\|^{2}\right\}\\
		&\leq\left(\max\limits_{i\in[m]}\left\{\frac{L_{i}}{\mu_{i}}\right\}\right)^{2}\sup\limits_{x\in\mathbb{R}^{n}}\min\limits_{i\in[m]}\left\{\frac{\left\langle\nabla f_{i}(x^{k}),x^{k}-x\right\rangle+g_{i}(x^{k})-g_{i}(x)}{\alpha^{k}_{i}}-\frac{1}{2}\|x-x^{k}\|^{2}\right\},
	\end{align*}
	where the second inequality is due to Proposition \ref{p2}(iii) and the fact that $\alpha^{k}_{i}\geq\mu_{i}$, the last inequality is given by $\alpha^{k}_{i}\leq L_{i}$. The above inequalities can be reformulated as
	$$u_{0}^{\mu}(x^{k})\leq\left(\max\limits_{i\in[m]}\left\{\frac{L_{i}}{\mu_{i}}\right\}\right)^{2}w_{1}^{\alpha^{k}}(x^{k}).$$
	Together with (\ref{E15}), the preceding inequality yields
	$$\frac{F_{i}(x^{k+1})-F_{i}(x^{k})}{\mu_{i}}\leq\frac{F_{i}(x^{k+1})-F_{i}(x^{k})}{\alpha^{k}_{i}}\leq-2\gamma \sigma(1-\sigma)\left(\min\limits_{i\in[m]}\left\{\frac{\mu_{i}}{L_{i}}\right\}\right)^{3} u_{0}^{\mu}(x^{k}).$$
	Then, for all $x\in\mathbb{R}^{n}$, we have
	$$\frac{F_{i}(x^{k+1})-F_{i}(x)}{\mu_{i}}\leq\frac{F_{i}(x^{k})-F_{i}(x)}{\mu_{i}} -2\gamma \sigma(1-\sigma)\left(\min\limits_{i\in[m]}\left\{\frac{\mu_{i}}{L_{i}}\right\}\right)^{3} u_{0}^{\mu}(x^{k}).$$
	Taking the supremum and minimum  with respect to $x\in\mathbb{R}^{n}$ and $i\in [m]$ on both sides, respectively, we obtain
	$$u_{0}^{\mu}(x^{k+1})\leq\left(1-2\gamma \sigma(1-\sigma)\left(\min\limits_{i\in[m]}\left\{\frac{\mu_{i}}{L_{i}}\right\}\right)^{3}\right)u_{0}^{\mu}(x^{k}).$$
	This completes the proof.
\end{proof}
\vspace{2mm}
%\begin{remark}
%	If $f_{i}$ is not ill-conditioned for $i\in[m]$, then the BBPGMO converges rapidly. 
%\end{remark}

Compared to the linear convergence of multiobjective proximal gradient method, the result in Theorem \ref{t4} is objective-independent but has a higher order coefficient. However, the conservative nature of the bound in Theorem \ref{t4} suggests that it can be improved in practice. In the following proposition, we present a better bound.
\vspace{2mm}
\begin{proposition}\label{p5}
	Under the assumptions of Theorem \ref{t4}, if $\{x^{k}\}$ is generated by Algorithm \ref{alg2} with $\alpha^{k}_{i}=\mu_{i},\ \forall i\in[m]$, then
	$$u_{0}^{\mu}(x^{k+1})\leq\left(1-2\gamma \sigma(1-\sigma)\min\limits_{i\in[m]}\left\{\frac{\mu_{i}}{L_{i}}\right\}\right)u_{0}^{\mu}(x^{k}).$$
\end{proposition}
\begin{proof}
	The proof follows a similar approach as in Theorem \ref{t4}, we omit it here.
\end{proof}
\vspace{2mm}
\begin{remark}\label{r4}
	For the strongly convex case, it is more reasonable to set $\alpha_i^k$ as given in (\ref{alpha_k}) rather than setting $\alpha_i^k = \mu_i$. Equation (\ref{alpha_k}) takes advantage of local information, whereas $\mu_i$ is too conservative as the curvature of $f_{i}$ can be quite different.
\end{remark}
\subsection{Convergence rates analysis of ABBPGMO}
In this subsection, we assume that $f_{i}$ is $L_{i}$-smooth for $i\in[m]$. We also analyze the convergence rates of adaptive BBPGMO.
\subsubsection{Global convergence}

\begin{theorem}\label{T4}
	Assume that $\Omega=\{x:F(x)\preceq F(x^{0})\}$ is a bounded set. Let $\{x^{k}\}$ be the sequence generated by Algorithm \ref{alg3}. Then, the following statements hold.
	\begin{itemize}
		\item[$\mathrm{(i)}$] $\{x_{k}\}$ has at least one accumulation point, and every accumulation point $x^{*}\in\Omega$ is a Pareto critical point.
		\item[$\mathrm{(ii)}$] $\min\limits_{0\leq s\leq k-1}\|x^{k+1}-x^{k}\|\leq\frac{\sqrt{2\min\limits_{i\in[m]}\{F_{i}(x^{0})-F_{i}^{*}\}}}
		{\sqrt{\alpha_{\min}}\sqrt{k}}$.
	\end{itemize}
\end{theorem}
\vspace{2mm}
\begin{proof}
	(i) Applying Theorem \ref{global1}, there exists $F^{*}$ such that $F^{*}\preceq F(x^{k})$ and $\lim\limits_{k\rightarrow\infty}F(x^{k})=F^{*}$. The assertion (i) can be obtained by using the same arguments as in the proof of \cite[Theorem 4.3]{TFY2019}.
	
	(ii) From (\ref{E6}) and the termination condition (line 15) in Algorithm \ref{alg3}, we have
	$$F_{i}(x^{k+1})-F_{i}(x^{k})\leq-\frac{\alpha^{k}_{i}}{2}\|x^{k+1}-x^{k}\|^{2}\leq- \frac{\alpha_{\min}}{2}\|x^{k+1}-x^{k}\|^{2},\ \forall i\in[m].$$
	Taking the sum of the above relation over $0,1,...,k-1$, we obtain
	\begin{align*}
		F_{i}(x^{k})-F_{i}(x^{0})\leq-\frac{\alpha_{\min}}{2}\sum\limits_{0\leq s\leq k-1}\|x^{s+1}-x^{s}\|^{2}, \ \forall i\in[m].
	\end{align*}
	Rearranging and using the fact that $F^{*}\leq F(x^{k})$, we have
	$$\frac{\alpha_{\min}}{2}\sum\limits_{0\leq s\leq k-1}\|x^{k+1}-x^{k}\|^{2}\leq F_{i}(x^{0})-F_{i}^{*}, \ \forall i\in[m].$$
	Thus, the desired result follows.
\end{proof}
\vspace{2mm}
\begin{corollary}
	Assume that $\Omega=\{x:F(x)\preceq F(x^{0})\}$ is a bounded set. Let $\{x^{k}\}$ be the sequence generated by Algorithm \ref{newalg}. Then, the following statements hold.
	\begin{itemize}
		\item[$\mathrm{(i)}$] $\{x_{k}\}$ has at least one accumulation point, and every accumulation point $x^{*}\in\Omega$ is a Pareto critical point.
		\item[$\mathrm{(ii)}$] $\min\limits_{0\leq s\leq k-1}\|x^{k+1}-x^{k}\|\leq\frac{\sqrt{2\min\limits_{i\in[m]}\frac{F_{i}(x^{0})-F_{i}^{*}}{L_{i}}}}
		{\sqrt{k}}$.
	\end{itemize}
\end{corollary}
\vspace{2mm}
\subsubsection{Strong convergence}
Before presenting the strong convergence, we establish a fundamental inequality.
\begin{lemma}\label{l5.9}
	Assume that $f_{i}$ is strongly convex with modulus $\mu_{i}\geq0,~i\in[m]$. Then, there exists $\lambda^{k}\in\Delta_{m}$ such that
	\begin{equation}\label{fineq}
		\begin{aligned}
			&~~~~\sum\limits_{i\in[m]}\lambda^{k}_{i}\frac{F_{i}(x^{k+1})-F_{i}(x)}{\alpha^{k}_{i}}\\
			&\leq\frac{1}{2}\|x^{k}-x\|^{2}-\frac{1}{2}\|x^{k+1}-x\|^{2}-\sum\limits_{i\in[m]}\lambda^{k}_{i}\frac{\mu_{i}}{2\alpha^{k}_{i}}\|x^{k}-x\|^{2},~\forall x\in\mathbb{R}^{n}.
		\end{aligned}
	\end{equation}
\end{lemma}
\begin{proof}
	Using the termination condition (line 15) in Algorithm \ref{alg3}, we have
	\begin{align*}
		&~~~~F_{i}(x^{k+1})-F_{i}(x^{k})\\
		&\leq \left\langle\nabla f_{i}(x^{k}),x^{k+1}-x^{k}\right\rangle + g_{i}(x^{k+1})-g_{i}(x^{k})+\frac{\alpha^{k}_{i}}{2}\|x^{k+1}-x^{k}\|^{2},~i\in[m].
	\end{align*}
	This, together with the $\mu_{i}$-strong convexity of $f_{i}$, yields
	\begin{align*}
		&~~~~F_{i}(x^{k+1}) - F_{i}(x)\\
		&=(F_{i}(x^{k+1})-F_{i}(x^{k})) + (F_{i}(x^{k})-F_{i}(x))\\
		&\leq \left(\left\langle\nabla f_{i}(x^{k}),x^{k+1}-x^{k}\right\rangle + g_{i}(x^{k+1})-g_{i}(x^{k})+\frac{\alpha^{k}_{i}}{2}\|x^{k+1}-x^{k}\|^{2}\right)\\
		&~~~~+\left(\left\langle\nabla f_{i}(x^{k}),x^{k}-x\right\rangle + g_{i}(x^{k})-g_{i}(x)-\frac{\mu_{i}}{2}\|x^{k}-x\|^{2}\right)\\
		&=\left\langle\nabla f_{i}(x^{k}),x^{k+1}-x\right\rangle + g_{i}(x^{k+1})-g_{i}(x)+\frac{\alpha^{k}_{i}}{2}\|x^{k+1}-x^{k}\|^{2}-\frac{\mu_{i}}{2}\|x^{k}-x\|^{2},
	\end{align*}
	for all $x\in\mathbb{R}^{n}.$
	On the other hand, using \cite[Theorem 6.39(iii)]{B2017} and (\ref{prox}), we have
	$$\sum\limits_{i\in[m]}\lambda^{k}_{i}\frac{g_{i}(x^{k+1})-g_{i}(x)}{\alpha^{k}_{i}}\leq\left\langle x^{k}-\sum\limits_{i\in[m]}\lambda^{k}_{i}\frac{\nabla f_{i}(x^{k})}{\alpha^{k}_{i}}-x^{k+1},x^{k+1}-x\right\rangle.$$	
	Then we use the last two inequalities to get
	\begin{align*}
		&~~~~\sum\limits_{i\in[m]}\lambda^{k}_{i}\frac{F_{i}(x^{k+1})-F_{i}(x)}{\alpha^{k}_{i}}\\
		&\leq \dual{\sum\limits_{i\in[m]}\lambda^{k}_{i}\frac{\nabla f_{i}(x^{k})}{\alpha^{k}_{i}},x^{k+1}-x}+\sum\limits_{i\in[m]}\lambda^{k}_{i}\frac{g_{i}(x^{k+1})-g_{i}(x)}{\alpha^{k}_{i}}+\frac{1}{2}\|x^{k+1}-x^{k}\|^{2}\\
		&~~~~-\sum\limits_{i\in[m]}\lambda^{k}_{i}\frac{\mu_{i}}{2\alpha^{k}_{i}}\|x^{k}-x\|^{2}\\
		&\leq \dual{x^{k}-x^{k+1},x^{k+1}-x}+\frac{1}{2}\|x^{k+1}-x^{k}\|^{2}-\sum\limits_{i\in[m]}\lambda^{k}_{i}\frac{\mu_{i}}{2\alpha^{k}_{i}}\|x^{k}-x\|^{2}\\
		&=\left(\frac{1}{2}\|x^{k}-x\|^{2}-\frac{1}{2}\|x^{k+1}-x\|^{2}-\frac{1}{2}\|x^{k+1}-x^{k}\|^{2}\right)+\frac{1}{2}\|x^{k+1}-x^{k}\|^{2}\\
		&~~~~-\sum\limits_{i\in[m]}\lambda^{k}_{i}\frac{\mu_{i}}{2\alpha^{k}_{i}}\|x^{k}-x\|^{2}\\
		&=\frac{1}{2}\|x^{k}-x\|^{2}-\frac{1}{2}\|x^{k+1}-x\|^{2}-\sum\limits_{i\in[m]}\lambda^{k}_{i}\frac{\mu_{i}}{2\alpha^{k}_{i}}\|x^{k}-x\|^{2},
	\end{align*} 
	for all $x\in\mathbb{R}^{n}.$		
\end{proof}
\vspace{2mm}
\par We are now in the position to prove the strong convergence of Algorithm \ref{alg3}.
\vspace{2mm}
\begin{theorem}\label{T5}
	Assume that $\Omega=\{x:F(x)\preceq F(x^{0})\}$ is a bounded set and $f_{i}$ is convex, $i\in[m]$. Let $\{x^{k}\}$ be the sequence generated by Algorithm \ref{alg3}. Then, the following statements hold.
	\begin{itemize}
		\item[$\mathrm{(i)}$] $\{x^{k}\}$ converges to some weak Pareto solution $x^{*}$.
		\item[$\mathrm{(ii)}$] $u_{0}^{\tau L}(x^{k})\leq\frac{R^{2}}{2k},\ \forall k,$ where $R:=\max\{\|x-y\|:x,y\in\Omega\}.$
	\end{itemize}
\end{theorem}
\vspace{2mm}
\begin{proof}
	(i) From Theorem \ref{T4}(i), there exists a Pareto critical point $x^{*}$ such that
	$F(x^{*})\preceq F(x^{k})$, and $x^{*}$ is an accumulation point of $\{x^{k}\}$. Moreover, $x^{*}$ is a weak Pareto solution due to the convexity of $f_{i}$.
	Applying fundamental inequality (\ref{fineq}) and the convexity of $f_{i},~i\in[m]$,  for all $x\in\mathbb{R}^{n}$ we have
	\begin{equation}\label{E25}
		\sum\limits_{i\in[m]}\lambda^{k}_{i}\frac{F_{i}(x^{k+1})-F_{i}(x)}{\alpha^{k}_{i}}\leq\frac{1}{2}\|x^{k}-x\|^{2}-\frac{1}{2}\|x^{k+1}-x\|^{2}.
	\end{equation}
	Substituting $x=x^{*}$ into the above inequality, we obtain
	$$\sum\limits_{i\in[m]}\lambda^{k}_{i}\frac{F_{i}(x^{k+1})-F_{i}(x^{*})}{\alpha^{k}_{i}}\leq\frac{1}{2}\|x^{k}-x^{*}\|^{2}-\frac{1}{2}\|x^{k+1}-x^{*}\|^{2}.$$
	Note that $F(x^{*})\preceq F(x^{k})$, it follows that
	$$\|x^{k+1}-x^{*}\|^{2}\leq\|x^{k}-x^{*}\|^{2}.$$
	Therefore, the sequence $\{\|x^{k}-x^{*}\|\}$ converges. This, together with the fact that $x^{*}$ is an accumulation point of $\{x^{k}\}$, implies that $\{x^{k}\}$ converges to $x^{*}$. 
	\par (ii) Taking the sum of (\ref{E25}) over $0$ to $k-1$, we obtain
	$$\sum\limits_{s=0}^{k-1}\sum\limits_{i\in[m]}\lambda^{s}_{i}\frac{F_{i}(x^{s+1})-F_{i}(x)}{\alpha^{s}_{i}}\leq\frac{1}{2}\|x^{0}-x\|^{2}-\frac{1}{2}\|x^{k}-x\|^{2}\leq\frac{1}{2}\|x^{0}-x\|^{2},$$
	for all $x\in\mathbb{R}^{n}$. Since $F(x^{k+1})\preceq F(x^{s+1})$ for all $s\leq k-1$, it leads to
	$$\sum\limits_{s=0}^{k-1}\sum\limits_{i\in[m]}\lambda^{s}_{i}\frac{F_{i}(x^{k})-F_{i}(x)}{\alpha^{s}_{i}}\leq\frac{1}{2}\|x^{0}-x\|^{2}.$$
	For all $F(x)\preceq F(x^{k})$, together with the fact that $\alpha^{k}_{i}< \tau L_{i}$, the preceding relation yields
	$$\sum\limits_{s=0}^{k-1}\sum\limits_{i\in[m]}\lambda^{s}_{i}\frac{F_{i}(x^{k})-F_{i}(x)}{\tau L_{i}}\leq\frac{1}{2}\|x^{0}-x\|^{2}.$$
	Denote $\bar{\lambda}^{k}_{i}:=\sum\limits_{s=0}^{k-1}\lambda^{s}_{i}/k$, we can deduce that $\bar{\lambda}^{k}\in\Delta_{m}$ and
	$$\sum\limits_{i\in[m]}\bar{\lambda}^{k}_{i}\frac{F_{i}(x^{k})-F_{i}(x)}{\tau L_{i}}\leq\frac{\|x^{0}-x\|^{2}}{2k}.$$
	Select $z^{k}\in\mathop{\arg\max}\limits_{x\in\mathbb{R}^{n}}\min\limits_{i\in[m]}\left\{\frac{F_{i}(x^{k})-F_{i}(x)}{\tau L_{i}}\right\}$, it holds that
	\begin{align*}
		u_{0}^{\tau L}(x^{k})&=\max\limits_{x\in\mathbb{R}^{n}}\min\limits_{i\in[m]}\left\{\frac{F_{i}(x^{k})-F_{i}(x)}{\tau L_{i}}\right\}\\
		&=\min\limits_{i\in[m]}\left\{\frac{F_{i}(x^{k})-F_{i}(z^{k})}{\tau L_{i}}\right\}\\
		&\leq\sum\limits_{i\in[m]}\bar{\lambda}^{k}_{i}\frac{F_{i}(x^{k})-F_{i}(z^{k})}{\tau L_{i}}\\
		&\leq\frac{\|x^{0}-z^{k}\|^{2}}{2k}.
	\end{align*}
	By the definition of $z^{k}$, we deduce that $z^{k}\in\{x:F(x)\preceq F(x^{k})\}\subset\Omega$, which implies $\|x^{0}-z^{k}\|\leq R$, the desired result follows.
\end{proof}

\begin{corollary}
	Assume that $\Omega=\{x:F(x)\preceq F(x^{0})\}$ is a bounded set and $f_{i}$ is convex, $i\in[m]$. Let $\{x^{k}\}$ be the sequence generated by Algorithm \ref{newalg}. Then, the following statements hold.
	\begin{itemize}
		\item[$\mathrm{(i)}$] $\{x^{k}\}$ converges to some weak Pareto solution $x^{*}$.
		\item[$\mathrm{(ii)}$] $u_{0}^{L}(x^{k})\leq\frac{R^{2}}{2k},\ \forall k,$ where $R:=\max\{\|x-y\|:x,y\in\Omega\}.$
	\end{itemize}
\end{corollary}
\vspace{2mm}
\subsubsection{Linear convergence}
\begin{theorem}\label{T6}
	Assume that $f_{i}$ is strongly convex with modulus $\mu_{i}>0$, $i\in[m]$. Let $\{x^{k}\}$ be the sequence generated by Algorithm \ref{alg3}. Then, the following statements hold.
	\begin{itemize}
		\item[$\mathrm{(i)}$] $\{x^{k}\}$ converges to some Pareto solution $x^{*}$.
		\item[$\mathrm{(ii)}$] $\|x^{k+1}-x^{*}\|\leq\sqrt{1-\min\limits_{i\in[m]}\left\{\frac{\mu_{i}}{\tau L_{i}}\right\}}\|x^{k}-x^{*}\|.$
	\end{itemize}
\end{theorem}
\vspace{2mm}
\begin{proof}
	(i) Since $f_{i}$ is strongly convex and $g_{i}$ is convex, then the level set $\{x:F(x)\preceq F(x^{0})\}\subset\{x:F_{i}(x)\leq F_{i}(x^{0})\}$ is bounded and any weak Pareto solution is Pareto solution. Therefore, assertion (i) is a consequence of Theorem \ref{T5}(i). 
	\par (ii) By substituting $x=x^{*}$ into inequality (\ref{fineq}), we obtain
	\begin{align*}
		&~~~~\sum\limits_{i\in[m]}\lambda^{k}_{i}\frac{F_{i}(x^{k+1})-F_{i}(x^{*})}{\alpha^{k}_{i}}\\
		&\leq\frac{1}{2}\|x^{k}-x^{*}\|^{2}-\frac{1}{2}\|x^{k+1}-x^{*}\|^{2}-\sum\limits_{i\in[m]}\lambda^{k}_{i}\frac{\mu_{i}}{2\alpha^{k}_{i}}\|x^{k}-x^{*}\|^{2},~\forall x\in\mathbb{R}^{n}.
	\end{align*}
	Applying $F(x^{*})\preceq F(x^{k})$, it follows that
	\begin{equation}\label{nsc}
		\|x^{k+1}-x^{*}\|\leq\sqrt{1-\sum\limits_{i\in[m]}\lambda^{k}_{i}\frac{\mu_{i}}{\alpha^{k}_{i}}}\|x^{k}-x^{*}\|\leq\sqrt{1-\min\limits_{i\in[m]}\left\{\frac{\mu_{i}}{\tau L_{i}}\right\}}\|x^{k}-x^{*}\|,
	\end{equation}
	where the last inequality holds due to the facts $\lambda^{k}\in\Delta_{m}$ and $\alpha^{k}_{i}<\tau L_{i},~i\in[m]$.
\end{proof}
\vspace{2mm}
\begin{corollary}\label{c3}
	Assume that $f_{i}$ is strongly convex with modulus $\mu_{i}>0$, $i\in[m]$. Let $\{x^{k}\}$ be the sequence generated by Algorithm \ref{newalg}. Then, the following statements hold.
	\begin{itemize}
		\item[$\mathrm{(i)}$] $\{x^{k}\}$ converges to some Pareto solution $x^{*}$.
		\item[$\mathrm{(ii)}$] $\|x^{k+1}-x^{*}\|\leq\sqrt{1-\min\limits_{i\in[m]}\left\{\frac{\mu_{i}}{ L_{i}}\right\}}\|x^{k}-x^{*}\|.$
	\end{itemize}
\end{corollary}
\vspace{2mm}
\par In the following, we give the relationships among the new multiobjective proximal gradient method, multiobjective proximal gradient method and proximal gradient method for SOPs.
\vspace{2mm}
\begin{remark}
	When $m=1$, both the new multiobjective proximal gradient method and the multiobjective proximal gradient method collapse to the proximal gradient method for SOPs. Notably, when $m\neq1$, the new multiobjective proximal gradient method is faster than the multiobjective proximal gradient method and has at least the same rate of convergence as the proximal gradient method for the $F_i$ with the smallest $\frac{\mu_i}{L_i}$, $i\in[m]$. On the other hand, the new multiobjective proximal gradient method exhibits rapid linear convergence as long as all differentiable components are not ill-conditioned. However, the multiobjective proximal gradient method may converge slowly even if all differentiable components are not ill-conditioned. 
\end{remark}
\vspace{2mm}
\begin{remark}\label{r8}
	Based on Proposition \ref{p5} and Corollary \ref{c3}, we observe that by setting $\alpha^{k}_{i}=\mu_{i}$ and $\alpha^{k}_{i}=L_{i}$ in Algorithm \ref{alg2}, the rates of linear convergence in terms of $\{u_{0}^{\mu}(x^{k+1})\}$ and $\{\|x^{k}-x^{*}\|^{2}\}$ are $1-2\gamma \sigma(1-\sigma)\min\limits_{i\in[m]}\left\{\frac{\mu_{i}}{L_{i}}\right\}$ and $1-\min\limits_{i\in[m]}\left\{\frac{\mu_{i}}{L_{i}}\right\}$, respectively. On the other hand, if we set $\alpha^{i}_{k}$ as (\ref{alpha_k}), we have $\mu_{i}\leq\alpha^{k}_{i}\leq L_{i}$. Intuitively, we believe that the rate of linear convergence for Algorithm \ref{alg2} is not worse than $1-2\gamma \sigma(1-\sigma)\min\limits_{i\in[m]}\left\{\frac{\mu_{i}}{L_{i}}\right\}$, since $2\gamma \sigma(1-\sigma)<1$.
\end{remark}
\subsection{Linear convergence with some linear objective functions}
In view of the first inequality in (\ref{nsc}), the rate of linear convergence in terms of $\|x^{k}-x^{*}\|$ is actually $\sqrt{1-\sum\limits_{i\in[m]}\lambda^{k}_{i}\frac{\mu_{i}}{\tau L_{i}}}$ for Algorithm \ref{newalg}. On the other hand, Lemma \ref{l5.9} is valid for weakly convex function $f_{i}$, i.e., $f_{i}(y)-f_{i}(x)\geq \dual{\nabla f_{i}(x),y-x}+\frac{\mu_{i}}{2}\nm{y-x}^{2},~\mu_{i}<0$. Consequently, Algorithm \ref{newalg} exhibits linear convergence rate as long as $\sum\limits_{i\in[m]}\lambda^{k}_{i}\frac{\mu_{i}}{\tau L_{i}}>0$. A similar statement holds for (new) PGMO without line search. However, such a statement holds under the restrictive condition $\sum\limits_{i\in[m]}\lambda^{k}_{i}\frac{\mu_{i}}{\tau L_{i}}>0$,  and there may exist counterexamples that show these algorithms do not converge linearly. Naturally, a question arises: Does ABBPGMO or PGMO without line search have a linear convergence rate without strong convexity assumption on some objective functions?
\par To the best of our knowledge,  in SOPs, this question has been addressed through the study of error bound conditions, which have been extensively explored (see, e.g., \cite{LT1993,NNG2019,Z2020} and references therein). However, in the context of MOPs, this area has received little attention \cite{TFY2020}. In the following, we do not delve into the study of error bound conditions for MOPs. Instead, we focus on linear constrained MOP with linear objective functions, which is described as follows:

\begin{align*}
	\min\limits_{x\in \mathcal{X}} f(x),\tag{LCMOP}\label{LCMOP}
\end{align*}
where $f:\mathbb{R}^{n}\mapsto\mathbb{R}^{m}$ is a vector-valued function; the component $f_{i}$ is linear for $i\in\mathcal{L}$, and $\mu_{i}$-strongly convex and $L_{i}$-smooth for $i\in[m]\setminus\mathcal{L}$, respectively; $\mathcal{X}=\{x:Ax\leq a,~Bx=b\}$ with $A\in\mathbb{R}^{\mathbb{|\mathcal{J}|}\times n},~B\in\mathbb{R}^{\mathbb{|\mathcal{E}|}\times n}$. This type of problem has wide applications in portfolio selection \cite{M1952}, and can be reformulated as (\ref{MCOP}) with $F_{i}=f_{i}+\mathbb{I}_{\mathcal{X}}~i\in[m]$, where 
\begin{equation*}
	\mathbb{I}_{\mathcal{X}}(x):=\left\{
	\begin{aligned}
		0,~~~~~~&x\in\mathcal{X},\\
		+\infty,~~~&x\notin\mathcal{X}.
	\end{aligned}
	\right.
\end{equation*}
When minimizing (\ref{LCMOP}) using ABBPGMO, the subproblem (\ref{dk}) is reformulated as follows:

\begin{align*}
	\min&~~~~t +\frac{1}{2}\nm{d}^{2}         \\
	{\rm s.t.}&\dual{\frac{\nabla f_{i}(x^{k})}{\alpha^{k}_{i}},d}\leq t,~i\in[m],\\
	&A(x^{k}+d)\leq a,\\
	&Bd=0.
\end{align*}
By KKT conditions, we obtain $d^{k}=-\left(\sum\limits_{i\in[m]}\lambda^{k}_{i}\frac{\nabla f_{i}(x^{k})}{\alpha^{k}_{i}}+\sum\limits_{j\in\mathcal{J}}\theta^{k}_{j}A_{j}+\sum\limits_{e\in\mathcal{E}}\xi^{k}_{e}B_{e}\right)$, where $A_{j}^{T}$ and $B_{e}^{T}$ is the $j$-th and $e$-th row of $A$ and $B$, respectively. The vector $(\lambda^{k},\theta^{k},\xi^{k})\in\mathbb{R}^{m+|\mathcal{J}|+|\mathcal{E}|}$ is a solution of the following Lagrangian dual problem:

\begin{align*}
	\min&~~~~\frac{1}{2}\nm{\sum\limits_{i\in[m]}\lambda_{i}\frac{\nabla f_{i}(x^{k})}{\alpha^{k}_{i}}+\sum\limits_{j\in\mathcal{J}}\theta_{j}A_{j}+\sum\limits_{e\in\mathcal{E}}\xi_{e}B_{e}}^{2} - \sum\limits_{j\in\mathcal{J}}\theta_{j}\left(\dual{A_{j},x^{k}}-a_{j}\right)     \\
	{\rm s.t.}&~~~~\lambda\in\Delta_{m},\\
	&~~~~\theta\in\mathbb{R}^{|\mathcal{J}|}_{+},\\
	&~~~~\xi\in\mathbb{R}^{|\mathcal{E}|}.
\end{align*}
And complementary slackness condition gives that
\begin{equation}\label{cse}
	\theta^{k}_{j}\left(\dual{A_{j},x^{k}+d^{k}}-a_{j}\right)=0,~\forall j\in\mathcal{J}.
\end{equation}
Denote $\mathcal{J}^{k}:=\{j\in\mathcal{J}:\dual{A_{j},x^{k}}=a_{j}\}$, this together with (\ref{cse}) and the fact that $x^{k+1}=x^{k}+d^{k}$ implies 
$$d^{k}=-\left(\sum\limits_{i\in[m]}\lambda^{k}_{i}\frac{\nabla f_{i}(x^{k})}{\alpha^{k}_{i}}+\sum\limits_{j\in\mathcal{J}^{k+1}}\theta^{k}_{j}A_{j}+\sum\limits_{e\in\mathcal{E}}\xi^{k}_{e}B_{e}\right).$$
Before presenting the linear convergence of ABBPGMO for (\ref{LCMOP}),  let's first define the following multiobjective linear programming problem:
\begin{align*}
	\min\limits_{x\in \mathcal{X}} f_{\mathcal{L}}(x).\tag{MLP}\label{MLP}
\end{align*}
It is important to note that every weakly Pareto solution of (\ref{MLP}) is also a weakly Pareto solution of (\ref{LCMOP}). 
\vspace{2mm}
\begin{proposition}\label{pl1}
	Denote $\Omega=\{x:f(x)\preceq f(x^{0})\}$. Let $\{x^{k}\}$ be the sequence generated by Algorithm \ref{alg3} for (\ref{LCMOP}). Then, the following statements hold.
	\begin{itemize}
		\item[$\mathrm{(i)}$] $\{x^{k}\}$ converges to some weakly Pareto solution $x^{*}$.
		\item[$\mathrm{(ii)}$] If $x^{*}$ is not a weakly Pareto solution of (\ref{MLP}), then $$\|x^{k+1}-x^{*}\|\leq\sqrt{1-(1-c_{1})\min\limits_{i\in[m]\setminus\mathcal{L}}\left\{\frac{\mu_{i}}{\tau L_{i}}\right\}}\|x^{k}-x^{*}\|,$$
		where $c_{1}$ is a positive constant proportional to $\alpha_{\min}$.
	\end{itemize}
\end{proposition}
\vspace{2mm}
\begin{proof}
	(i) Since $f_{i}$ is strongly convex for $i\in[m]\setminus\mathcal{L}$, the level set $\Omega\subset\{x:f_{i}(x)\leq f_{i}(x^{0}),~i\in[m]\setminus\mathcal{L}\}$ is bounded. Then assertion (i) is a consequence of Theorem \ref{T5}(i).
	\par(ii) We refer to Theorem \ref{T6}(ii), which states:
	\begin{equation}\label{lce}
		\|x^{k+1}-x^{*}\|\leq\sqrt{1-\sum\limits_{i\in[m]}\lambda^{k}_{i}\frac{\mu_{i}}{\alpha^{k}_{i}}}\|x^{k}-x^{*}\|.
	\end{equation} 
	Given that $x^{*}$ is not a weakly Pareto solution of (\ref{MLP}), we deduce $0\notin C^{k}:=\{\sum\limits_{i\in\mathcal{L}}\lambda_{i} g_{i}+\sum\limits_{j\in\mathcal{J}^{k}}\theta_{j}A_{j}+\sum\limits_{e\in\mathcal{E}}\xi_{e}B_{e}:\lambda\in\Delta_{|\mathcal{L}|},\theta\in\mathbb{R}^{|\mathcal{J}^{k}|}_{+},\xi\in\mathbb{R}^{|\mathcal{E}|}\}$, where $g_{i}$ represents the gradient of the linear function $f_{i},~i\in\mathcal{L}$. Furthermore, since  $\mathcal{J}^{k}\subset\mathcal{J}$, we can infer that $|\{C^{k}\}|$ is finite. Consequently, we can define $\epsilon:=\min\limits_{x\in \{C^{k}\}}\|x\|>0$. By direct calculation, we have
	\begin{equation}\label{e5.13}
		\begin{aligned}
			\|d^{k}\|&=\nm{\sum\limits_{i\in[m]}\lambda^{k}_{i}\frac{\nabla f_{i}(x^{k})}{\alpha^{k}_{i}}+\sum\limits_{j\in\mathcal{J}^{k+1}}\theta^{k}_{j}A_{j}+\sum\limits_{e\in\mathcal{E}}\xi^{k}_{e}B_{e}}\\
			&\geq\nm{\sum\limits_{i\in\mathcal{L}}\lambda^{k}_{i}\frac{g_{i}}{\alpha^{k}_{i}}+\sum\limits_{j\in\mathcal{J}^{k+1}}\theta^{k}_{j}A_{j}+\sum\limits_{e\in\mathcal{E}}\xi^{k}_{e}B_{e}}-\nm{\sum\limits_{i\in[m]\setminus\mathcal{L}}\lambda^{k}_{i}\frac{\nabla f_{i}(x^{k})}{\alpha^{k}_{i}}}\\
			&=\frac{\sum\limits_{i\in\mathcal{L}}\lambda^{k}_{i}}{\alpha_{\min}}\nm{\sum\limits_{i\in\mathcal{L}}\frac{\lambda^{k}_{i}}{\sum\limits_{i\in\mathcal{L}}\lambda^{k}_{i}}g_{i}+\sum\limits_{j\in\mathcal{J}^{k+1}}\frac{\alpha_{\min}\theta^{k}_{j}}{\sum\limits_{i\in\mathcal{L}}\lambda^{k}_{i}}A_{j}+\sum\limits_{e\in\mathcal{E}}\frac{\alpha_{\min}\xi^{k}_{e}}{\sum\limits_{i\in\mathcal{L}}\lambda^{k}_{i}}B_{e}}\\
			&~~~~-\nm{\sum\limits_{i\in[m]\setminus\mathcal{L}}\lambda^{k}_{i}\frac{\nabla f_{i}(x^{k})}{\alpha^{k}_{i}}}\\
			&\geq\frac{\epsilon\sum\limits_{i\in\mathcal{L}}\lambda^{k}_{i}}{\alpha_{\min}} -\nm{\sum\limits_{i\in[m]\setminus\mathcal{L}}\lambda^{k}_{i}\frac{\nabla f_{i}(x^{k})}{\alpha^{k}_{i}}},
		\end{aligned}
	\end{equation}
	where the second equality follows from the fact that $\alpha^{k}_{i}=\alpha_{\min}$ for all $i\in\mathcal{L}$, and the last inequality is given by the definition of $\epsilon$.
	By simple calculation, we have 
	\begin{align*}
		\frac{1}{2}\|d^{k}\|^{2}&=\frac{1}{2}\nm{\sum\limits_{i\in[m]}\lambda^{k}_{i}\frac{\nabla f_{i}(x^{k})}{\alpha^{k}_{i}}+\sum\limits_{j\in\mathcal{J}}\theta^{k}_{j}A_{j}+\sum\limits_{e\in\mathcal{E}}\xi^{k}_{e}B_{e}}^{2}\\
		&\leq\frac{1}{2}\nm{\sum\limits_{i\in[m]}\lambda^{k}_{i}\frac{\nabla f_{i}(x^{k})}{\alpha^{k}_{i}}+\sum\limits_{j\in\mathcal{J}}\theta^{k}_{j}A_{j}+\sum\limits_{e\in\mathcal{E}}\xi^{k}_{e}B_{e}}^{2}- \sum\limits_{j\in\mathcal{J}}\theta^{k}_{j}\left(\dual{A_{j},x^{k}}-a_{j}\right) \\
		&\leq\frac{1}{2}\max\limits_{i\in[m]\setminus\mathcal{L}}\nm{\frac{\nabla f_{i}(x^{k})}{\mu_{i}}}^{2},
	\end{align*}
	where the last inequality is due to $\mu_{i}$-strong convexity of $f_{i},i\in[m]\setminus\mathcal{L}$, and the fact that $(\lambda^{k},\theta^{k},\xi^{k})$ is a solution of dual problem.
	This together with (\ref{e5.13}) implies
	$$\frac{\epsilon\sum\limits_{i\in\mathcal{L}}\lambda^{k}_{i}}{\alpha_{\min}}\leq\max\limits_{i\in[m]\setminus\mathcal{L}}\nm{\frac{\nabla f_{i}(x^{k})}{\mu_{i}}}+\nm{\sum\limits_{i\in[m]\setminus\mathcal{L}}\lambda^{k}_{i}\frac{\nabla f_{i}(x^{k})}{\alpha^{k}_{i}}}\leq2\max\limits_{i\in[m]\setminus\mathcal{L}}\nm{\frac{\nabla f_{i}(x^{k})}{\mu_{i}}}.$$
	Denoting $R:=\max\{\nm{x-y}:x,y\in\Omega\}$, and utilizing the $L_{i}$-smoothness of $f_{i},i\in[m]\setminus\mathcal{L}$, we derive an upper bound of $\|\nabla f_{i}(x^{k})\|$:
	\begin{align*}
		\|\nabla f_{i}(x^{k})\|&\leq\|\nabla f_{i}(x^{0})\|+\|\nabla f_{i}(x^{k})-\nabla f_{i}(x^{0})\|\\
		&\leq \|\nabla f_{i}(x^{0})\| + L_{i}R.
	\end{align*}
	Therefore, we obtain
	\begin{equation}\label{dse}
		\sum\limits_{i\in\mathcal{L}}\lambda^{k}_{i}\leq\max\limits_{i\in[m]\setminus\mathcal{L}}\frac{2(\|\nabla f_{i}(x^{0})\| + L_{i}R)}{\epsilon\mu_{i}}\alpha_{\min}.
	\end{equation}
	Substituting the above bound and $\alpha^{k}_{i}\leq\tau L_{i}$ into (\ref{lce}), we have the desired result.
\end{proof}
\vspace{2mm}
\par As mentioned in \cite{CTY2023}, linear objectives can often introduce significant imbalances in MOPs, which decelerates the convergence of SDMO. In the context of BBDMO, Chen et al. provided two examples \cite[Examples 2,3]{CTY2023} to illustrate why a small value of $\alpha_{\min}$ is sufficient to mitigate the influence of the linear objectives in direction-finding subproblems. In what follows, we attempt to confirm the statement from a theoretical perspective.
\vspace{2mm}
\begin{remark}
	From the proof of Proposition \ref{pl1}(ii), it is evident that by choosing a sufficiently small value of $\alpha_{\min}$, the sum of dual variables of liner objectives $\sum\limits_{i\in\mathcal{L}}\lambda^{k}_{i}$ tends to $0$. This effectively mitigates the influence of the linear objectives in direction-finding subproblems. Furthermore, since $c_{1}$ is proportional to $\alpha_{\min}$, the linear convergence rate can also be improved by selecting a sufficiently small value for $\alpha_{\min}$. Overall, this analysis supports the idea that appropriately choosing a small $\alpha_{\min}$ is beneficial in dealing with the impact of linear objectives and improving the convergence rate of the algorithm.
\end{remark}
\vspace{2mm}
\begin{remark}
	The assumption of Proposition \ref{pl1}(ii) seems restrictive.  In practice, the Pareto set of (\ref{LCMOP}) can be a $(m-1)$-dimensional manifold, and the Pareto set of (\ref{MLP}) can be a $(|\mathcal{L}|-1)$-dimensional sub-manifold within the $(m-1)$-dimensional manifold. As a result, for a random initial point $x^{0}$, the probability that $x^{*}$ is not a weakly Pareto solution of (\ref{MLP}) can be $1$. Additionally, when $|\mathcal{J}|=0$, meaning (\ref{MLP}) has no inequality constraints, then the problem can either have no weakly Pareto solution, or every feasible point can be considered a weakly Pareto solution. In other words, ABBPGMO converges linearly for equality constrained (\ref{LCMOP}). 
\end{remark}
\vspace{2mm}
\par Next, we analyze convergence rate of PGMO without line search for (\ref{LCMOP}).
\vspace{2mm}
\begin{proposition}\label{pl2}
	Let $\{x^{k}\}$ be the sequence generated by PGMO without line search for (\ref{LCMOP}). Then, the following statements hold.
	\begin{itemize}
		\item[$\mathrm{(i)}$] $\{x^{k}\}$ converges to some weakly Pareto solution $x^{*}$.
		\item[$\mathrm{(ii)}$] If $x^{*}$ is not a weakly Pareto solution of (\ref{MLP}), then there exists $\kappa\in(0,1)$ such that $\|x^{k}-x^{*}\|\leq\kappa^{k}\|x^{0}-x^{*}\|.$
	\end{itemize}
\end{proposition}
\vspace{2mm}
\begin{proof}
	(i) The proof follows a similar approach as in Proposition \ref{pl1}(i).
	\par (ii) Setting $\alpha^{k}_{i}=L_{\max}$ in ABBPGMO, it coincides with PGMO without line search. Consequently, (\ref{lce}) collapses to:
	\begin{small}
		\begin{equation}\label{e5.15}
			\|x^{k+1}-x^{*}\|\leq\sqrt{1-\sum\limits_{i\in[m]}\lambda^{k}_{i}\frac{\mu_{i}}{L_{\max}}}\|x^{k}-x^{*}\|.
		\end{equation}
	\end{small}
	Substituting $\alpha^{k}_{i}=L_{\max}$ into (\ref{e5.13}), it follows that
	\begin{small}
		\begin{equation*}
			\begin{aligned}
				\epsilon\sum\limits_{i\in\mathcal{L}}\lambda^{k}_{i}&\leq\nm{\sum\limits_{i\in[m]\setminus\mathcal{L}}\lambda^{k}_{i}\nabla f_{i}(x^{k})}+L_{\max}\|d^{k}\|\\
				&\leq\left(\sum\limits_{i\in[m]\setminus\mathcal{L}}\lambda^{k}_{i}\right)\max\limits_{i\in[m]\setminus\mathcal{L}}\left\{\|\nabla f_{i}(x^{0})\| + L_{i}R\right\}+L_{\max}\|d^{k}\|\\
				&=\left(1-\sum\limits_{i\in\mathcal{L}}\lambda^{k}_{i}\right)\max\limits_{i\in[m]\setminus\mathcal{L}}\left\{\|\nabla f_{i}(x^{0})\| + L_{i}R\right\}+L_{\max}\|d^{k}\|
			\end{aligned}
		\end{equation*}
	\end{small}
	Rearranging the above inequality, we have
	$$\sum\limits_{i\in\mathcal{L}}\lambda^{k}_{i}\leq\frac{\max\limits_{i\in[m]\setminus\mathcal{L}}\left\{\|\nabla f_{i}(x^{0})\| + L_{i}R\right\}+L_{\max}\nm{d^{k}}}{\max\limits_{i\in[m]\setminus\mathcal{L}}\left\{\|\nabla f_{i}(x^{0})\| + L_{i}R\right\}+\epsilon}.$$
	Since $\{x^{k}\}$ converges to a weakly Pareto point, it follows that $d^{k}\rightarrow0$. Then there exists $K>0$ such that $\|d^{k}\|\leq\frac{\epsilon}{2L_{\max}}, ~\forall k\geq K$. This implies that $$\sum\limits_{i\in\mathcal{L}}\lambda^{k}_{i}\leq c_{2},~\forall k\geq K,$$
	where $c_{2}:=\frac{\max\limits_{i\in[m]\setminus\mathcal{L}}\left\{\|\nabla f_{i}(x^{0})\| + L_{i}R\right\}+\frac{\epsilon}{2}}{\max\limits_{i\in[m]\setminus\mathcal{L}}\left\{\|\nabla f_{i}(x^{0})\| + L_{i}R\right\}+\epsilon}<1.$ By substituting the above bound into (\ref{e5.15}), we obtain
	$$\|x^{k+1}-x^{*}\|\leq\sqrt{1-(1-c_{2})\frac{\min\limits_{i\in[m]\setminus\mathcal{L}}\mu_{i}}{L_{\max}}}\|x^{k}-x^{*}\|,~\forall k\geq K.$$
	Without loss of generality, there exists $\kappa\in(0,1)$ such that $\|x^{k}-x^{*}\|\leq\kappa^{k}\|x^{0}-x^{*}\|.$
\end{proof}
\vspace{2mm}
\par From Proposition \ref{pl2}, PGMO without line search only achieves $R$-linear convergence. The following example illustrates that at the early stage of the method, $\sum\limits_{i\in\mathcal{L}}\lambda^{k}_{i}$ can equal to $1$. In other words, we can not obtain $Q$-linear convergence of the method for (\ref{MCOP}).
\vspace{2mm}
\begin{example}
	Consider the multiobjective optimization problem:
	\begin{align*}
		\min\limits_{x\in \mathcal{X}}\left( f_{1}(x), f_{2}(x)\right),
	\end{align*}
	where $f_{1}(x)=\frac{1}{2}x_{1}^{2}+\frac{1}{2}x_{2}^{2}$, $f_{2}(x)=cx_{1}$ ($c$ is a relative small positive constant), and $\mathcal{X}=\{(x_{1},x_{2}):x_{1}\geq 0,x_{2}=0\}.$ By simple calculations, we have
	$$\nabla f_{1}(x)=(x_{1},x_{2})^{T},~\nabla f_{2}(x)=(c,0)^{T},$$
	and the Pareto set is $(0,0)^{T}$. Given a feasible $x^{0}$, at the early stage ($x^{k}_{1}>c$) of PGMO without line search, we have $\lambda_{i}^{k}=1$, and $d^{k}=(-c,0)$. At this stage, we have $\frac{\nm{x^{k+1}-x^{*}}}{\nm{x^{k}-x^{*}}}=\frac{x^{k}_{1}-c}{x^{k}_{1}}$, which tends to $0$ for sufficient small $c$. 
\end{example}

%\subsubsection{Comparing with weighted scalarization} 
%In this subsection, we compare the convergence rates among ABBPGMO, PGMO without line search and the weighted scalarization method. The linear convergence of ABBPGMO and PGMO without line search for \ref{LCMOP} is mainly due to $\sum_{i\in\mathcal{L}}\lambda_{i}^{k}\neq1$. In other words, strongly convex objectives play the key roles in linear convergence. 
%When minimizing \ref{LCMOP} using weighted scalarization method, the method with $\sum_{i\in\mathcal{L}}\lambda_{i}\neq1$ converges linearly, and the rate of linear convergence is presented as follows.

\section{Solving the subproblem via its dual}\label{sec6}
The efficiency of BBPGMO does not only depend on the outer iteration but also how to solve the subproblem efficiently. Motivated by \cite[Section 6]{TFY2022}, we obtain the descent direction by solving the subproblem via its dual, and the dual can be solved by Frank-Wolfe/conditional method efficiently. The subproblem (\ref{dk}) can be reformulated as:
$$\min\limits_{x\in\mathbb{R}^{n}}\max_{\lambda\in\Delta_{m}}\left\{\frac{
	\left\langle\sum\limits_{i\in[m]}\lambda_{i}\nabla f_{i}(x^{k}),x-x^{k}\right\rangle + \sum\limits_{i\in[m]}\lambda_{i}(g_{i}(x)-g_{i}(x^{k}))}{\alpha^{k}_{i}}+\frac{1}{2}\|x-x^{k}\|^{2}\right\}.$$
By Sion's minimax theorem, the above problem is equivalent to
$$\max_{\lambda\in\Delta_{m}}\min\limits_{x\in\mathbb{R}^{n}}\left\{\frac{
	\left\langle\sum\limits_{i\in[m]}\lambda_{i}\nabla f_{i}(x^{k}),x-x^{k}\right\rangle + \sum\limits_{i\in[m]}\lambda_{i}(g_{i}(x)-g_{i}(x^{k}))}{\alpha^{k}_{i}}+\frac{1}{2}\|x-x^{k}\|^{2}\right\}.$$
On the other hand, we can deduce that

\begin{align*}
	&~~~\min\limits_{x\in\mathbb{R}^{n}}\left\{\frac{
		\left\langle\sum\limits_{i\in[m]}\lambda_{i}\nabla f_{i}(x^{k}),x-x^{k}\right\rangle + \sum\limits_{i\in[m]}\lambda_{i}(g_{i}(x)-g_{i}(x^{k}))}{\alpha^{k}_{i}}+\frac{1}{2}\|x-x^{k}\|^{2}\right\}\\
	&=\min\limits_{x\in\mathbb{R}^{n}}\left\{\sum\limits_{i\in[m]}\lambda_{i}\frac{ g_{i}(x)}{\alpha^{k}_{i}}+\frac{1}{2}\left\|x-(x^{k}-\sum\limits_{i\in[m]}\lambda_{i}\frac{\nabla f_{i}(x^{k})}{\alpha^{k}_{i}})\right\|^{2}\right\}-\frac{1}{2}\left\|\sum\limits_{i\in[m]}\lambda_{i}\frac{\nabla f_{i}(x^{k})}{\alpha^{k}_{i}}\right\|^{2}\\
	&~~~~-\sum\limits_{i\in[m]}\lambda_{i}\frac{g_{i}(x^{k})}{\alpha^{k}_{i}}\\
	&=\mathcal{M}_{\sum\limits_{i\in[m]}\lambda_{i}\frac{ g_{i}}{\alpha^{k}_{i}}}\left(x^{k}-\sum\limits_{i\in[m]}\lambda_{i}\frac{\nabla f_{i}(x^{k})}{\alpha^{k}_{i}}\right)-\frac{1}{2}\left\|\sum\limits_{i\in[m]}\lambda_{i}\frac{\nabla f_{i}(x^{k})}{\alpha^{k}_{i}}\right\|^{2}-\sum\limits_{i\in[m]}\lambda_{i}\frac{g_{i}(x^{k})}{\alpha^{k}_{i}}.
\end{align*}

Hence, the dual problem of (\ref{dk}) can be stated as follows:
\begin{align*}\tag{DP}\label{DP}
	-&\min\limits_{\lambda}\omega(\lambda)\\
	&\mathrm{ s.t.} \ \lambda\in\Delta_{m},
\end{align*}
where $$\omega(\lambda):=\frac{1}{2} \left\|\sum\limits_{i\in[m]}\lambda_{i}\frac{\nabla f_{i}(x^{k})}{\alpha^{k}_{i}}\right\|^{2}+\sum\limits_{i\in[m]}\lambda_{i}\frac{g_{i}(x^{k})}{\alpha^{k}_{i}}-\mathcal{M}_{\sum\limits_{i\in[m]}\lambda_{i}\frac{ g_{i}}{\alpha^{k}_{i}}}\left(x^{k}-\sum\limits_{i\in[m]}\lambda_{i}\nabla f_{i}(x^{k})\right).$$
From Proposition \ref{p3}, we have 
\begin{equation}
	P_{\alpha^{k}}(x^{k}) = {\rm Prox}_{\sum\limits_{i\in[m]}\lambda_{i}^{k}\frac{g_{i}}{\alpha^{k}_{i}}}\left(x^{k}-\sum\limits_{i\in[m]}\lambda_{i}^{k}\frac{\nabla f_{i}(x^{k})}{\alpha^{k}_{i}}\right),
\end{equation}
where $\lambda^{k}\in\Delta_{m}$ is a solution of (\ref{DP}). Note that the dual problem is a convex problem with unit simplex constraint, which can be efficiently solved by Frank-Wolfe method. The following proposition introduces how to compute $\nabla\omega$.
\vspace{2mm}
\begin{proposition}\label{p6}
	The function $\omega:\mathbb{R}^{m}\rightarrow\mathbb{R}$ is continuously differentiable and $$\nabla\omega(\lambda)=\frac{g(x^{k})}{\alpha^{k}}-\frac{Jf(x^{k})}{\alpha^{k}}(P_{\alpha^{k}}(\lambda)-x^{k})-\frac{g(P_{\alpha^{k}}(\lambda))}{\alpha^{k}},$$
	where $$\frac{g(x^{k})}{\alpha^{k}}=(\frac{g_{1}(x^{k})}{\alpha^{k}_{1}},\frac{g_{2}(x^{k})}{\alpha_{2}^{k}},...,\frac{g_{m}(x^{k})}{\alpha_{m}^{k}})^{T},$$ $$\frac{Jf(x^{k})}{\alpha^{k}}=(\frac{\nabla f_{1}(x^{k})}{\alpha^{k}_{1}},\frac{\nabla f_{2}(x^{k})}{\alpha_{2}^{k}},...,\frac{\nabla f_{m}(x^{k})}{\alpha_{m}^{k}})^{T},$$ and
	$$P_{\alpha^{k}}(\lambda)={\rm Prox}_{\sum\limits_{i\in[m]}\lambda_{i}\frac{g_{i}}{\alpha^{k}_{i}}}\left(x^{k}-\sum\limits_{i\in[m]}\lambda_{i}\frac{\nabla f_{i}(x^{k})}{\alpha^{k}_{i}}\right).$$
\end{proposition}
\begin{proof}
	We use \cite[Theorem 4.13]{BS2000} to get
	\begin{align*}
		&~~~~~\nabla_{\lambda}\mathcal{M}_{\sum\limits_{i\in[m]}\lambda_{i}\frac{ g_{i}}{\alpha^{k}_{i}}}\left(x^{k}-\sum\limits_{i\in[m]}\lambda_{i}\frac{\nabla f_{i}(x^{k})}{\alpha^{k}_{i}}\right)\\
		&=\frac{g\left({\rm Prox}_{\sum\limits_{i\in[m]}\lambda_{i}\frac{g_{i}}{\alpha^{k}_{i}}}\left(x^{k}-\sum\limits_{i\in[m]}\lambda_{i}\frac{\nabla f_{i}(x^{k})}{\alpha^{k}_{i}}\right)\right)}{\alpha^{k}}\\
		&~~~~~+\frac{Jf(x^{k})}{\alpha^{k}}\left({\rm Prox}_{\sum\limits_{i\in[m]}\lambda_{i}\frac{g_{i}}{\alpha^{k}_{i}}}\left(x^{k}-\sum\limits_{i\in[m]}\lambda_{i}\frac{\nabla f_{i}(x^{k})}{\alpha^{k}_{i}}\right)-x^{k}+\sum\limits_{i\in[m]}\lambda_{i}\frac{\nabla f_{i}(x^{k})}{\alpha^{k}_{i}}\right).
	\end{align*}
	On the other hand, we have
	$$\nabla_{\lambda}\left(\frac{1}{2} \nm{\sum\limits_{i\in[m]}\lambda_{i}\frac{\nabla f_{i}(x^{k})}{\alpha^{k}_{i}}}^{2}+\sum\limits_{i\in[m]}\lambda_{i}\frac{g_{i}(x^{k})}{\alpha^{k}_{i}}\right)=\frac{Jf(x^{k})}{\alpha^{k}}\left(\sum\limits_{i\in[m]}\lambda_{i}\frac{\nabla f_{i}(x^{k})}{\alpha^{k}_{i}}\right)+\frac{g(x^{k})}{\alpha^{k}}.$$
	The desired result follows by adding the above two equalities.
\end{proof}
\vspace{2mm}
\par Proposition \ref{p6} highlights the importance of the cheap proximal operation for $\sum\limits_{i\in[m]}\lambda_{i}\frac{g_{i}}{\alpha^{k}_{i}}$ in solving the dual problem. Fortunately,  the cheap proximal operation can be found in machine learning and statistics, where $g_{i}(x)=\|x\|_{1}$ for $i\in[m]$.
\section{Numerical results}\label{sec7} 
In this section, we present numerical results to demonstrate the performance of BBPGMO for various problems. All numerical experiments were implemented in Python 3.7 and executed on a personal computer equipped with an Intel Core i7-11390H, 3.40 GHz processor, and 16 GB of RAM. For BBPGMO, we set $\alpha_{\min}=10^{-3}$ and $\alpha_{\max}=10^{3}$ in equation (\ref{alpha_k})\footnote{For larger-scale and more complicated problems, smaller values for $\alpha_{\min}$ and larger values for $\alpha_{\max}$ should be selected.}. In the line search procedure, we set $\sigma=10^{-4}$ and $\gamma=0.5$. To ensure that the algorithms terminate after a finite number of iterations, we use the stopping criterion $|d(x)|\leq 10^{-6}$ for all tested algorithms. We also set the maximum number of iterations to 500.

\subsection{Comparing with PGMO$_{\mu}$ and PGMO$_{L}$}
As described in Remark \ref{r8}, in the case of strong convexity, BBPGMO may outperform PGMO$_{\mu}$ and PGMO$_{L}$, where $\alpha^k_i=\mu_i$ and $\alpha^k_i=L_i$ for $i\in[m]$, respectively. In the following, we present comparative numerical results to validate this statement. We consider a series of quadratic problems defined as follows:
$$f_{i}(x)=\frac{1}{2}\left\langle x,A_{i}x\right\rangle + \left\langle b_{i},x\right\rangle,\ g_{i}(x)=\frac{1}{n}\|x\|_{1},~i=1,2,$$
where $A_i\in\mathbb{R}^{n\times n}$ is a diagonal matrix with diagonal elements randomly generated from the uniform distribution [1, 100]. Each component of $b_i\in\mathbb{R}^n$ is randomly generated from the uniform distribution [-10, 10]. Table \ref{tab1} provides a problem illustration and the corresponding numerical results. The second column presents the dimension of the variables, while $x_L$ and $x_U$ represent the lower and upper bounds of the variables, respectively. For each problem, we perform 200 computations using the same initial points for different tested algorithms. The initial points are randomly selected within the specified lower and upper bounds. Box constraints are handled by augmented line search, which ensures that $x_L\leq x+td\leq x_U$. The recorded averages from the 200 runs include the number of iterations, the number of function evaluations, and the CPU time.

\begin{table}[h]
	\centering
	\caption{Description of all test problems and number of average iterations (iter), number of average function evaluations (feval) and average CPU time (time ($ms$)) of BBPGMO, PGMO$_{\mu}$($\alpha^{k}_{i}=\mu_{i}$) and PGMO$_{L}$($\alpha^{k}_{i}=L_{i}$) implemented on different test problems.}
	\label{tab1}
	\resizebox{.95\columnwidth}{!}{
		% \usepackage{multirow}
		% \usepackage{multirow}
		\begin{tabular}{llll|rrrlrrrlrrr}
			\hline
			{Problem} &
			{$n$} &
			{$x_{L}$} &
			{$x_{U}$} &
			\multicolumn{1}{l}{BBPGMO} &
			\multicolumn{1}{l}{} &
			\multicolumn{1}{l}{} &
			&
			\multicolumn{1}{l}{PGMO$_{\mu}$} &
			\multicolumn{1}{l}{} &
			\multicolumn{1}{l}{} &
			&
			\multicolumn{1}{l}{PGMO$_{L}$} &
			\multicolumn{1}{l}{} &
			\multicolumn{1}{l}{} \\ \cline{5-7} \cline{9-11} \cline{13-15} 
			&
			&
			&
			&
			\multicolumn{1}{c}{iter} &
			\multicolumn{1}{c}{feval} &
			\multicolumn{1}{c}{time} &
			\multicolumn{1}{c}{} &
			\multicolumn{1}{c}{iter} &
			\multicolumn{1}{c}{feval} &
			\multicolumn{1}{c}{time} &
			\multicolumn{1}{c}{} &
			\multicolumn{1}{c}{iter} &
			\multicolumn{1}{c}{feval} &
			\multicolumn{1}{c}{time} \\ \hline
			a & 2  & (-2,-2)     & (2,2)     & \textbf{3.12}  & \textbf{3.50}          & \textbf{2.11}            &  & 12.12  & 27.68  & 4.30   &  & 7.33   & {7.33}  & {2.73} \\
			b & 10  & (-2,...,-2)   & (2,...,2)   & \textbf{18.95} & \textbf{26.22}          & \textbf{25.23}           &  & 128.68  & 738.29  & 158.91   &  & 101.09  & {101.09} & {25.55} \\
			c & 50  & (-2,...,-2)  & (2,...,2)   & \textbf{18.74} & \textbf{25.29} & \textbf{48.20}  &  & 83.32  & 432.25  & 133.44  &  & 60.48  & 60.48          & {53.44}        \\
			d & 100  & (-2,...,-2) & (2,...,2) & \textbf{26.73} & \textbf{39.51} & \textbf{88.90}  &  & 76.79  & 420.34  & 328.28  &  & 79.08  & 79.08          & {92.11}         \\
			e & 100 & 100(-1,...,-1) & 100(1,...,1) & \textbf{54.98} & \textbf{87.47} & {112.66} &  & 222.40 & 1223.27 & 218.91 &  & 354.38 & 354.38         & \textbf{61.88}        \\ \hline
		\end{tabular}
	}
\end{table}

\par In Table \ref{tab1}, we provide the average number of iterations (iter), average number of function evaluations (feval), and average CPU time (time ($ms$)) for each test problem across the different algorithms. In terms of the average number of iterations and average number of function evaluations, the numerical results confirm that BBPGMO outperforms both PGMO$_{\mu}$ and PGMO$_{L}$.
\subsection{Comparing with PGMO}
In this subsection, we conduct a comparison between BBPGMO and PGMO for different problems. Each objective function in the tested problems consists of two components: $g_i=\frac{1}{n}\|x\|_{1}$ for $i\in[m]$, and the details of $f$ are provided in Table \ref{tab2}. The second and third columns present the dimension of variables and objective functions, respectively. The lower and upper bounds of variables are denoted by $x_{L}$ and $x_{U}$, respectively. For each problem, we perform 200 computations using the same initial points for different tested algorithms. The initial points are randomly selected within the specified lower and upper bounds. Box constraints are handled by augmented line search, which ensures that $x_{L}\leq x+t d\leq x_{U}$. The recorded averages from the 200 runs include the number of iterations, number of function evaluations, CPU time, and stepsize.
\begin{small}
	\begin{table}[h]
		\centering
		\caption{Description of all test problems used in numerical experiments.}
		\label{tab2}
		\resizebox{.9\columnwidth}{!}{
			\begin{tabular}{lllllllllll}
				\hline
				Problem    &  & $n$     &           & $m$     &  & $x_{L}$               &  & $x_{U}$             &  & Reference \\ \hline
				BK1        &  & 2     & \textbf{} & 2     &  & (-5,-5)         &  & (10,10)       &  & \cite{BK1}         \\
				DD1        &  & 5     & \textbf{} & 2     &  & (-20,...,-20)   &  & (20,...,20)   &  & \cite{DD1998}         \\
				Deb        &  & 2     & \textbf{} & 2     &  & (0.1,0.1)       &  & (1,1)         &  & \cite{D1999}          \\
				Far1       &  & 2     & \textbf{} & 2     &  & (-1,-1)         &  & (1,1)         &  & \cite{BK1}         \\
				FDS        &  & 5     & \textbf{} & 3     &  & (-2,...,-2)     &  & (2,...,2)     &  & \cite{FD2009}         \\
				FF1        &  & 2     & \textbf{} & 2     &  & (-1,-1)         &  & (1,1)         &  & \cite{BK1}         \\
				Hil1       &  & 2     & \textbf{} & 2     &  & (0,0)           &  & (1,1)         &  & \cite{Hil1}         \\
				Imbalance1 &  & 2     & \textbf{} & 2     &  & (-2,-2)         &  & (2,2)         &  & \cite{CTY2023}         \\
				Imbalance2 &  & 2     & \textbf{} & 2     &  & (-2,-2)         &  & (2,2)         &  & \cite{CTY2023}         \\
				JOS1a      &  & 50    & \textbf{} & 2     &  & (-2,...,-2)     &  & (2,...,2)     &  & \cite{JO2001}         \\
				JOS1b      &  & 100   & \textbf{} & 2     &  & (-2,...,-2)     &  & (2,...,2)     &  & \cite{JO2001}         \\
				JOS1c      &  & 100   & \textbf{} & 2     &  & (-50,...,-50)   &  & (50,...,50)   &  & \cite{JO2001}         \\
				JOS1d      &  & 100   & \textbf{} & 2     &  & (-100,...,-100) &  & (100,...,100) &  & \cite{JO2001}        \\
				LE1        &  & 2     & \textbf{} & 2     &  & (-5,-5)         &  & (10,10)       &  & \cite{BK1}        \\
				PNR        &  & 2     & \textbf{} & 2     &  & (-2,-2)         &  & (2,2)         &  & \cite{PN2006}        \\
				VU1        &  & 2     & \textbf{} & 2     &  & (-3,-3)         &  & (3,3)         &  & \cite{BK1}        \\
				WIT1       &  & 2     & \textbf{} & 2     &  & (-2,-2)         &  & (2,2)         &  & \cite{W2012}       \\
				WIT2       &  & 2     & \textbf{} & 2     &  & (-2,-2)         &  & (2,2)         &  & \cite{W2012}        \\
				WIT3       &  & 2     & \textbf{} & 2     &  & (-2,-2)         &  & (2,2)         &  & \cite{W2012}        \\
				WIT4       &  & 2     & \textbf{} & 2     &  & (-2,-2)         &  & (2,2)         &  & \cite{W2012}        \\
				WIT5       &  & 2     & \textbf{} & 2     &  & (-2,-2)         &  & (2,2)         &  & \cite{W2012}        \\
				WIT6       &  & 2 & \textbf{} & 2 &  & (-2,-2)         &  & (2,2)         &  & \cite{W2012}        \\ \hline
			\end{tabular}
		}
	\end{table}
\end{small}

\begin{figure}[H]
	\centering
	\subfigure[BBPGMO]
	{
		\begin{minipage}[H]{.4\linewidth}
			\centering
			\includegraphics[scale=0.3]{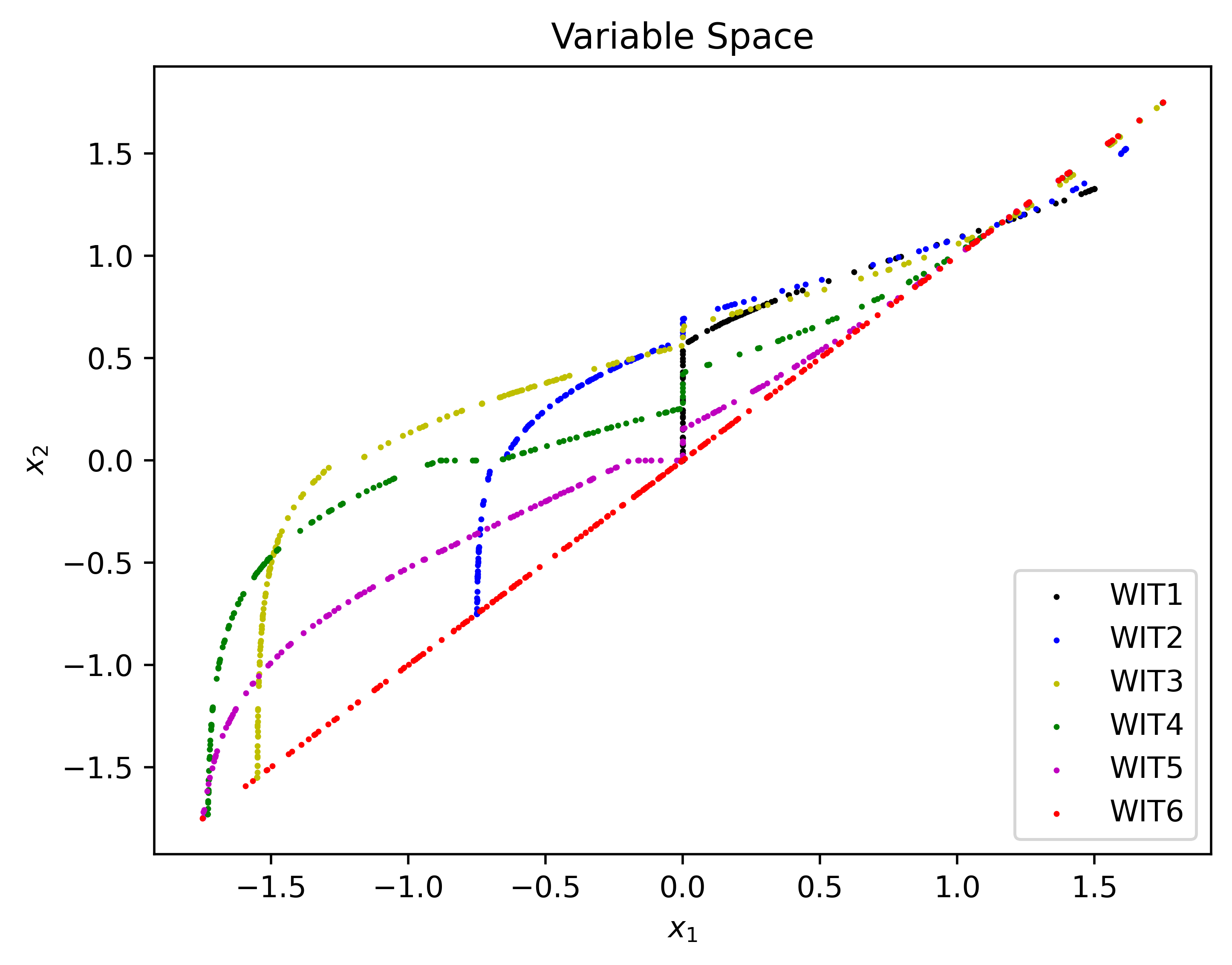} \\
			\includegraphics[scale=0.3]{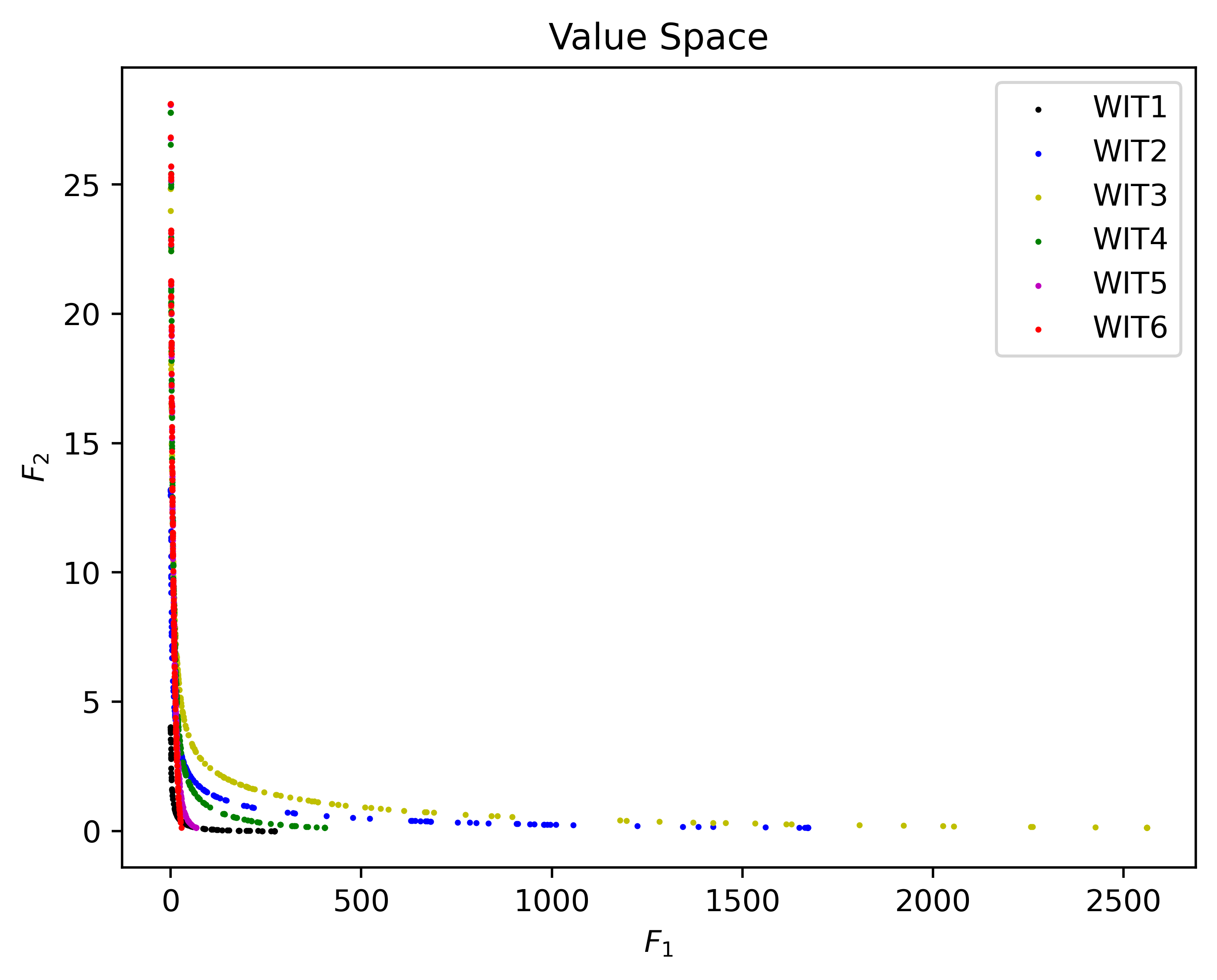}
		\end{minipage}
	}
	\subfigure[PGMO]
	{
		\begin{minipage}[H]{.4\linewidth}
			\centering
			\includegraphics[scale=0.3]{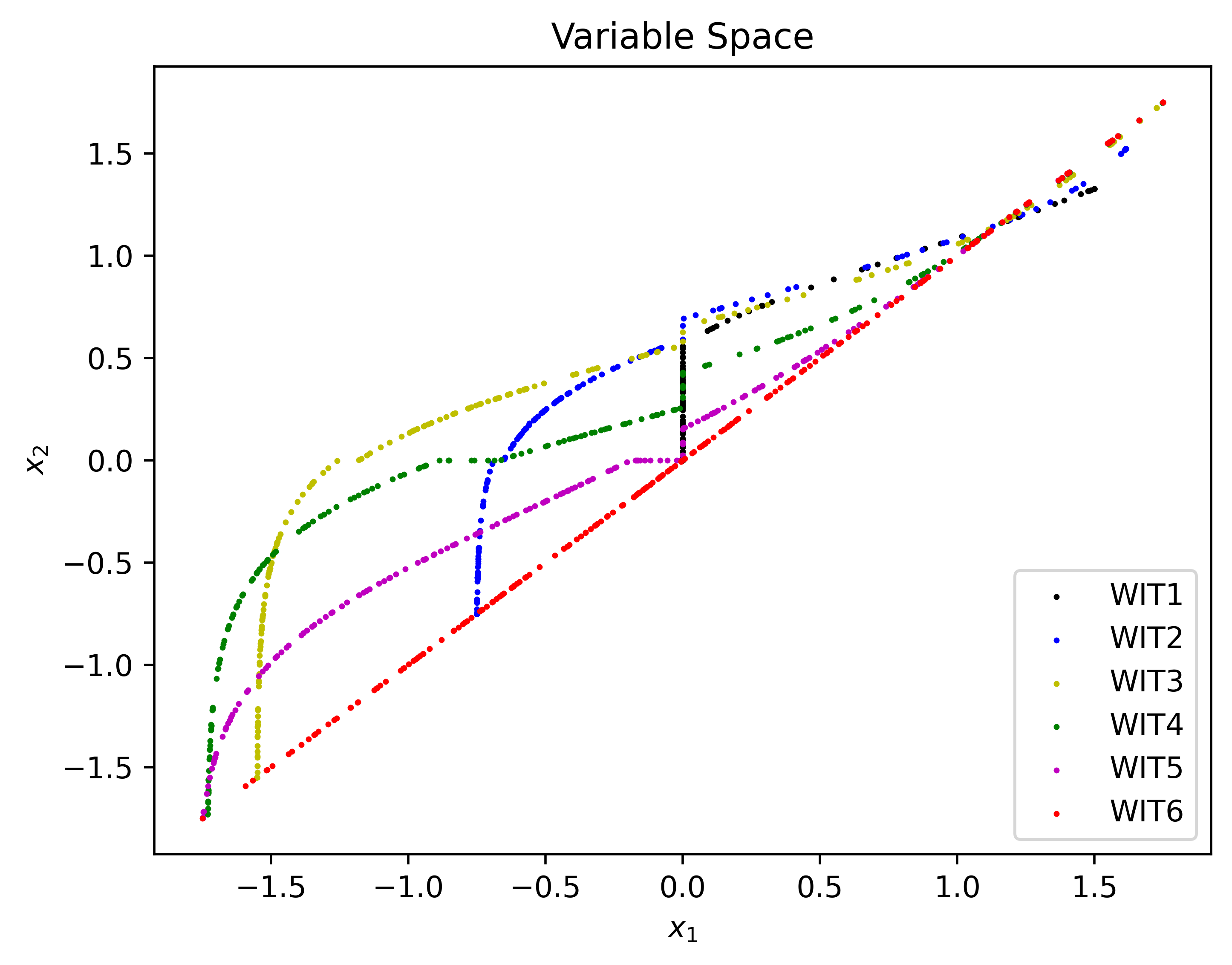} \\
			\includegraphics[scale=0.3]{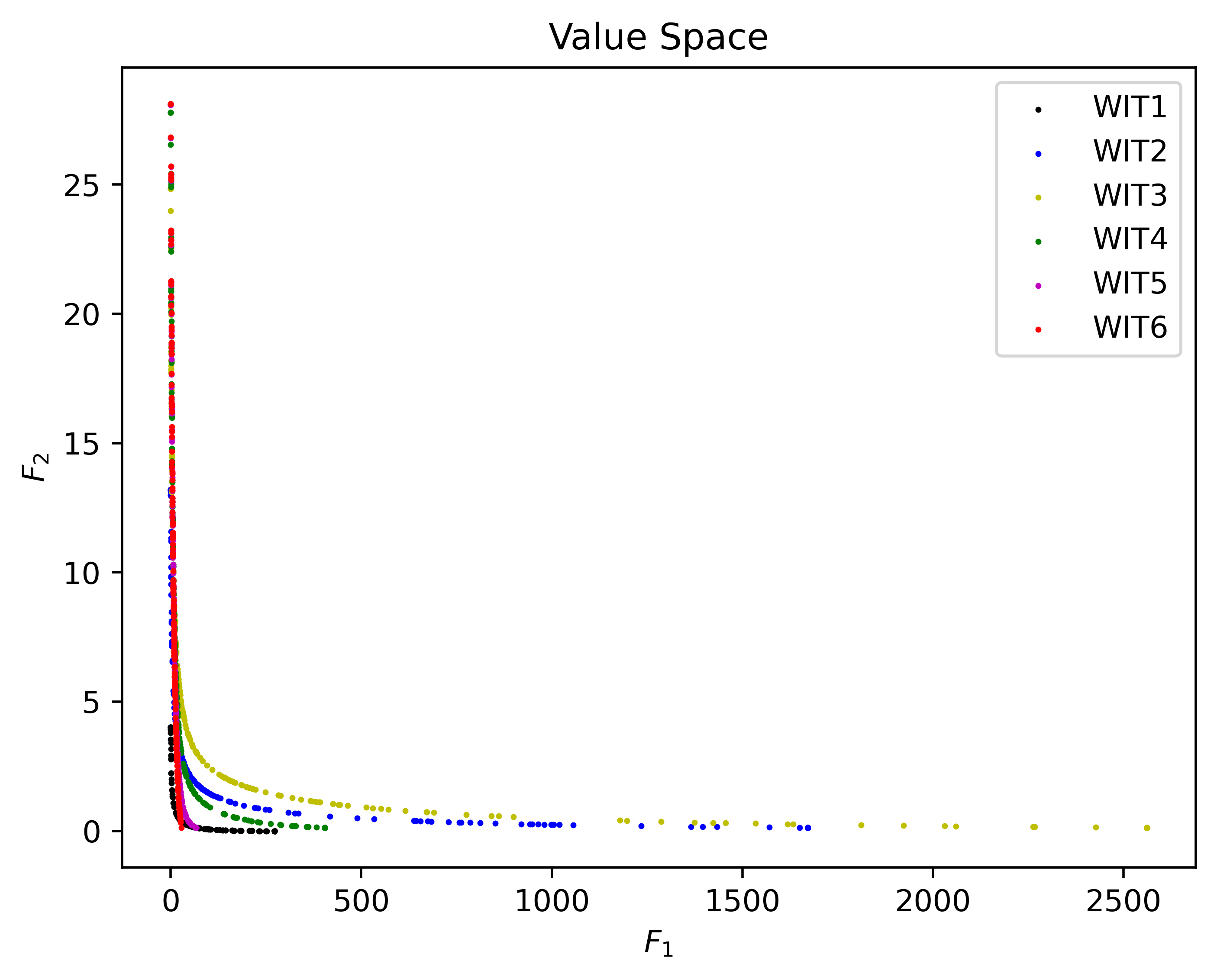}
		\end{minipage}
	}
	
	\caption{Numerical results obtained by BBPGMO and PGMO for problems WIT1-6.}
	\label{f1}
\end{figure}

\begin{figure}[H]
	\centering
	\subfigure[FDS]
	{
		\begin{minipage}[H]{.22\linewidth}
			\centering
			\includegraphics[scale=0.18]{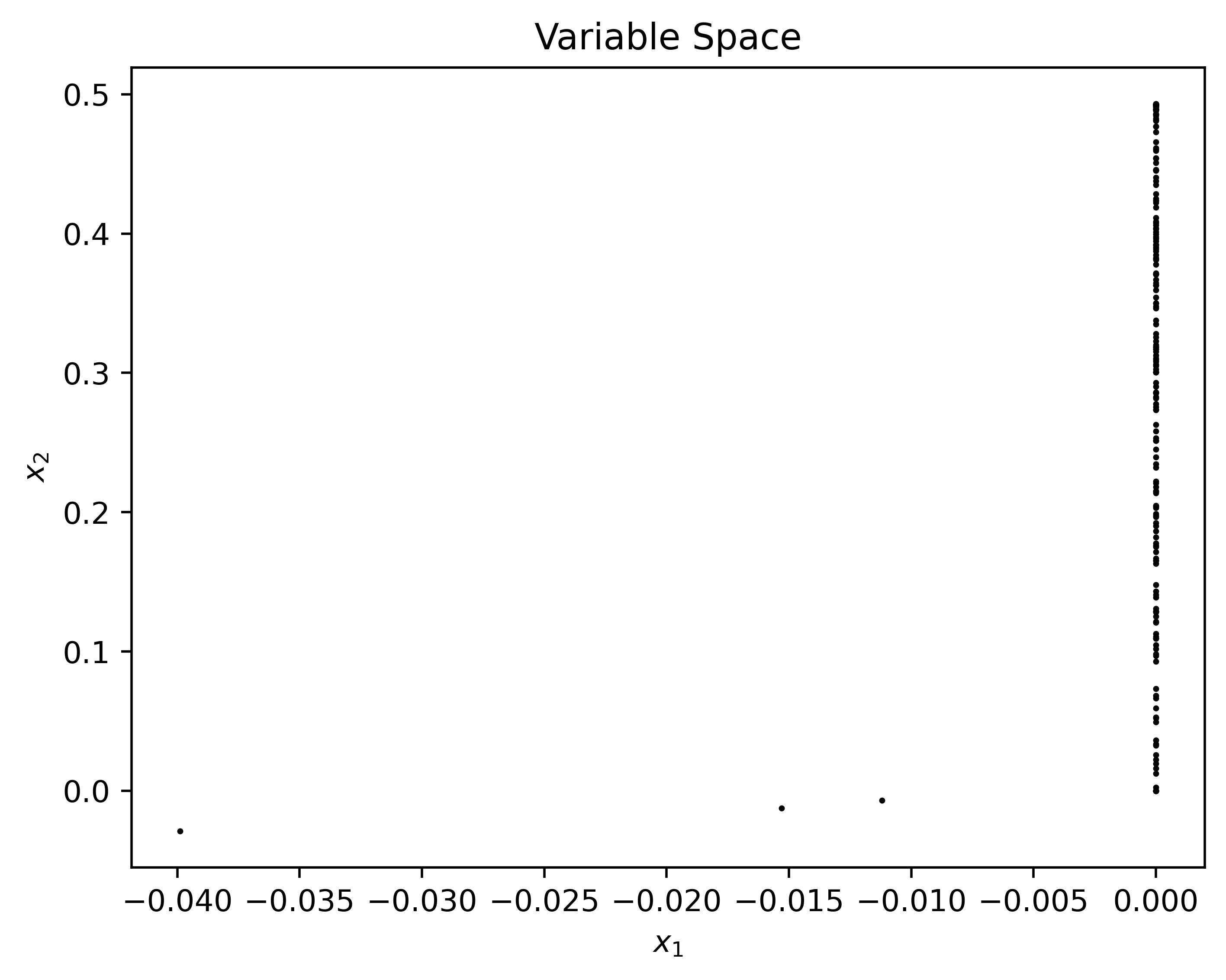} \\
			\includegraphics[scale=0.175]{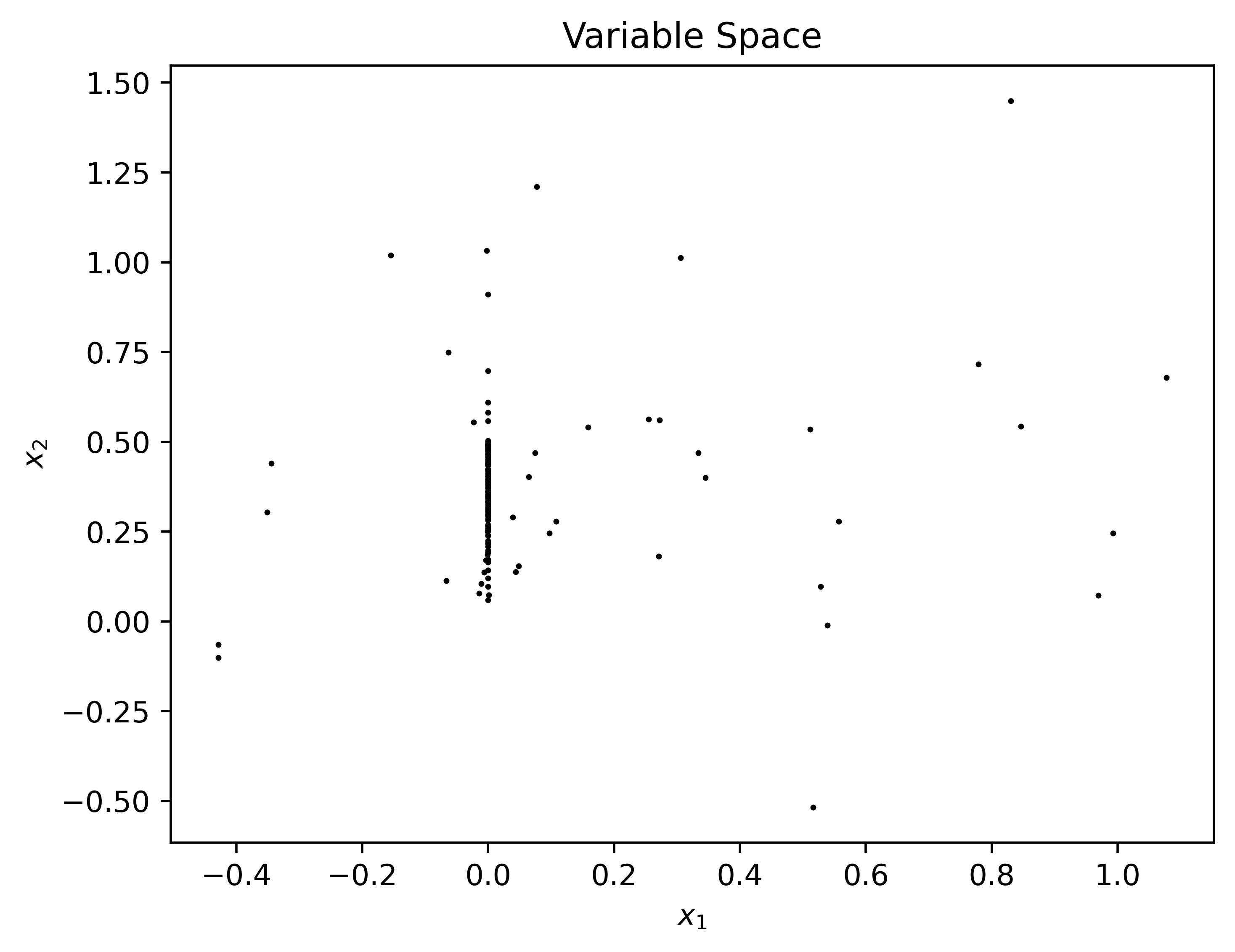}
		\end{minipage}
	}
	\subfigure[Deb]
	{
		\begin{minipage}[H]{.22\linewidth}
			\centering
			\includegraphics[scale=0.18]{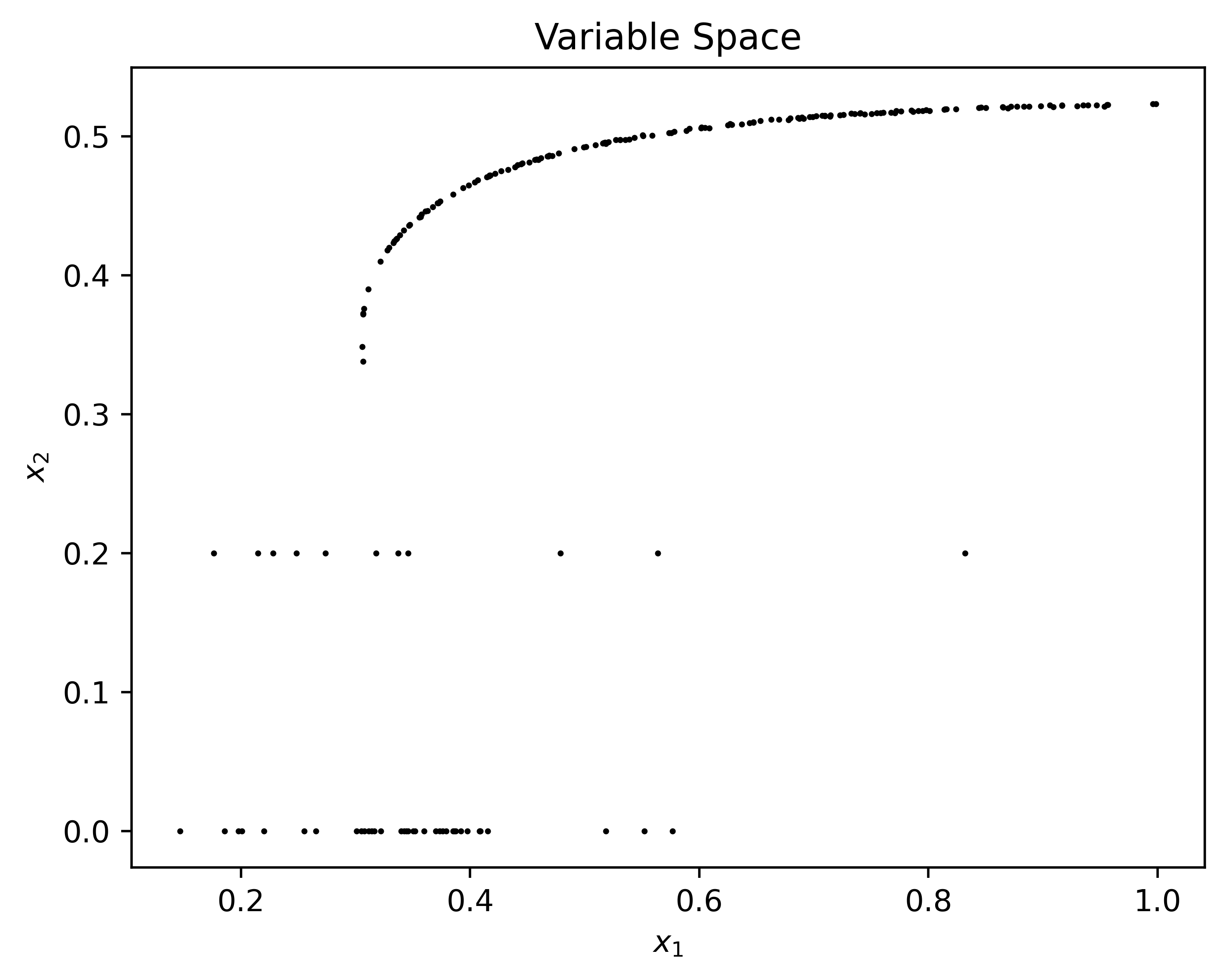} \\
			\includegraphics[scale=0.18]{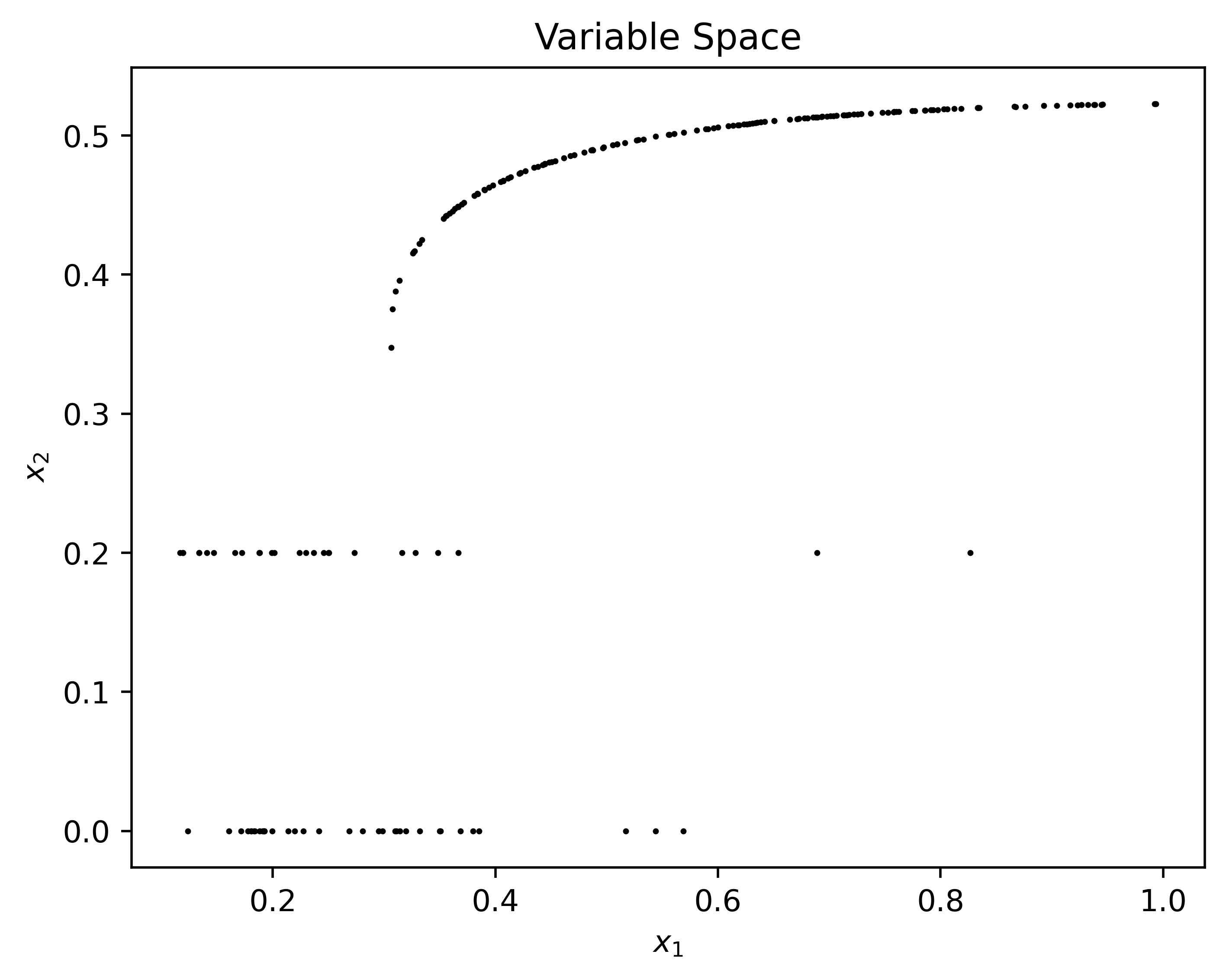}
		\end{minipage}
	}
	\subfigure[VU1]
	{
		\begin{minipage}[H]{.24\linewidth}
			\centering
			\includegraphics[scale=0.18]{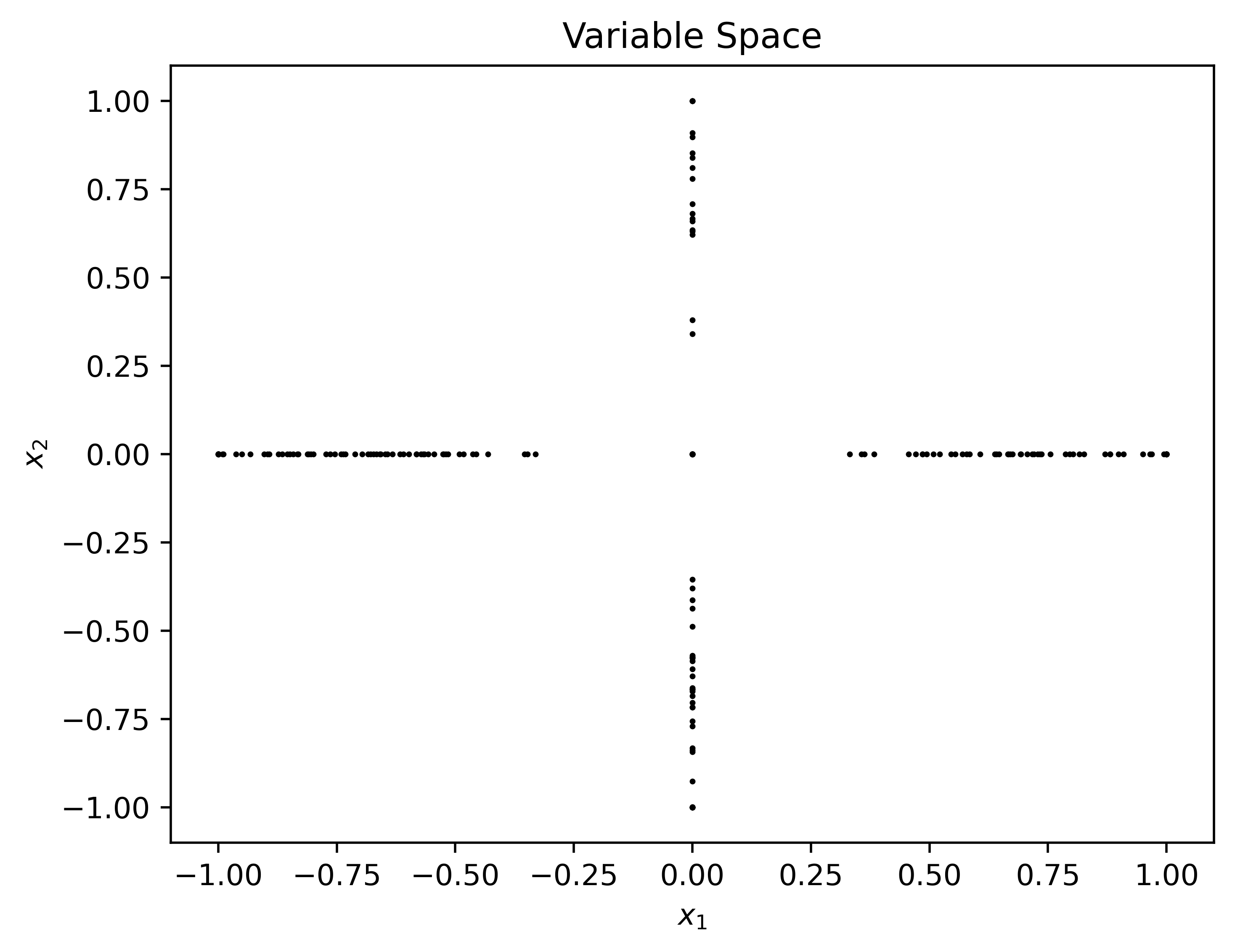} \\
			\includegraphics[scale=0.18]{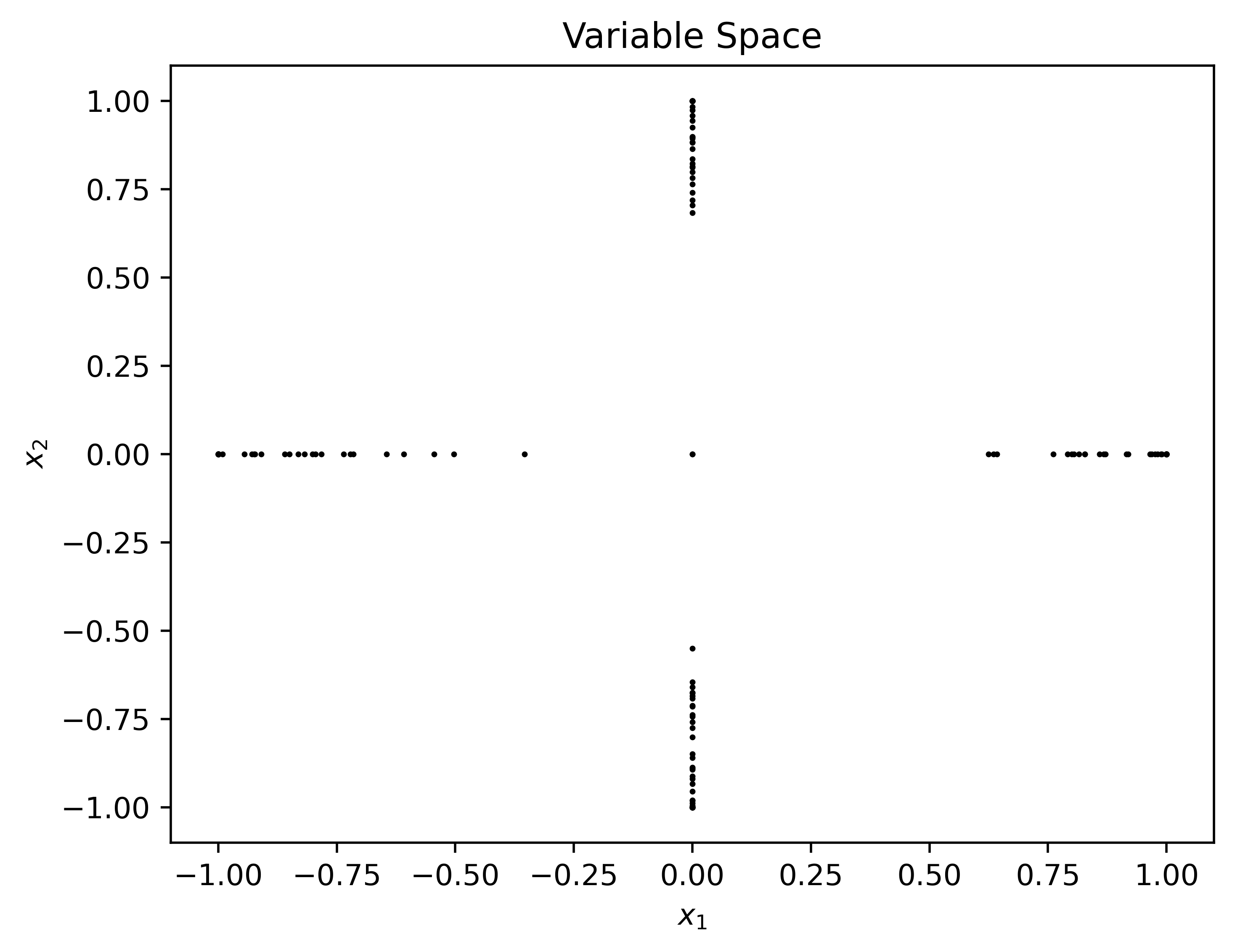}
		\end{minipage}
	}
	\subfigure[Far1]
	{
		\begin{minipage}[H]{.24\linewidth}
			\centering
			\includegraphics[scale=0.18]{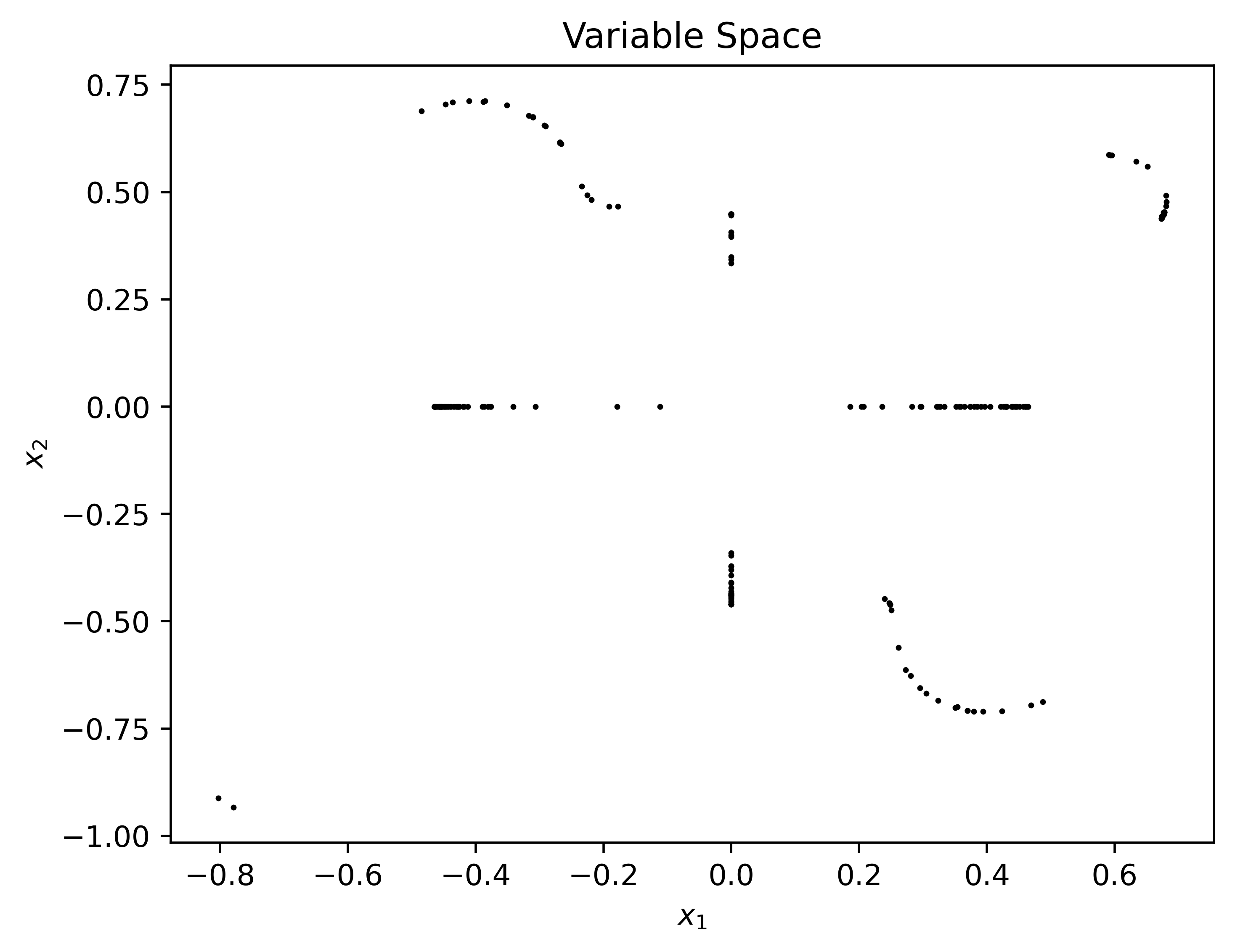} \\
			\includegraphics[scale=0.18]{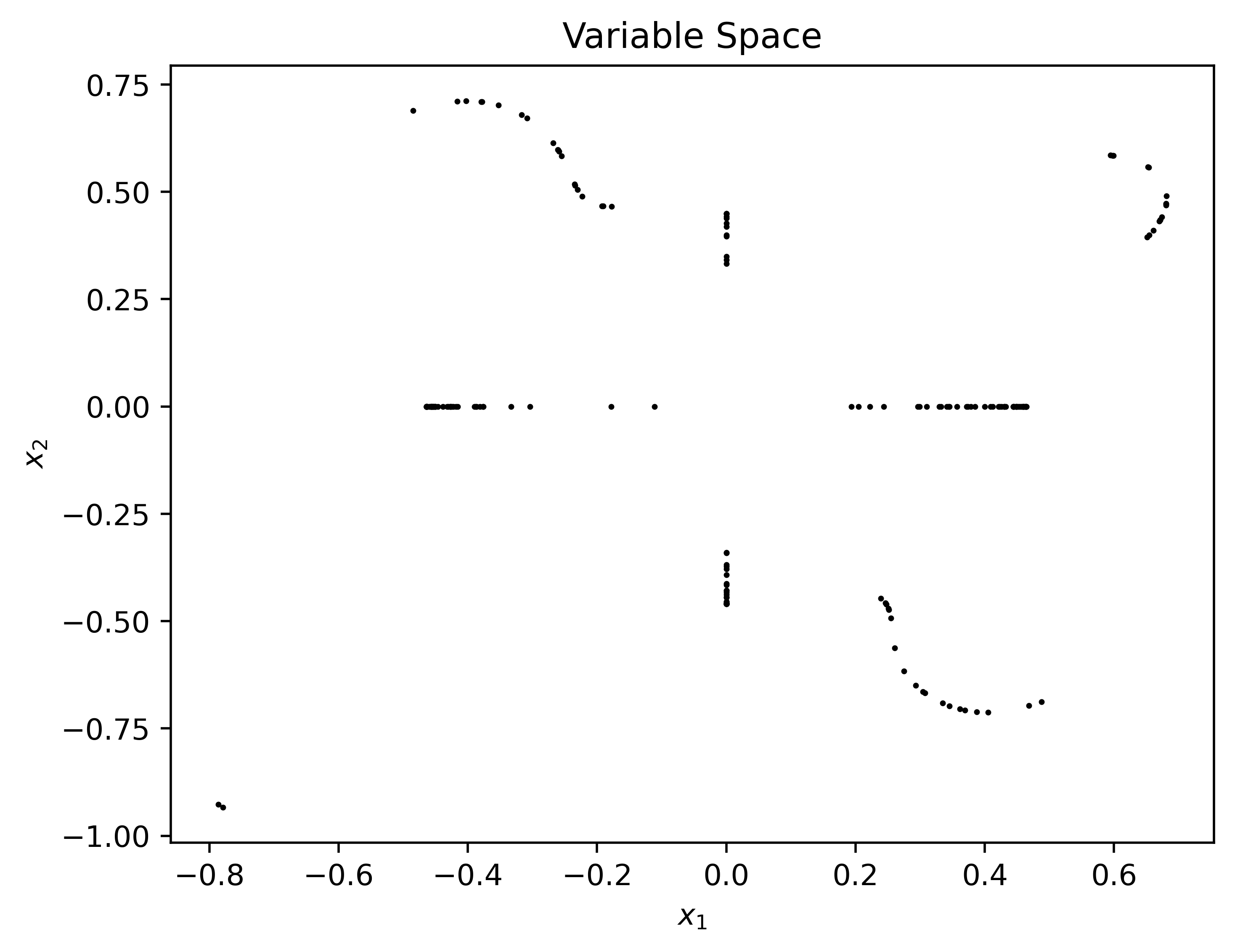}
		\end{minipage}
	}
	\caption{Numerical results in variable space obtained by BBPGMO ({\bf top}) and PGMO for problems FDS, Deb, VU1, and Far1.}
	\label{f2}
\end{figure}
\begin{figure}[H]
	\centering
	\subfigure[FDS]
	{
		\begin{minipage}[H]{.22\linewidth}
			\centering
			\includegraphics[scale=0.2]{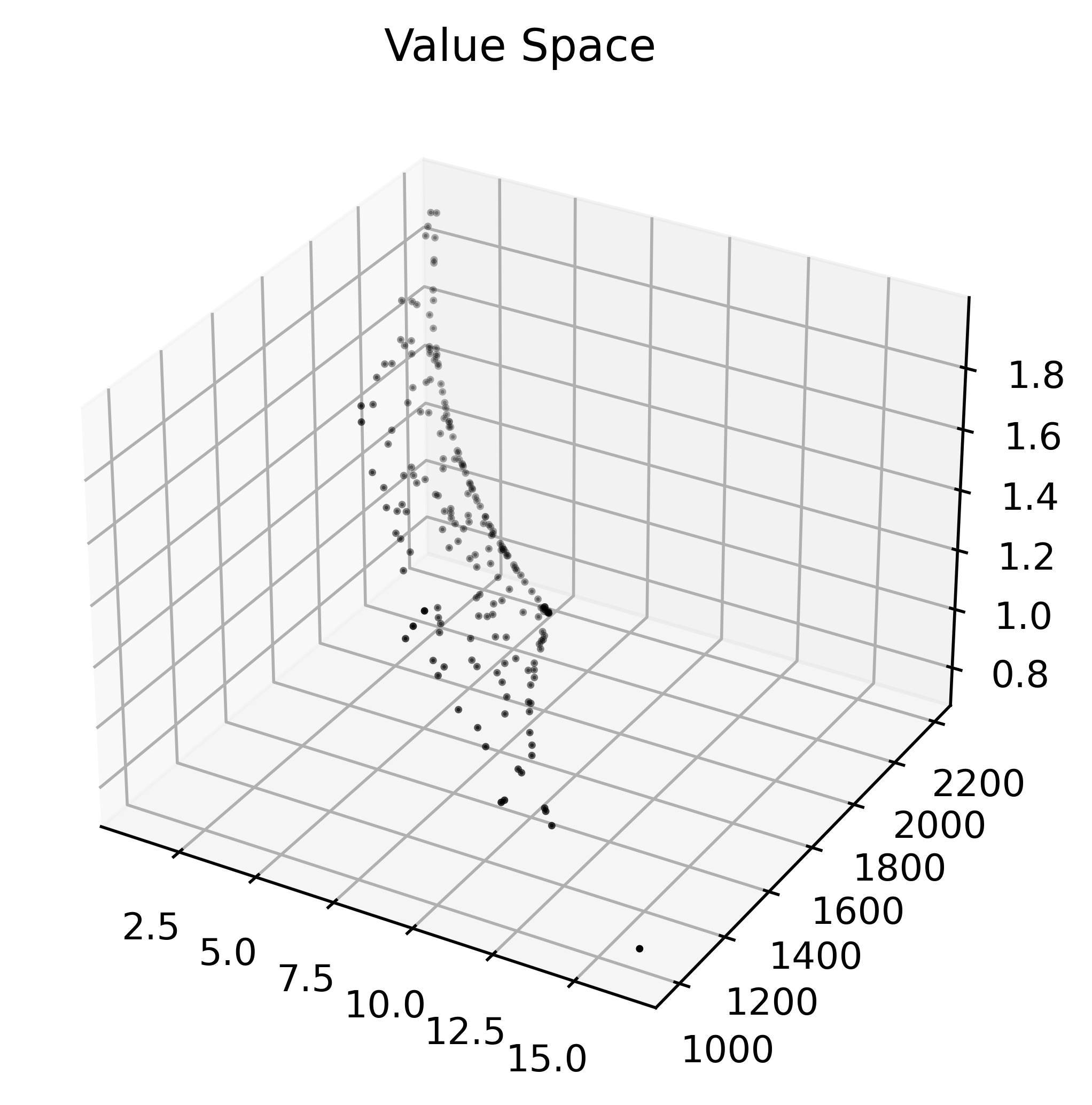} \\
			\includegraphics[scale=0.2 ]{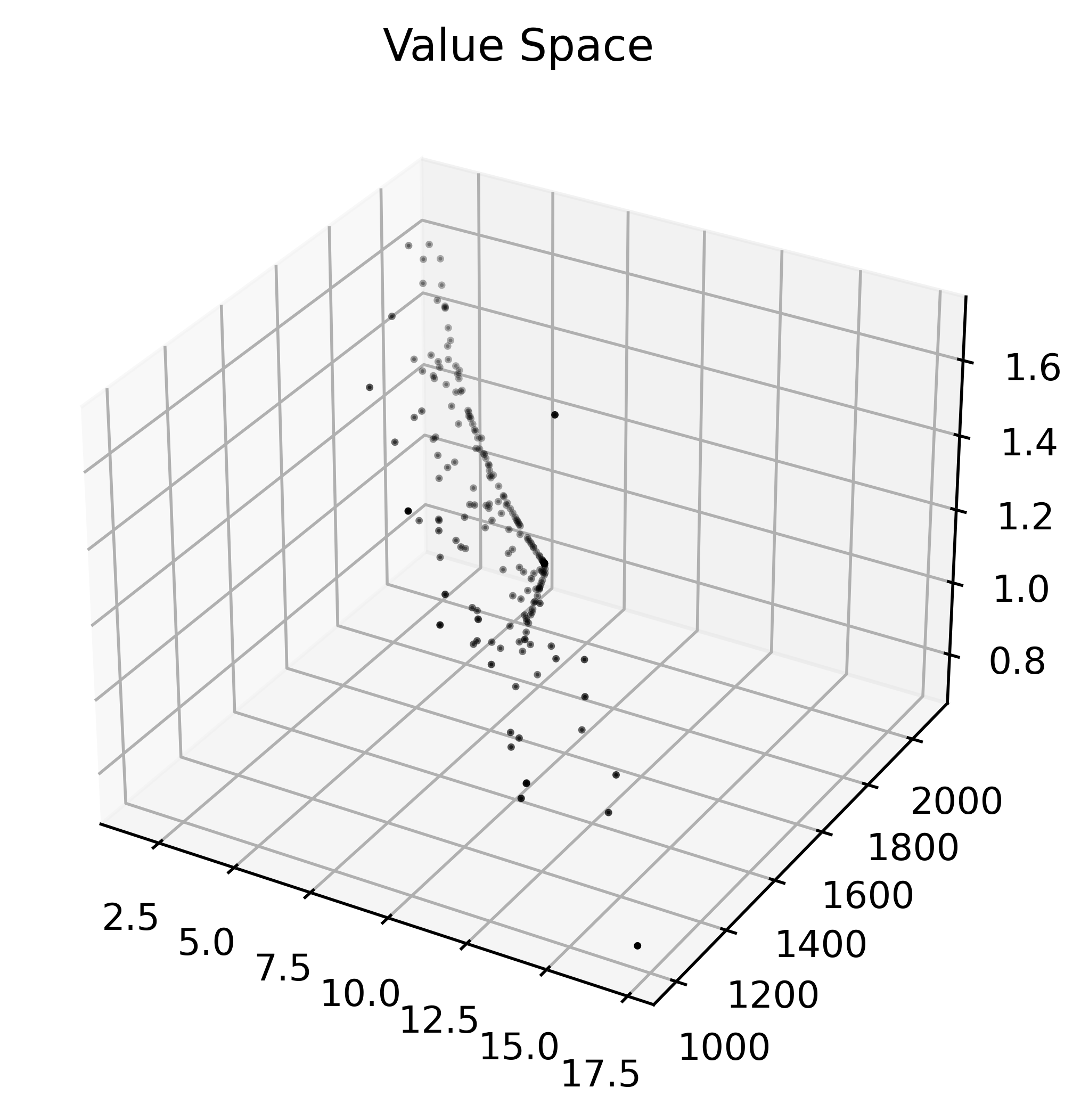}
		\end{minipage}
	}
	\subfigure[Deb]
	{
		\begin{minipage}[H]{.22\linewidth}
			\centering
			\includegraphics[scale=0.18]{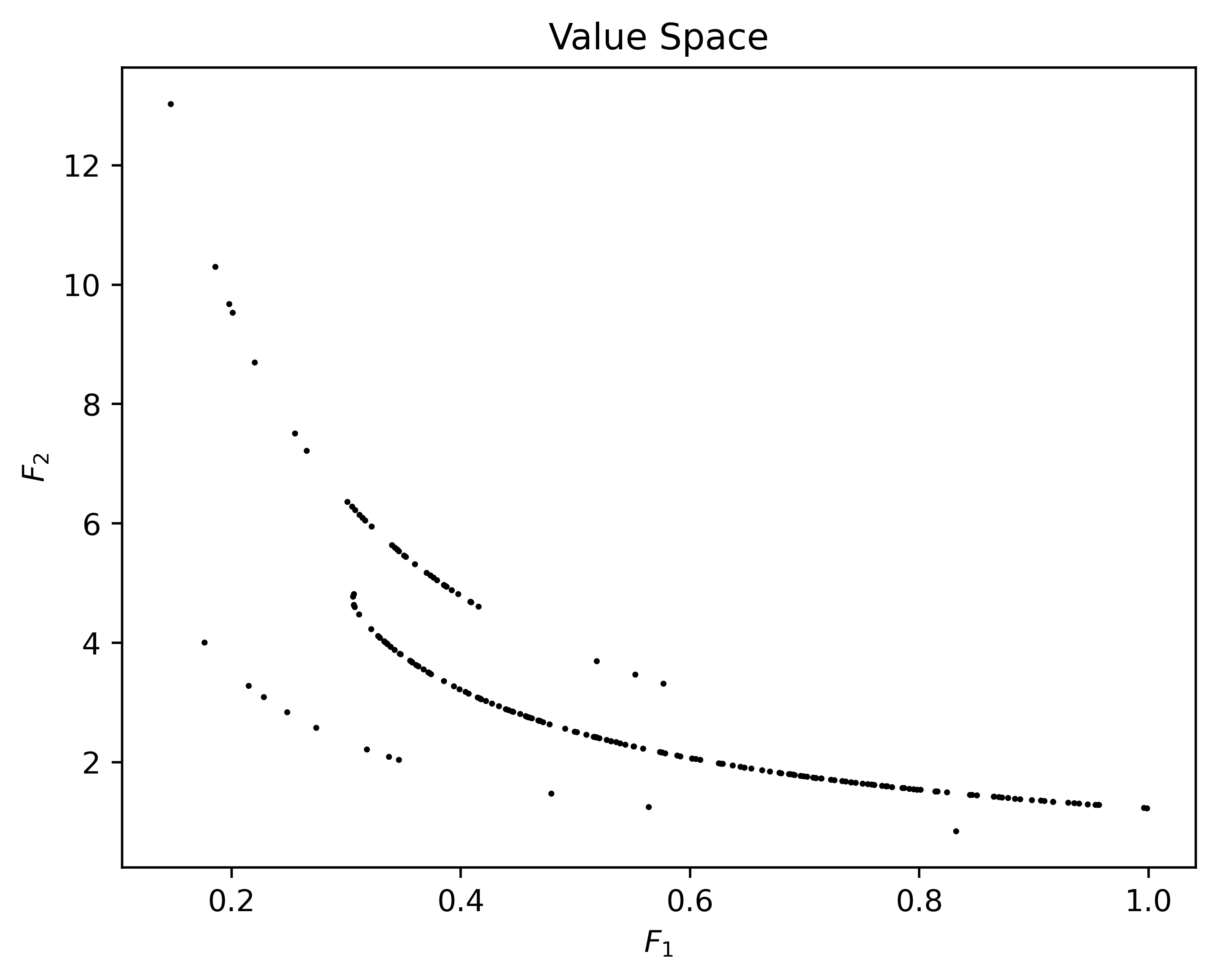} \\
			\includegraphics[scale=0.18]{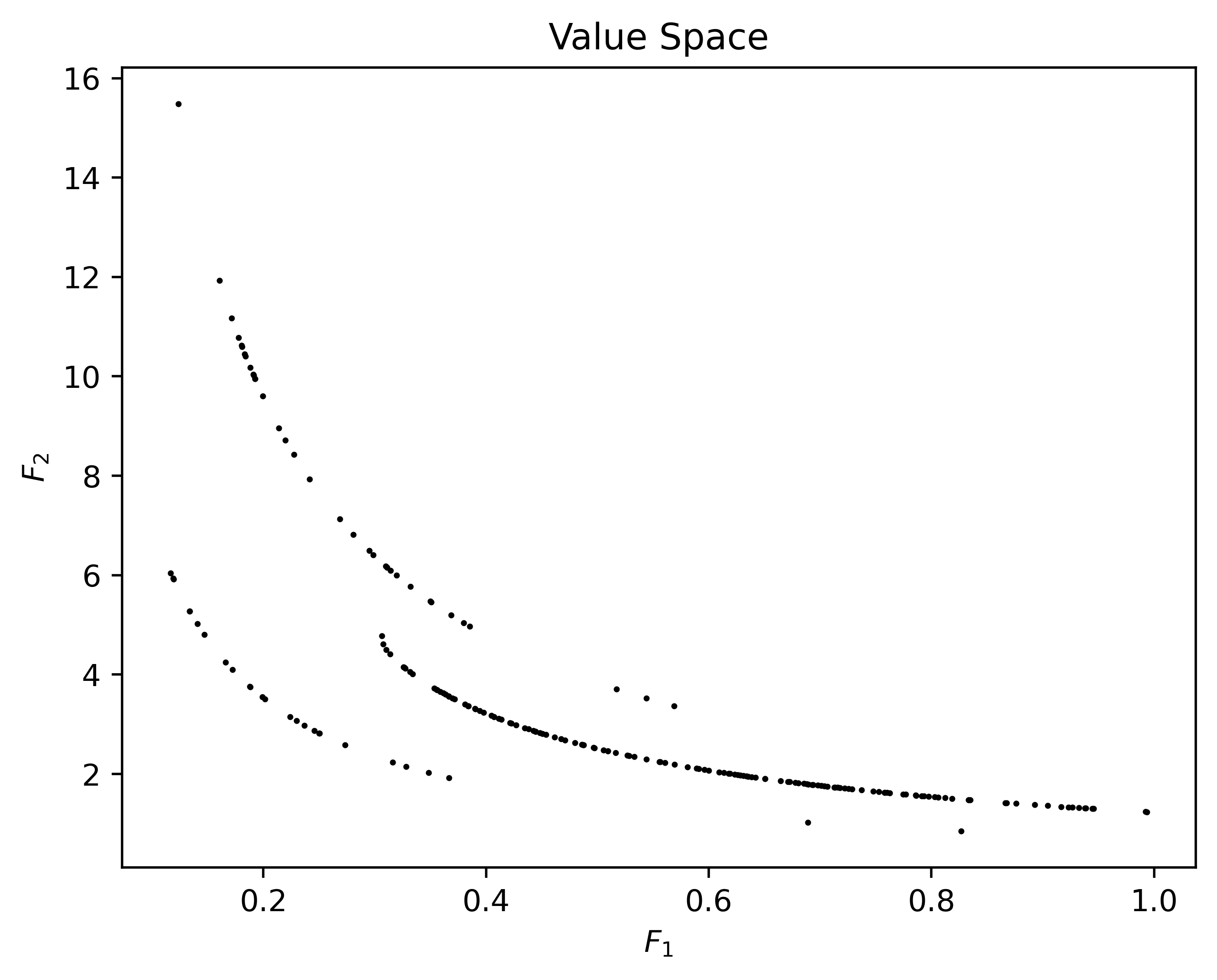}
		\end{minipage}
	}
	\subfigure[VU1]
	{
		\begin{minipage}[H]{.24\linewidth}
			\centering
			\includegraphics[scale=0.18]{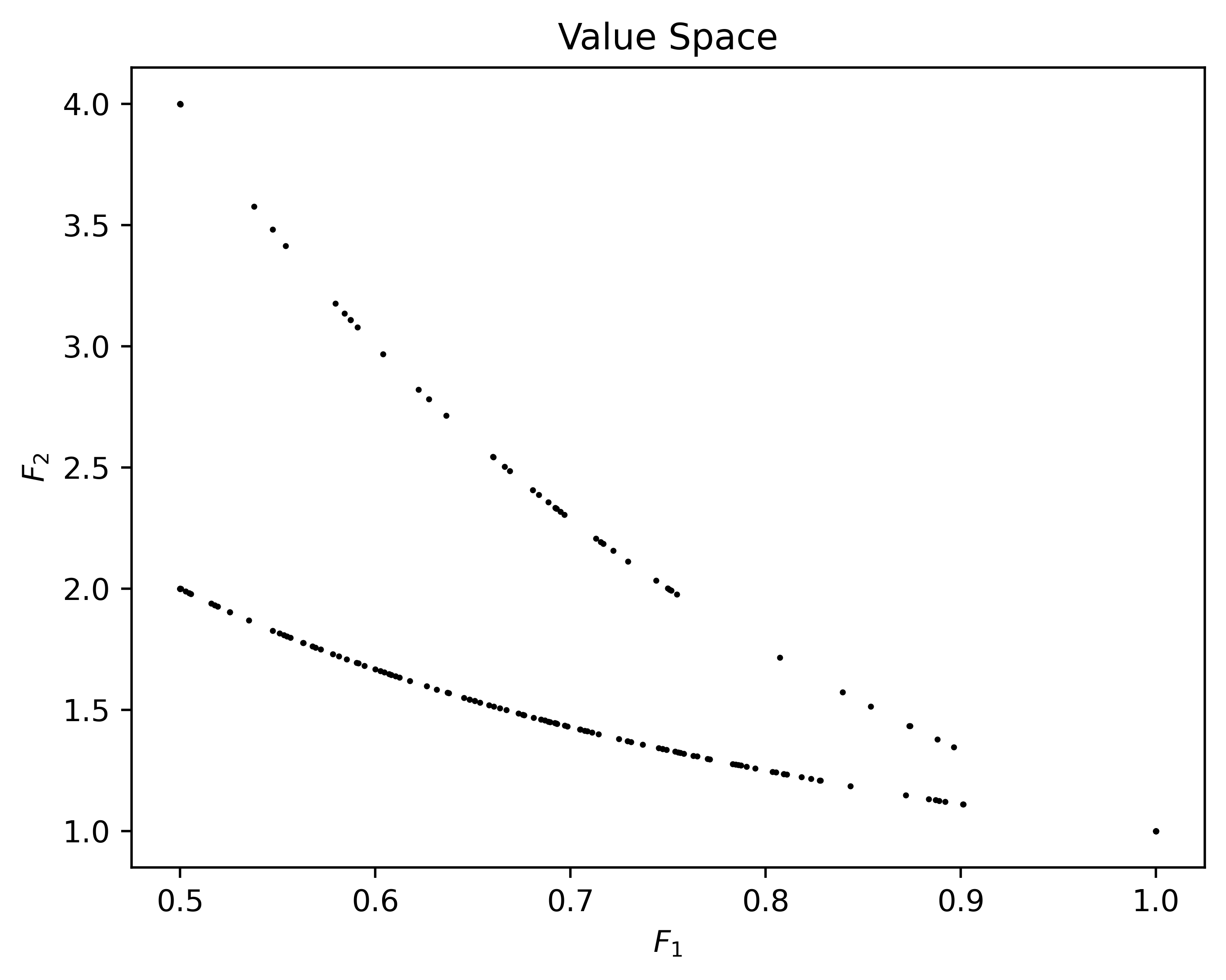} \\
			\includegraphics[scale=0.18]{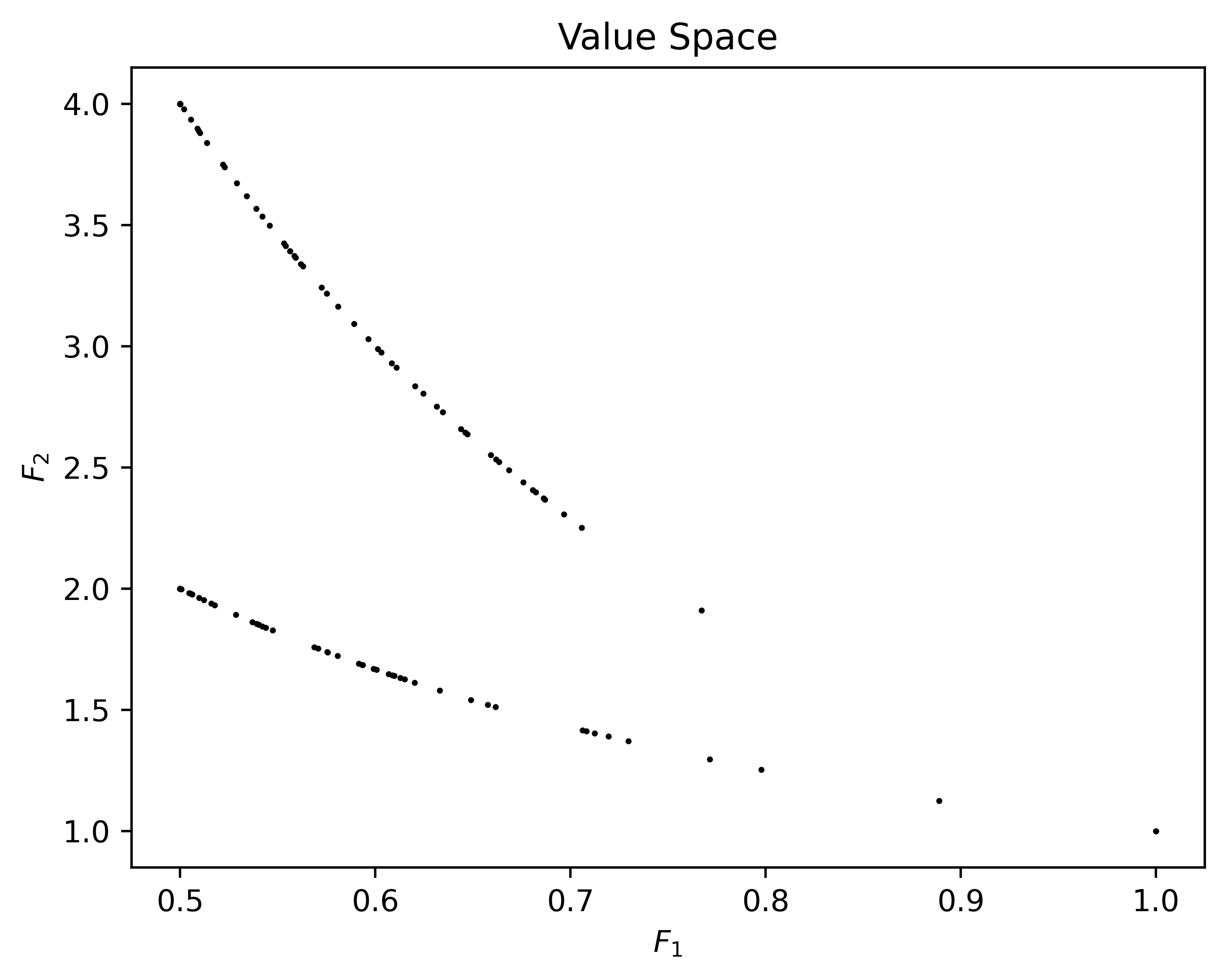}
		\end{minipage}
	}
	\subfigure[Far1]
	{
		\begin{minipage}[H]{.24\linewidth}
			\centering
			\includegraphics[scale=0.18]{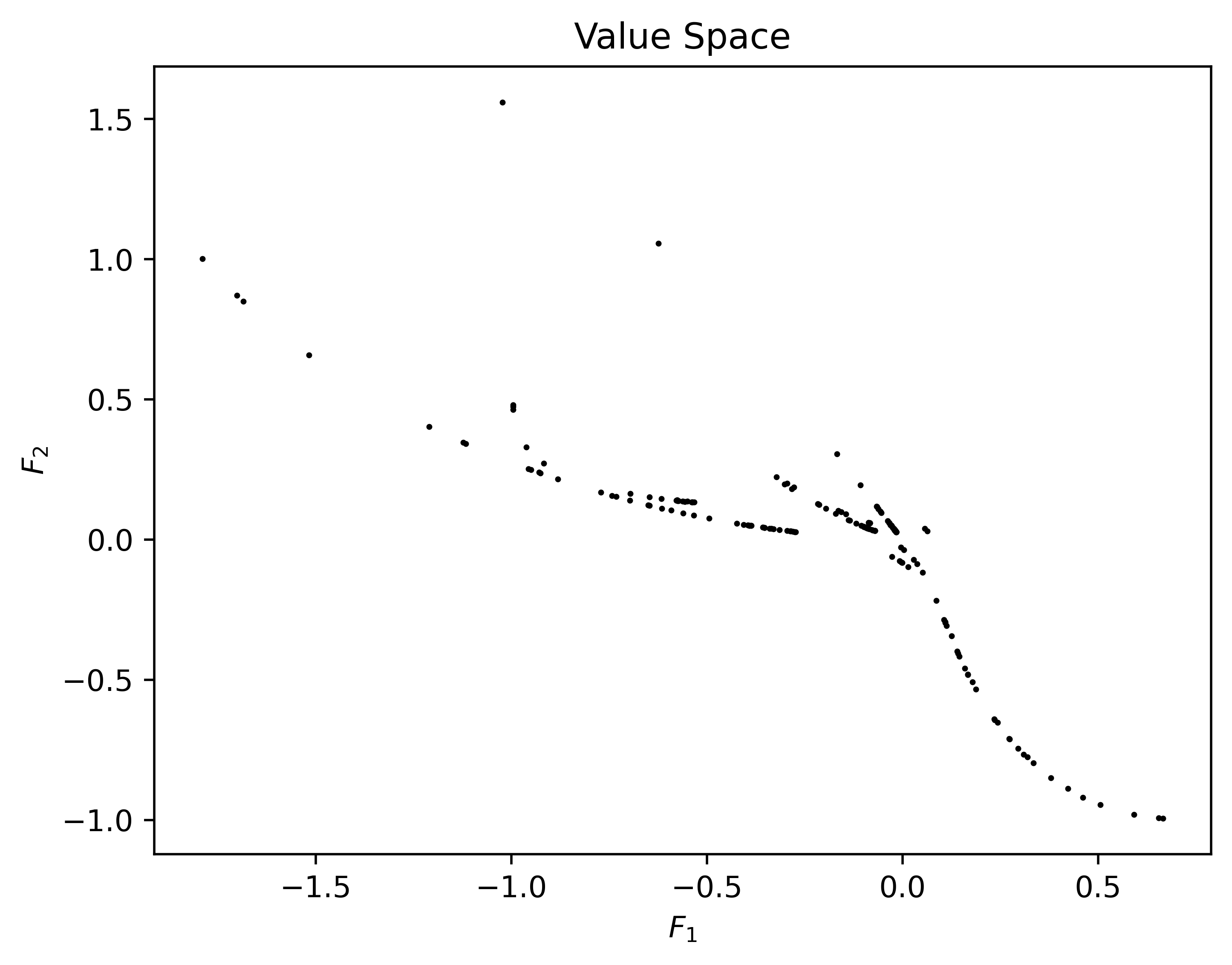} \\
			\includegraphics[scale=0.18]{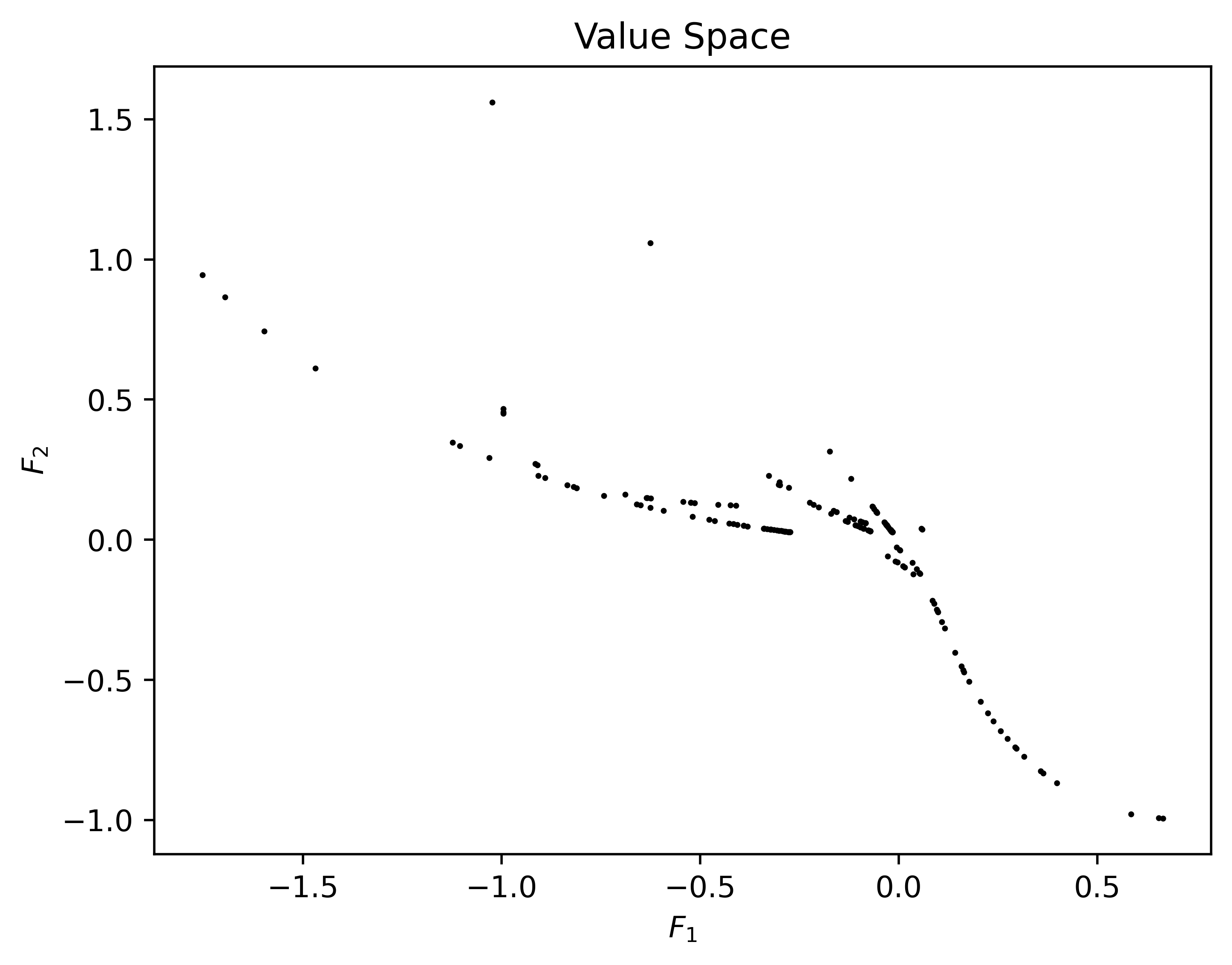}
		\end{minipage}
	}
	\caption{Numerical results in value space obtained by BBPGMO ({\bf top}) and PGMO for problems FDS, Deb, VU1, and Far1.}
	\label{f3}
\end{figure}

\begin{figure}[H]
	\centering
	\subfigure[DD1]
	{
		\begin{minipage}[H]{.22\linewidth}
			\centering
			\includegraphics[scale=0.18]{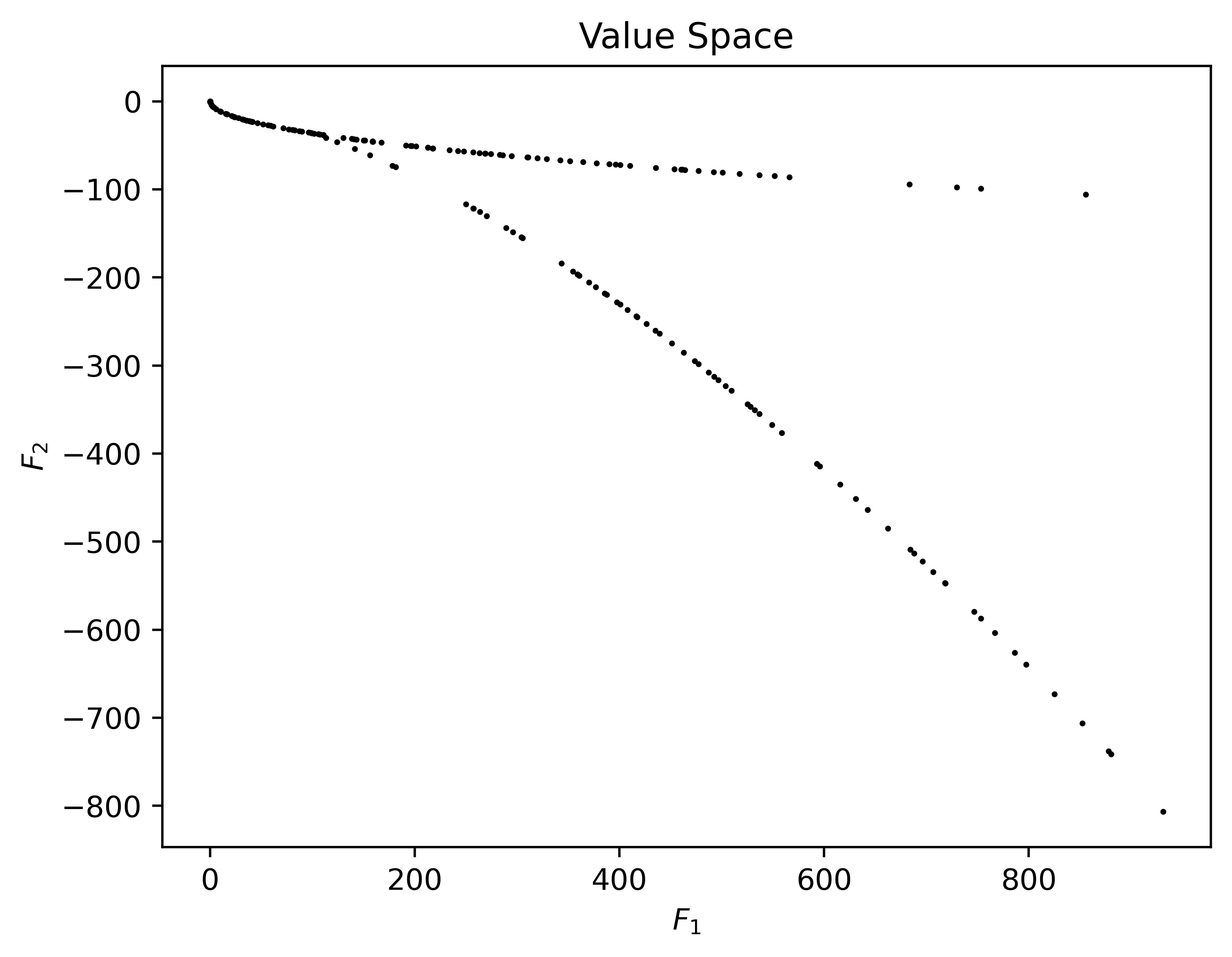} \\
			\includegraphics[scale=0.18]{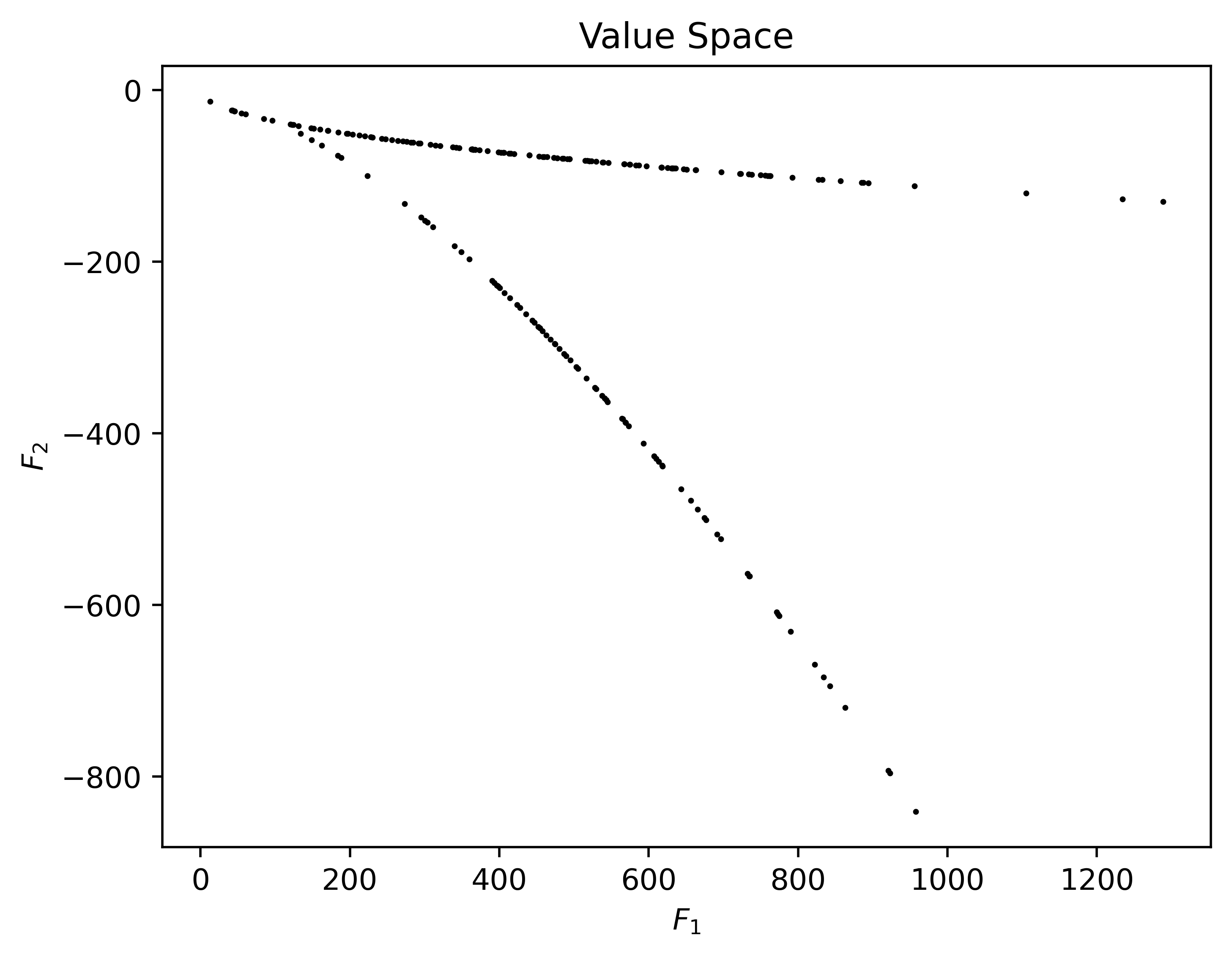}
		\end{minipage}
	}
	\subfigure[PNR]
	{
		\begin{minipage}[H]{.22\linewidth}
			\centering
			\includegraphics[scale=0.18]{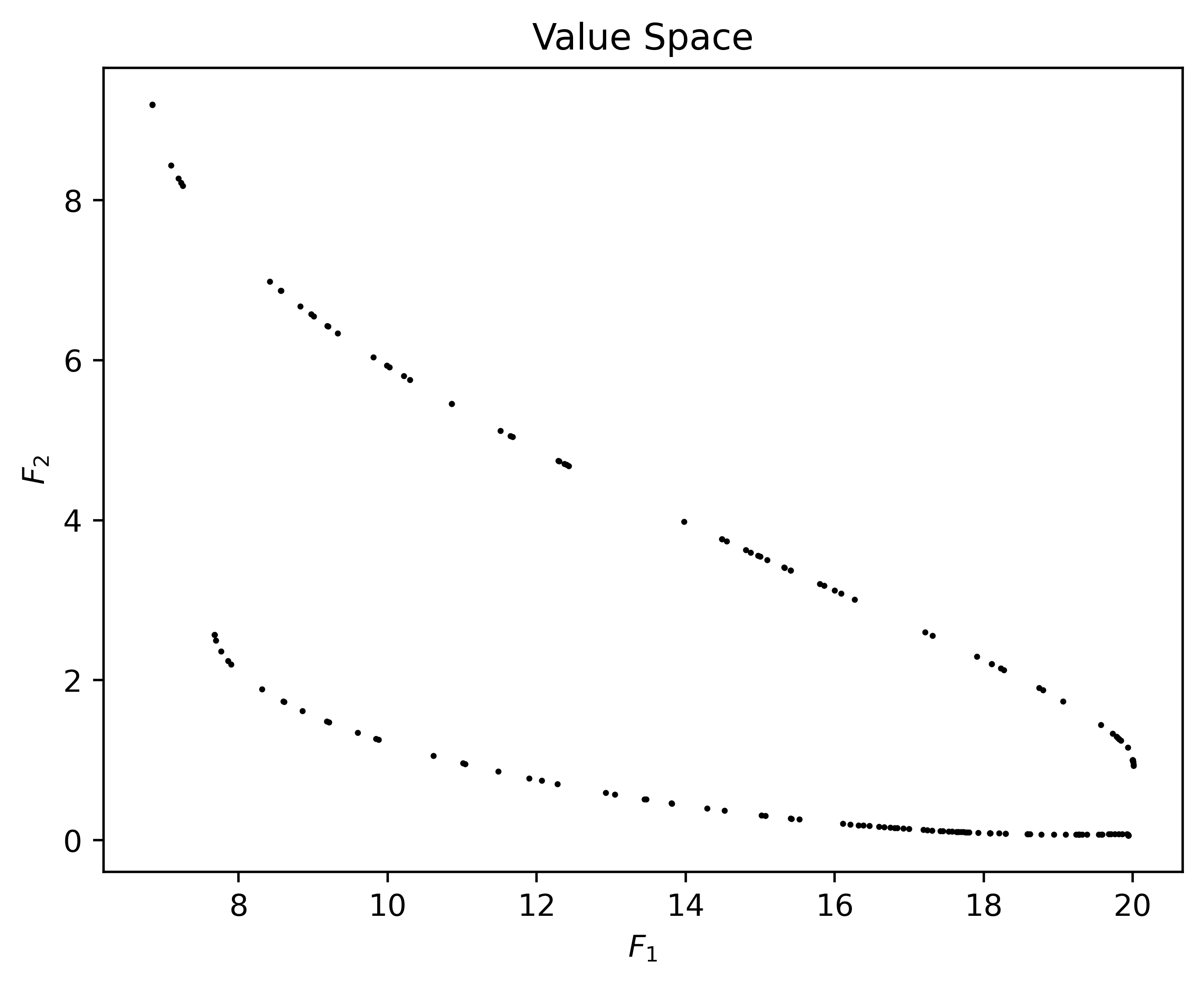} \\
			\includegraphics[scale=0.18]{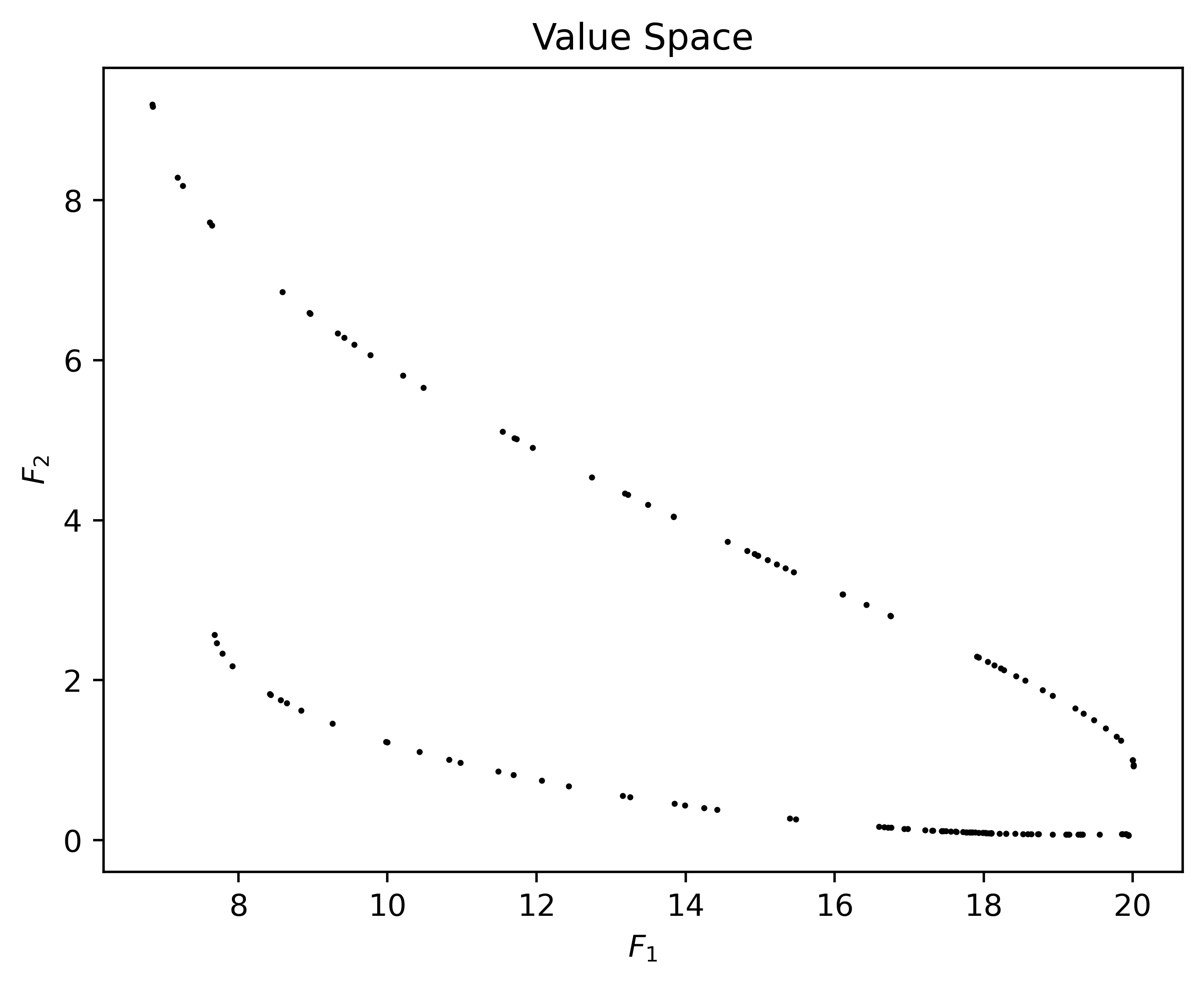}
		\end{minipage}
	}
	\subfigure[Hil1]
	{
		\begin{minipage}[H]{.24\linewidth}
			\centering
			\includegraphics[scale=0.18]{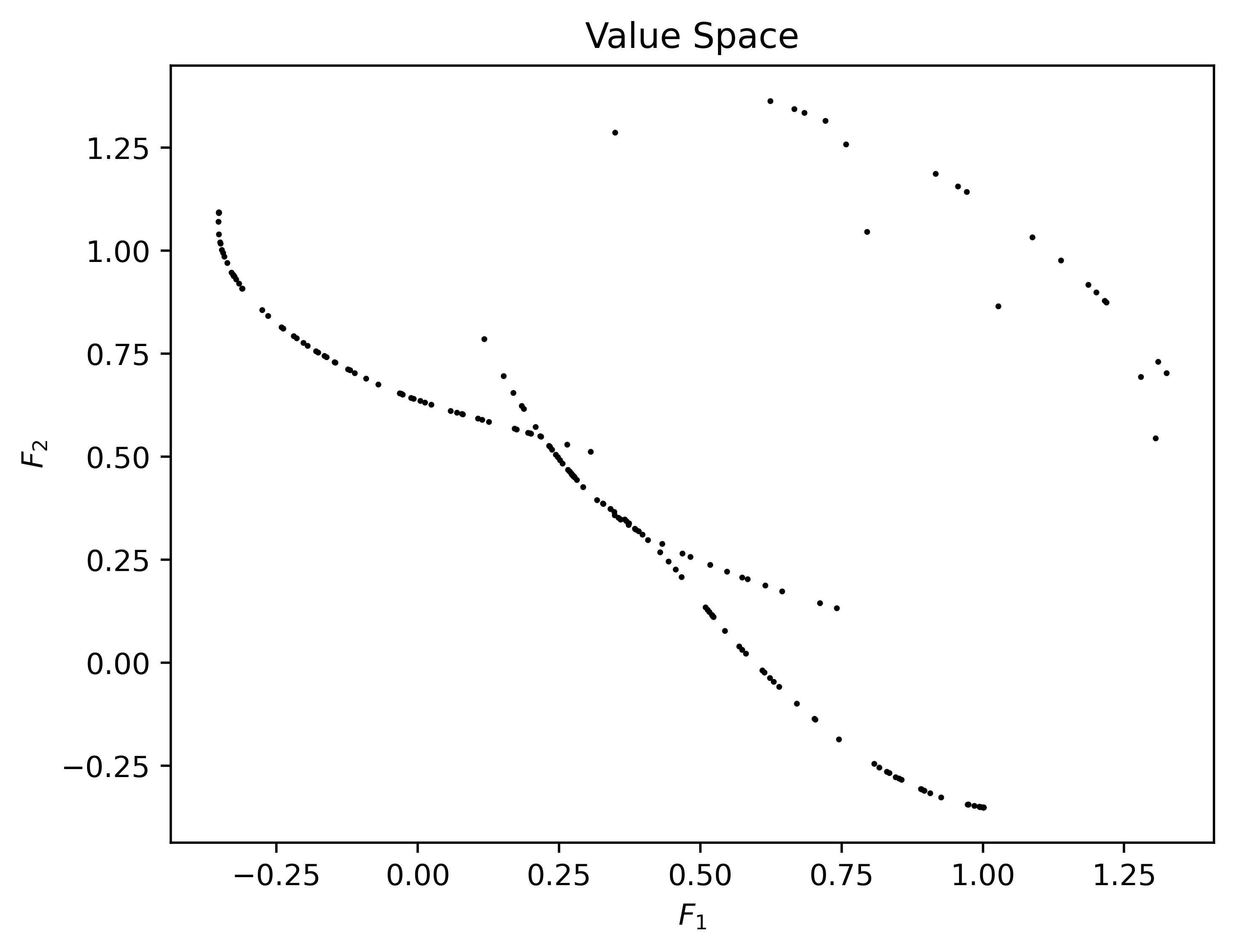} \\
			\includegraphics[scale=0.18]{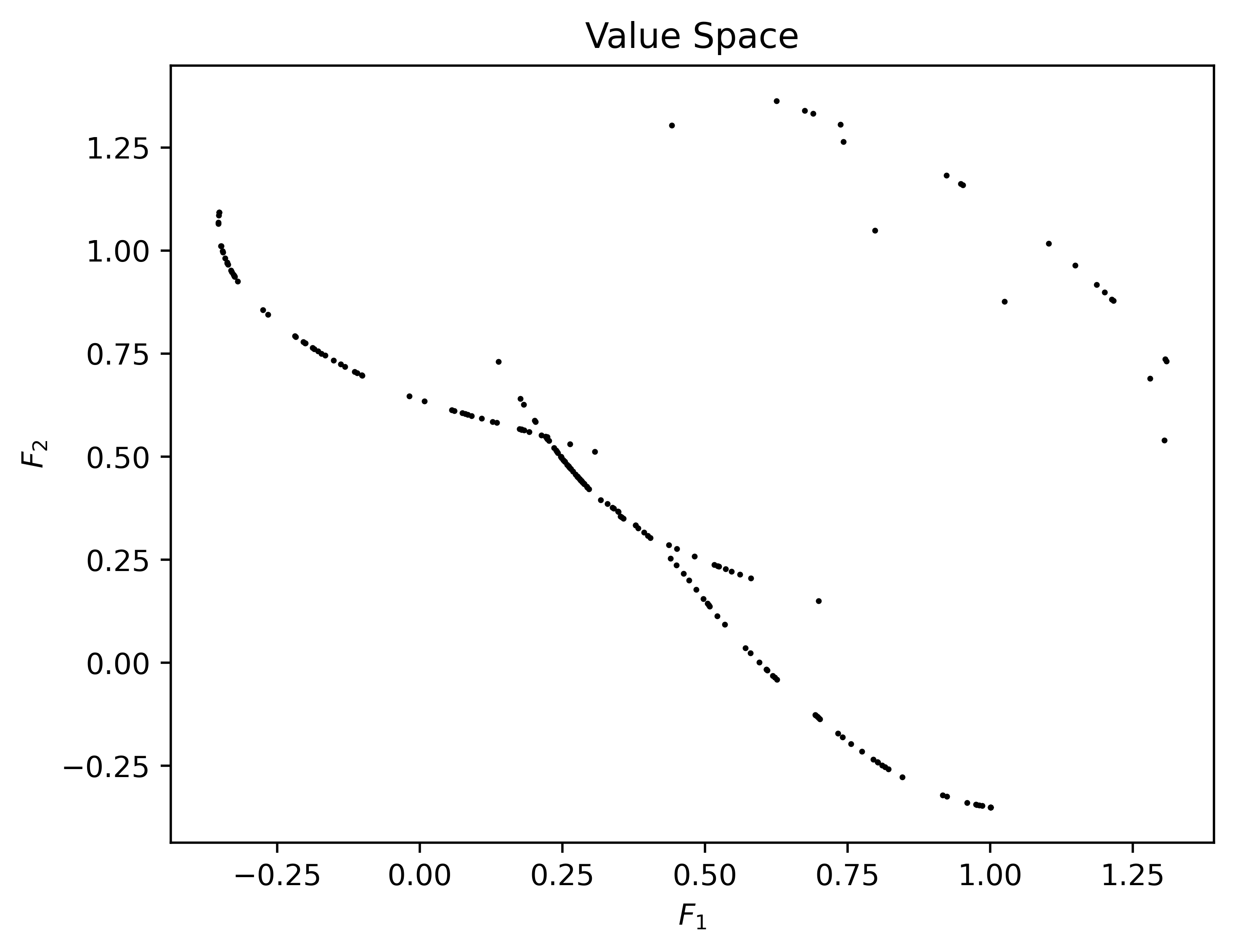}
		\end{minipage}
	}
	\subfigure[BK1]
	{
		\begin{minipage}[H]{.24\linewidth}
			\centering
			\includegraphics[scale=0.18]{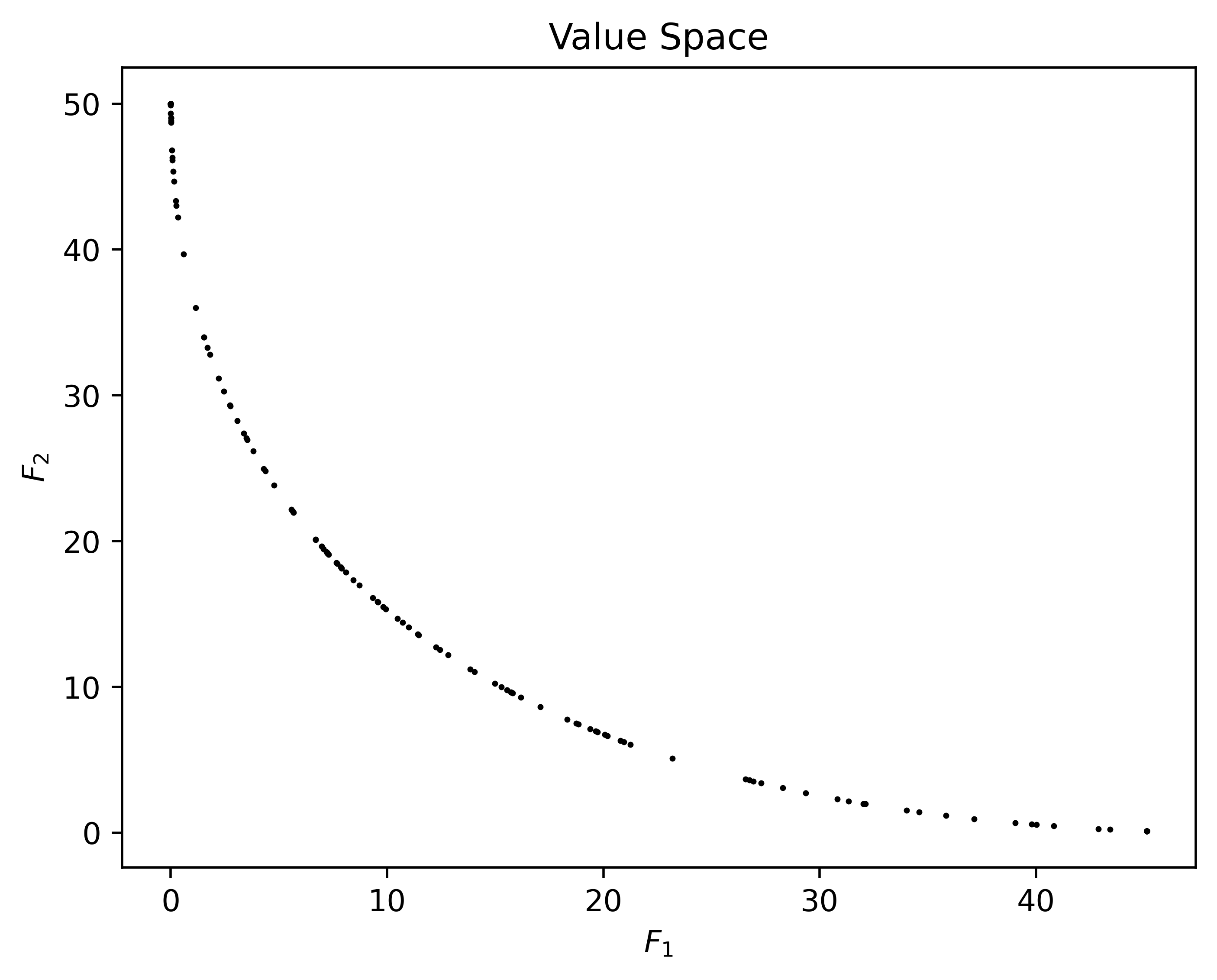} \\
			\includegraphics[scale=0.18]{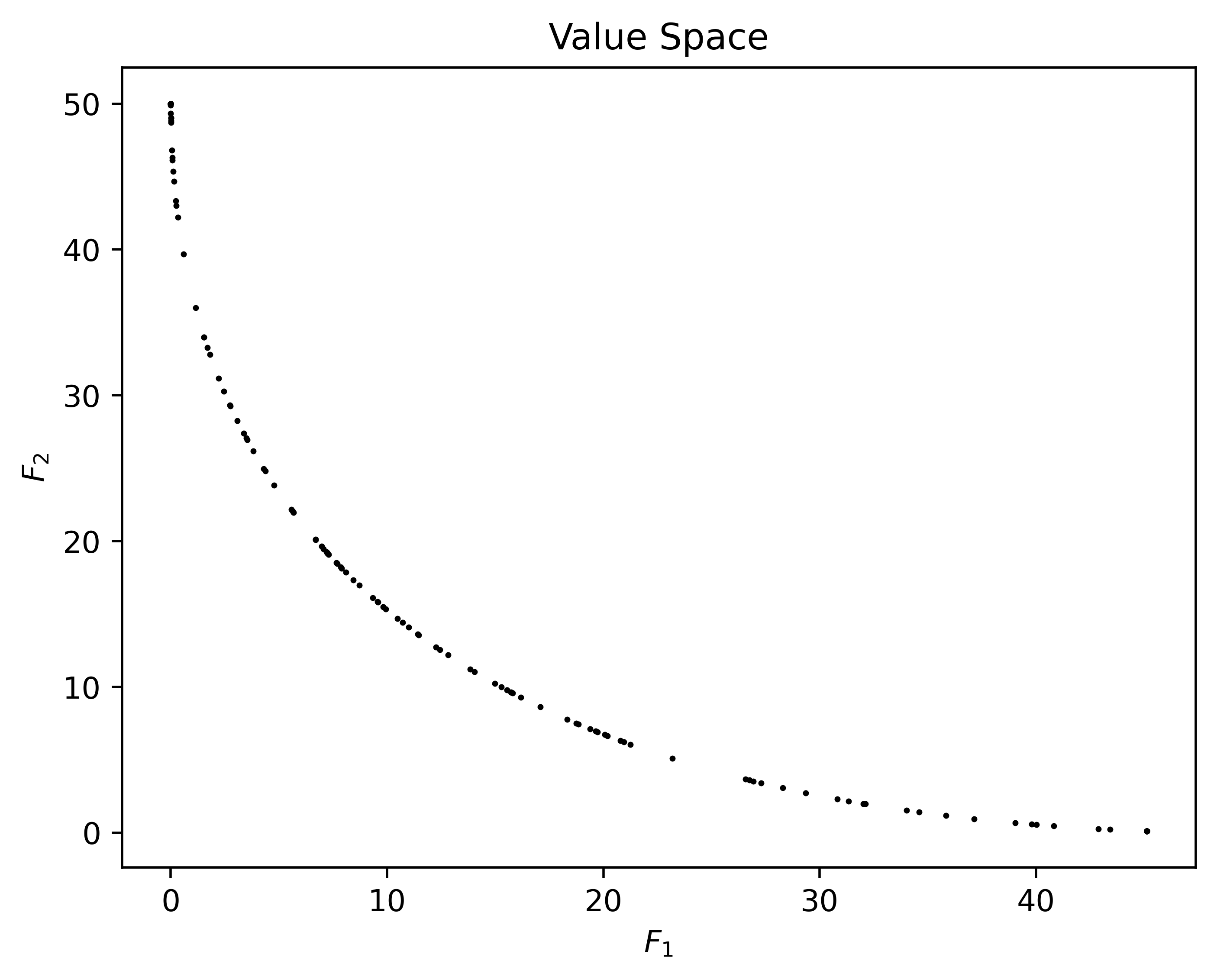}
		\end{minipage}
	}
	\caption{Numerical results in value space obtained by BBPGMO ({\bf top}) and PGMO for problems DD1, PNR, Hil1, and BK1.}
	\label{f4}
\end{figure}
\begin{table}[h]
	\centering
	\caption{Number of average iterations (iter), number of average function evaluations (feval), average CPU time (time($ms$)) and average stepsize (stepsize) of BBPGMO and PGMO implemented on different test problems {\bf with} line search.}
	\label{tab3}
	\resizebox{.95\columnwidth}{!}{
		\begin{tabular}{llrrrrlrrrr}
			\hline
			{Problem} &  & \multicolumn{4}{l}{BBPGMO}                                &  & \multicolumn{4}{l}{PGMO}                        \\ \cline{3-6} \cline{8-11} 
			&
			&
			\multicolumn{1}{l}{ite} &
			\multicolumn{1}{l}{feval} &
			\multicolumn{1}{l}{time ($ms$)} &
			\multicolumn{1}{l}{stepsize} &
			&
			\multicolumn{1}{l}{ite} &
			\multicolumn{1}{l}{feval} &
			\multicolumn{1}{l}{time ($ms$)} &
			\multicolumn{1}{l}{stepsize} \\ \hline
			BK1                      &  & \textbf{1.00}       & \textbf{1.00} & \textbf{1.41}   & 1.00 &  & {3.16} & 4.15   & 7.11         & 0.84 \\
			DD1                      &  & \textbf{4.54}  & \textbf{4.91} & \textbf{34.30} & 0.98 &  & 41.20          & 70.54  & 153.36       & 0.73 \\
			Deb                      &  & \textbf{6.96}  & \textbf{10.93} & \textbf{33.20}  & 0.68 &  & 27.97          & 255.00  & 94.92        & 0.12 \\
			Far1                     &  & {6.77} & \textbf{7.87} & \textbf{11.25}   & 0.94 &  & \textbf{6.07}         & 21.14   & 19.30         & 0.33 \\
			FDS                      &  & \textbf{3.44} & \textbf{3.81} & \textbf{24.22}  & 0.93 &  & 181.48         & 782.81  & 1206.17       & 0.25 \\
			FF1                      &  & \textbf{2.24}  & \textbf{2.40}  & 2.58             & 0.97 &  & 3.43           & 3.58    & \textbf{2.34} & 0.99 \\
			Hil1                     &  & \textbf{8.41} & \textbf{9.21} & \textbf{14.22}   & 0.65 &  & 10.68          & 20.59   & 28.44         & 0.33 \\
			Imbalance1               &  & \textbf{2.44}  & \textbf{3.11}  & \textbf{32.58}  & 0.92 &  & 2.92           & 5.64   & {55.55}        & 0.70 \\
			Imbalance2               &  & \textbf{1.00}  & \textbf{1.00} & \textbf{1.88}   & 1.00 &  & 83.92         & 593.11 & 1465.16       & 0.03 \\
			JOS1a                    &  & \textbf{1.00}  & \textbf{1.00}  & \textbf{1.80}   & 1.00 &  & 151.08        & 198.97  & 25.23        & 0.97 \\
			JOS1b                    &  & \textbf{1.00}  & \textbf{1.00}  & \textbf{2.81}   & 1.00 &  & 265.82         & 289.11  & 43.59        & 0.99 \\
			JOS1c                    &  & \textbf{1.00}  & \textbf{1.00}  & \textbf{1.72}  & 1.00 &  & 385.68         & 387.23  & 33.91        & 1.00 \\
			JOS1d                    &  & \textbf{1.00}  & \textbf{1.00}  & \textbf{1.56}   & 1.00 &  & 406.49         & 409.25  & 41.72        & 1.00 \\
			LE1                      &  & \textbf{5.46}  & \textbf{6.27} & \textbf{6.17}   & 0.71 &  & 12.16          & 16.72   & 9.38         & 0.70 \\
			PNR                      &  & \textbf{3.31}  & \textbf{3.72} & \textbf{8.05}   & 0.95 &  & 10.07          & 39.91   & 34.61        & 0.18 \\
			VU1                      &  & \textbf{2.08}  & \textbf{2.15}  & 2.66    & 0.98 &  & 12.90          & 12.97   & \textbf{2.03}          & 0.99 \\
			WIT1                     &  & \textbf{2.95}  & \textbf{3.26}  & \textbf{7.27}   & 0.96 &  & 27.64          & 145.06  & 115.55        & 0.12 \\
			WIT2                     &  & \textbf{3.16}  & \textbf{3.37} & \textbf{9.06}   & 0.97 &  & 48.10          & 286.32  & 178.05        & 0.06 \\
			WIT3                     &  & \textbf{3.94}  & \textbf{4.26} & \textbf{16.64}   & 0.97 &  & 18.61          & 79.91  & 77.42        & 0.15 \\
			WIT4                     &  & \textbf{4.01}  & \textbf{4.17} & \textbf{16.09}   & 0.98 &  & 6.51          & 19.26   & 25.08        & 0.30 \\
			WIT5                     &  & \textbf{3.21}  & \textbf{3.46} & \textbf{12.19}   & 0.98 &  & 5.13          & 12.38   & 22.11         & 0.41 \\
			WIT6                     &  & \textbf{1.00} & \textbf{1.00} & \textbf{3.67}   & 1.00 &  & 1.87          & 2.87   & 6.72        & 0.73 \\ \hline
		\end{tabular}
	}
\end{table}
\par The obtained Pareto sets and Pareto fronts for some test problems are depicted in Figures \ref{f1}-\ref{f4}. Notably, Figure \ref{f2} illustrates that the solutions for problems FDS, Deb, VU1, and Far1 exhibit sparsity, validating the sparsity of the objective functions. This sparsity property holds significant importance in various applications such as machine learning, image restoration, and signal processing. 
\par Table \ref{tab3} provides the average number of iterations (iter), average number of function evaluations (feval), average CPU time (time ($ms$)), and average stepsize (stepsize) for each tested algorithm across the different problems. The numerical results confirm that BBPGMO outperforms PGMO in terms of average iterations, average function evaluations, and average CPU time. The average stepsize of BBPGMO is robust and falls within the range of $[0.65,1]$ for different problems, whereas the stepsize for PGMO exhibits significant variation. Furthermore, PGMO exhibits poor performance on problems DD1, Deb, FDS, imbalance2, JOS1a-d, and WIT1-2, which feature imbalanced and high-dimensional objective functions. Based on the performance of BBPGMO on these problems, we conclude that it is well-suited for addressing such challenges.
\subsection{Application to Markowitz Portfolio Selection}
In this subsection, we consider the Markowitz portfolio selection problem \cite{M1952}. Suppose there are $n$ securities, the expected returns $\mu\in\mathbb{R}^{n}$ and variance of returns $\Sigma\in\mathbb{R}^{n\times n}$ are known. The $E$-$V$ rule of Markowitz portfolio selection suggests investor selects one of efficient portfolios, which is a Pareto solution of the following bi-objective optimization problem:
\begin{align*}
	\min\limits_{x}&~(-\mu^{T}x,x^{T}\Sigma x)\\
	\mathrm{ s.t.} &~~~x\in\Delta_{n}.
\end{align*}
As a specific example, the expected returns $\mu\in\mathbb{R}^{n}$ and variance of returns $\Sigma\in\mathbb{R}^{n\times n}$ are estimated from real data on eight types of securities. The data can be found at\\ \href{https://vanderbei.princeton.edu/ampl/nlmodels/markowitz/}{https://vanderbei.princeton.edu/ampl/nlmodels/markowitz/} and we use the data between the years $1983$ and $1994$ to estimate $\mu$ and $\Sigma$ which are given as follow:
$$\mu=(1.0672, 1.1228, 1.1483, 1.1440, 1.1329, 1.1029, 1.1975, 0.9952)^{T},$$

\begin{small}
	\begin{equation*}
		\Sigma = 	\begin{pmatrix} 
			0.0005 & 0.0004 & 0.0007 & 0.0005 & -0.0007 & 0.0006 & 0.0001 & -0.0015\\
			0.0004 & 0.0216 & 0.0110 & 0.0116 & 0.0138 & 0.0092 & 0.0208 & 0.0027\\
			0.0007 & 0.0110 & 0.0149 & 0.0162 & 0.0211 & 0.0056 & 0.0158 & -0.0007\\
			0.0005 & 0.0116 & 0.0162 & 0.0181 & 0.0252 & 0.0059 & 0.0164 & -0.0015\\
			-0.0007 & 0.0138 & 0.0211 & 0.0252 & 0.0430 & 0.0070 & 0.0159 & -0.0019\\
			0.0006 & 0.0092 & 0.0056 & 0.0059 & 0.0070 & 0.0045 & 0.0073 & -0.0006\\
			0.0001 & 0.0208 & 0.0158 & 0.0164 & 0.0159 & 0.0073 & 0.0672 & 0.0190\\
			-0.0015 & 0.0027 & -0.0007 & -0.0015 & -0.0019 & -0.0006 & 0.0190 & 0.0189	
		\end{pmatrix}.
	\end{equation*}
\end{small}

Figure \ref{f5} illustrates the obtained efficient $E,V$ combinations with 100 random start points in $\Delta_{8}$. The average number of iterations (iter), average number of function evaluations (feval), and average CPU time (time ($ms$)) are recorded in Table \ref{tab4}. The numerical results confirm that BBPGMO outperforms PGMO in Markowitz portfolio selection problem. 

\begin{figure}[H]
	\centering
	\subfigure[BBPGMO]
	{
		\begin{minipage}[H]{.4\linewidth}
			\centering
			\includegraphics[scale=0.35]{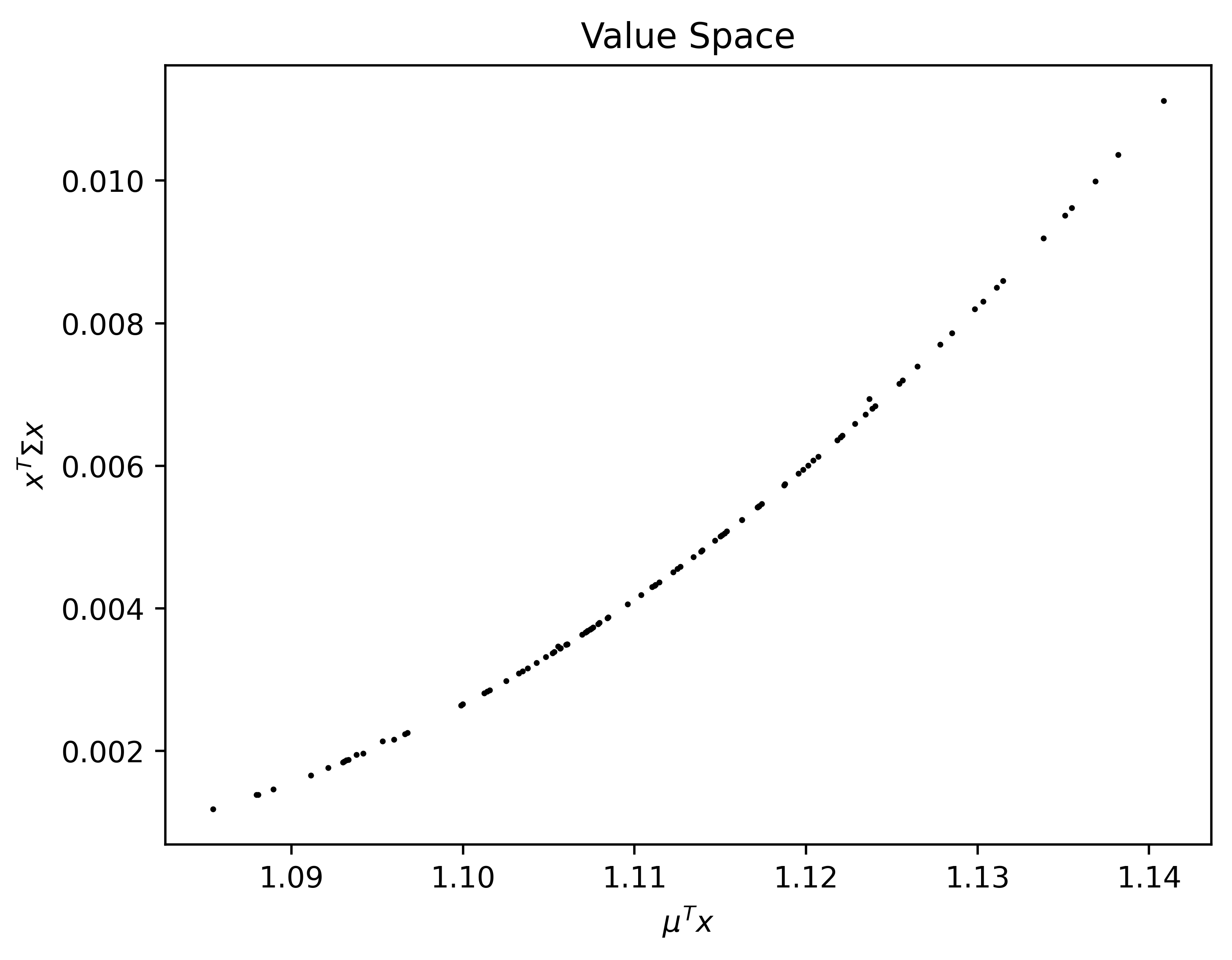}
		\end{minipage}
	}
	\subfigure[PGMO]
	{
		\begin{minipage}[H]{.4\linewidth}
			\centering
			\includegraphics[scale=0.35]{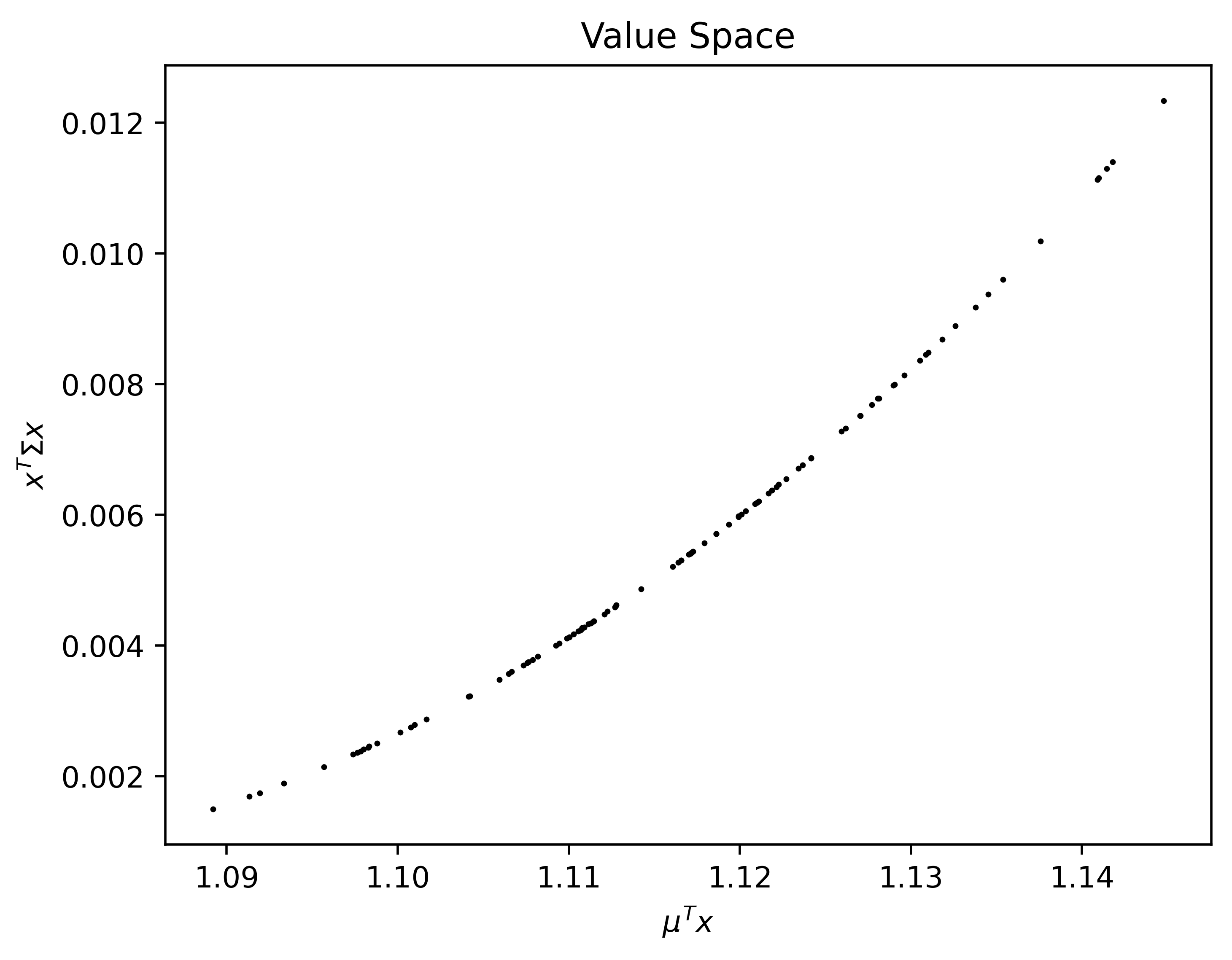} 		
		\end{minipage}
	}
	
	\caption{Numerical results obtained by BBPGMO and PGMO for the Markowitz portfolio selection problem.}
	\label{f5}
\end{figure}

\begin{table}[h]
	\centering
	\caption{Number of average iterations (iter), number of average function evaluations (feval), and average CPU time (time ($ms$)) of BBPGMO and PGMO implemented on the Markowitz portfolio selection problem.}
	\label{tab4}
	\begin{tabular}{crlrlrl}
		\hline
		\multicolumn{1}{l}{} & iter          &  & feval         &  & time ($ms$)            &  \\ \hline
		BBPGMO               & \textbf{7.19} &  & \textbf{9.36} &  & \textbf{349.53} &  \\ \hline
		PGMO                 & 269.23        &  & 269.23        &  & 9454.38         &  \\ \hline
	\end{tabular}
\end{table}
\section{Conclusions}
In this paper, we proposed two types of proximal gradient methods for MCOPs and analyzed their convergence rates. Notably, in the case of strong convexity, the proposed method converges linearly at a rate of $\sqrt{1-\min\limits_{i\in[m]}\left\{\frac{\mu_{i}}{L_{i}}\right\}}$, whereas the linear convergence rate of PGMO is $\sqrt{1-\frac{\mu_{\min}}{L_{\max}}}$. The improved linear convergence confirms the BBPGMO's superiority, and validates that the Barzilai-Borwein method can alleviate interference and imbalances among objectives. Interestingly, from the perspective of complexity, it also reveals that optimizing multiple objective functions simultaneously may be easier than optimizing the most difficult one (as long as $\lambda^{k}_{i}\neq1$ for the worst $\frac{\mu_{i}}{L_{i}},~i\in[m]$). Moreover, we obtained the linear convergence of BBPGMO for MOPs with some linear objectives. To the best of our knowledge, this is the first result demonstrating linear convergence of gradient descent methods for MOPs with some linear objectives. By setting $g_{i}(x)=0$, $i\in[m]$ or $g_{i}(x)=\mathbb{I}_{\mathcal{X}}(x)$, $i\in[m]$, all the theoretical results of BBPGMO are satisfied for corresponding gradient descent method and projected gradient method, respectively.
\par From a methodological perspective, it may be worth considering the following points:
\begin{itemize}
	\item From theoretical point of view, it is worth noting that BBPGMO can exhibit slow convergence when applied to ill-conditioned MOPs. Fortunately, the utilization of Barzilai-Borwein's rule within multiobjective gradient descent methods does not impede the implementation of other acceleration strategies. Given the enhanced theoretical attributes associated with the Barzilai-Borwein methods, there exists an avenue of exploration into the applicability of conjugate gradient methods \cite{LP2018}, the Nesterov's accelerated methods \cite{SP2022,SP2023,TFY2022}, and preconditioning methods \cite{GB2015,HL2015,W2015} based on the BBDMO or BBPGMO, respectively.
	\item Recently, researchers have increasingly recognized multi-task learning as  multiobjective optimization and have developed effective algorithms based on SDMO to train models (see, e.g., \cite{LZ2019,MR2020,SK2018}). However, loss functions in machine learning often include an $\ell_{1}$-regularized term to mitigate overfitting. On the other hand, as emphasized by Chen et al. \cite{CB2018}: \textit{``Task
		imbalances impede proper training because they manifest as imbalances between backpropagated gradients."} Fortunately, the BBPGMO is a first-order method capable of effectively handling imbalanced and high-dimensional multiobjective composite optimization problems. Theore, applying the BBPGMO to multi-task learning is a promising direction for future research.
\end{itemize}

\bibliographystyle{abbrv}
\bibliography{erences}

\begin{acknowledgements}
This work was funded by the Major Program of the National Natural Science Foundation of China [grant numbers 11991020, 11991024]; the National Natural Science Foundation of China [grant numbers 11971084, 12171060]; NSFC-RGC (Hong Kong) Joint Research Program [grant number 12261160365]; the Team Project of Innovation Leading Talent in Chongqing [grant number CQYC20210309536]; the Natural Science Foundation of Chongqing [grant number ncamc2022-msxm01]; and Foundation of Chongqing Normal University [grant numbers 22XLB005, 22XLB006].
\end{acknowledgements}

\end{document}